\documentclass[a4paper,11pt,oneside]{amsart}
\usepackage{geometry}

\usepackage[english]{babel}
\usepackage{amsmath}
\usepackage{amsthm}
\usepackage{amssymb}
\usepackage{mathtools}
\usepackage{xcolor}
\usepackage{enumitem}
\usepackage{xpatch}
\usepackage{multicol}  
\usepackage{needspace}
\usepackage{setspace}
\setdisplayskipstretch{}

\makeatletter
 \def\@textbottom{\vskip \z@ \@plus 14pt}
 \let\@texttop\relax
\makeatother

\definecolor{linkblue}{RGB}{1,1,190}
\definecolor{citered}{RGB}{190,1,1}

\usepackage[linkcolor=linkblue
           ,urlcolor=linkblue
           ,citecolor=citered
           ,ocgcolorlinks        
           ,bookmarksopen=True,  
           ]{hyperref}
\usepackage[ocgcolorlinks]{ocgx2}
\makeatletter
  \AtBeginDocument{
    \hypersetup{
      pdftitle  = {\@title},
    }
  }
\makeatother
  
\usepackage[capitalize]{cleveref}

\usepackage{microtype}

\newtheorem{teor}{Theorem}[section]
\newtheorem{defi}[teor]{Definition}
\newtheorem{lemma}[teor]{Lemma}
\newtheorem{prop}[teor]{Proposition}
\newtheorem{cor}[teor]{Corollary}

\theoremstyle{remark}
\newtheorem{remark}[teor]{Remark}

\newtheorem{example}[teor]{Example}
\newtheorem{examples}[teor]{Examples}

\crefname{teor}{Theorem}{Theorems}
\crefname{defi}{Definition}{Definitions}
\crefname{lemma}{Lemma}{Lemmas}
\crefname{prop}{Proposition}{Propositions}
\crefname{cor}{Corollary}{Corollaries}
\crefname{ques}{Question}{Questions}
\crefname{setting}{Setting}{Settings}
\crefname{Obs}{Observation}{Observations}
\crefname{Com}{Comment}{Comments}
\crefname{remark}{Remark}{Remarks}
\crefname{remarks}{Remarks}{Remarks}
\crefname{example}{Example}{Examples}
\crefname{examples}{Examples}{Examples}

\DeclareMathOperator{\add}{add}
\DeclareMathOperator{\id}{id}
\DeclareMathOperator{\supp}{supp}
\DeclareMathOperator{\size}{size}
\DeclareMathOperator{\Hom}{Hom}
\DeclareMathOperator{\im}{im}
\DeclareMathOperator{\End}{End}

\DeclarePairedDelimiter{\card}{\lvert}{\rvert}

\setlist[enumerate,1]{label=\textup{(\arabic*)}, ref=\textup{(}\arabic*\textup{)},
  itemsep=0.5em plus 0.15em minus 0.05em,
  topsep=0.5em plus 0.15em minus 0.05em,
  leftmargin=0.75cm}
\setlist[enumerate,2]{label=\textup{(\roman*)}, ref=\textup{(}\roman*\textup{)}
  itemsep=0.5em plus 0.15em minus 0.05em,
  topsep=0.5em plus 0.15em minus 0.05em}
\setlist[itemize, 1]{itemsep=0.5em plus 0.15em minus 0.05em,
  topsep=0.5em plus 0.15em minus 0.05em, leftmargin=0.75cm}

\newlist{equivenumerate}{enumerate}{1}
\setlist[equivenumerate,1]{%
  label=\textup{(\alph*)},
  ref=\textup{(}\alph*\textup{)},
  itemsep=0.5em plus 0.15em minus 0.05em,
  topsep=0.5em plus 0.15em minus 0.05em,
  leftmargin=0.75cm
}

\newlist{proofenumerate}{enumerate}{1}
\setlist[proofenumerate,1]{%
  itemsep=0.5em plus 0.15em minus 0.05em,
  topsep=0.5em plus 0.15em minus 0.05em,
  wide, labelindent=0pt
}

\xpatchcmd{\paragraph}{\normalfont}{{\normalfont\bfseries}}{}{}

\newcommand{\defit}[1]{\textsf{#1}}

\newcommand{\bN}{\mathbb N}
\newcommand{\bQ}{\mathbb Q}
\newcommand{\bR}{\mathbb R}

\newcommand{\cA}{\mathcal A}
\newcommand{\cB}{\mathcal B}
\newcommand{\cC}{\mathcal C}

\DeclareMathOperator{\Tr}{Tr}

\DeclareMathOperator{\Pic}{Pic}
\DeclareMathOperator{\Vmon}{V}

\DeclareMathOperator{\rep}{right}
\DeclareMathOperator{\lep}{left}

\title[Infinite direct sums via monoids]{A monoid-theoretical approach to infinite direct-sum decompositions of modules}

\author{Zahra Nazemian}
\address{University of Graz\\
         Department of Mathematics and Scientific Computing\\
         NAWI Graz\\
          Heinrichstra\ss e 36\\
          8010 Graz, Austria}
\email{zahra.nazemian@uni-graz.at}

\author{Daniel Smertnig}
\address{Faculty of Mathematics and Physics (FMF)\\
  University of Ljubljana
  and Institute of Mathematics, Physics and Mechanics (IMFM)\\
  Jadranska ulica 21\\
  1000 Ljubljana, Slovenia}
\email{daniel.smertnig@fmf.uni-lj.si}



\begin{document}

\begin{abstract}
  Let $\cC$ be a class of modules over a ring $R$, closed under direct sums over index sets of cardinality $\kappa$, direct summands, and isomorphisms, and such that the isomorphism classes form a set.
  The associated monoid of modules $\Vmon(\cC)$ encodes the behavior of finite direct-sum decompositions of modules in $\cC$.
  We endow $\Vmon(\cC)$ with an additional operation reflecting $\kappa$-indexed direct sums, and study the resulting \emph{$\kappa$-monoid} $\Vmon^{\kappa}(\cC)$.

  The crucial \emph{braiding-property} and an equivalent universal property in the category of $\kappa$-monoids allow us to show: if every module in $\cC$ is a direct sum of at most $\kappa$ many modules generated by strictly fewer than $\lambda$ many elements, with $\lambda \le \kappa$ an infinite regular cardinal, then all relations on  $\Vmon^{\kappa}(\cC)$ are induced by relations between direct sums indexed by sets of cardinality strictly less than $\lambda$.
  A theorem of Kaplansky states that every projective module is a direct sum of countably generated modules.
  Our results augment this, by showing that also all relations between infinite direct sums of projective modules are induced from relations between countable direct sums of countably generated projective modules.
  
  If every projective module over a ring $R$ is a direct sum of finitely generated projective modules (e.g., if $R$ is hereditary), then all relations between direct sums are even induced from relations between finite direct sums of finitely generated projective modules.
  Then the monoid of finitely generated projective modules $\Vmon(R)$ completely determines the $\kappa$-monoid $\Vmon^{\kappa}(R)$.
  Together with the realization result of Bergman and Dicks, this characterizes the $\kappa$-monoids appearing as $\Vmon^{\kappa}(R)$ for a hereditary ring as the universal $\kappa$-extensions of reduced commutative monoids with order-unit.
  In general, the $\aleph_0$-monoid $\Vmon^{\aleph_0}(R)$ fully determines $\Vmon^{\kappa}(R)$.
  Herbera and Příhoda's characterization of monoids of countably generated projective modules $\Vmon^*(R)$ over semilocal noetherian rings, yields a characterization of $\Vmon^{\kappa}(R)$ for these rings.
  Finally, we characterize two-generated $\aleph_0$-monoids that can be realized as $\Vmon^{\aleph_0}(R)$ for hereditary rings $R$.
\end{abstract}

\date{}

\maketitle

\section{Introduction}

Let $\cC$ be a class of (right) modules over some (unital, associative) ring $R$, and assume that $\cC$ is closed under finite direct sums and under isomorphisms (typically, one also assumes that $\cC$ is closed under direct summands).
To avoid set-theoretic issues, we require that the isomorphism classes of $\cC$ form a set.
Then the set of isomorphism classes together with the operation induced by the direct sum form a monoid, denoted by $\Vmon(\cC)$.
Understanding direct-sum decompositions in $\cC$ reduces to the study of the arithmetic of the monoid $\Vmon(\cC)$.
For instance, the class $\cC$ has the Krull--Remak--Schmidt--Azumaya (KRSA) property, meaning that (finite) direct-sum decompositions into indecomposables exist and are unique, if and only if the monoid $\Vmon(\cC)$ is a free abelian monoid.

Of particular interest is the case where $\cC$ is the class of all finitely generated projective $R$-modules, in which case one writes $\Vmon(R)$ for the monoid (the Grothendieck group of $\Vmon(R)$ is $K_0(R)$).
If $M$ is a module and $\add(M)$ is the class of all direct summands of $M^n$ for $n \ge 0$, then $\Vmon(\add(M)) \cong \Vmon(\End(M))$, so that in principle the study of finite direct-sum decompositions can always be reduced to the study of some monoid of finitely generated projective modules.

The monoid-theoretical approach to (finite) direct-sum decompositions was pioneered by Facchini, Herbera, and Wiegand in the last two decades (see the surveys \cite{Facchini06,BaethWiegand13} and the monographs \cite{LeuschkeWiegand12,Facchini19}), and has proven fruitful in several interesting cases.
Two of the main directions are as follows.
First, if $R$ is semilocal, then $\Vmon(R)$ is a Diophantine monoid (a submonoid of $\bN_0^n$ defined by homogeneous linear equations), as it can be embedded into $\Vmon(R/J(R))$ via a divisor homomorphism \cite{FacchiniHerbera00,FacchiniHerbera00a}.
Diophantine monoids are finitely generated reduced Krull monoids, and conversely, every finitely generated reduced Krull monoid is isomorphic to a Diophantine monoid \cite{ChapmanKrauseOeljeklaus02}.
More generally, if $\cC$ is closed under direct summands and every module in $\cC$ has a semilocal endomorphism ring, then $\Vmon(\cC)$ is a reduced Krull monoid \cite{Facchini02}.
This reduces the study of the arithmetic of $\Vmon(\cC)$ to the (non-trivial) determination of its divisor class group and the distribution of prime divisors (the recent survey \cite{GeroldingerZhong20} is a starting point into the extensive literature on factorization theory in Krull monoids).
Second, let $(R,\mathfrak m)$ be a commutative noetherian local ring and let $\widehat R$ be its $\mathfrak m$-adic completion.
If $\mathcal M(R)$ is the class of all finitely generated $R$-modules, then $\Vmon(\mathcal M(R))$ embeds into the free abelian monoid $\Vmon(\mathcal M(\widehat R))$ via a divisor homomorphism \cite{Wiegand01}.
Again, the monoid $\Vmon(\mathcal M(R))$ is a reduced Krull monoid, and this permits one to study $V(\mathcal M(R))$ in terms of $\Vmon(\mathcal M(\widehat R))$.
In several cases of one- or two-dimensional domains, the divisor class groups and the distribution of prime divisors are sufficiently well-understood to allow one to study the arithmetic of the monoids \cite{Baeth09,BaethGeroldinger14,BaethLuckas11,BaethGeroldingerGrynkiewiczSmertnig15,Grynkiewicz22}.

In a converse direction, there are realization results showing that very large classes of commutative monoids appear as $\Vmon(R)$ of certain classes of rings, thereby showing that monoids of modules can be essentially arbitrarily complicated in general.
By a seminal theorem of Bergman \cite{Bergman} and Bergman and Dicks \cite{BergWar}, every reduced commutative monoid with order-unit is isomorphic to $\Vmon(R)$ for a hereditary $K$-algebra $R$ (with an arbitrary choice of ground field $K$).
Ara and Goodearl recently showed that, if one allows non-unital algebras, this realization is possible with $R$ a Leavitt path algebra of a separated graph \cite{AraGoodearl12}.
Herbera and Facchini showed that every reduced finitely generated Krull monoid appears as $\Vmon(R)$ for a semilocal ring \cite{FacchiniHerbera00a}, and Wiegand has a similar realization result for finitely generated modules over two-dimensional local UFDs \cite{Wiegand01}.
Facchini and Wiegand realized every reduced Krull monoid as $\Vmon(\cC)$ with $\cC$ a class of modules whose endomorphism rings are semilocal \cite{FacchiniWiegand04}.

For (von Neumann) regular rings $R$, the monoid $\Vmon(R)$ is always a refinement monoid.
However, Wehrung constructed a refinement monoid of cardinality $\aleph_2$ that cannot be realized as $\Vmon(R)$ of a regular ring \cite{Frid}.
In a recent breakthrough Ara, Bosa, and Pardo showed that every finitely generated refinement monoid appears as $\Vmon(R)$ of a regular ring \cite{PJuan}.
It is still open whether every countable refinement monoid appears as $\Vmon(R)$ of a, possibly non-unital, regular ring.
Realization results of this type are useful to construct negative results (that is, counterexamples) in (finite) direct-sum decompositions, as it is usually much easier to construct a monoid exhibiting certain degenerate behavior and then realize that monoid as $\Vmon(R)$, than it is to construct a corresponding ring directly.

For a long time, the focus remained on finitely generated projective modules, undoubtedly to some degree due to a result of Bass \cite{Bass63}, showing that non-finitely generated projective modules are free in many circumstances.
However, despite Bass's result, over many natural classes of rings, the non-finitely generated projective modules exhibit a rich and interesting behavior.
Powered by recent results yielding a better understanding of non-finitely generated projective modules, Herbera and Příhoda completely characterized the monoids of \emph{countably} generated projective modules $\Vmon^*(R)$ for noetherian semilocal rings \cite{PavelHerbera}.
(Without the noetherian hypothesis partial results are available \cite{infinitepullback}.)
Recently, Herbera, Příhoda, and Wiegand initiated the study of $\Vmon^*(M)$, the monoid of isomorphism classes of countably generated modules that are direct summands of some number of copies of $M$ \cite{Herbera14,HerberaPrihoda14,HerberaPrihodaWiegand23}.
For hereditary noetherian prime rings, the monograph by Levy and Robson \cite[Chapter 8]{LevyRobson11} completely describes the non-finitely generated projective modules.
As a related notion, Antoine, Ara, Bosa, Perera, and Vilalta recently introduced the Cuntz semigroup of a ring \cite{AntoineAraBosaPereraVilalta25,AntoineAraBosaPereraVilalta26}, inspired by the Cuntz semigroup of a $C^*$-algebra, in which a weaker equivalence relation than isomorphism is imposed on the countably generated projective modules.

A classical theorem of Kaplansky shows that every projective module is a direct sum of countably generated modules \cite{Kaplansky58}, so that the determination of $\Vmon^*(R)$ is sufficient to find all indecomposable projective modules.
However, the class of countably generated projective modules is also closed under \emph{countable} direct sums, a fact that is not reflected in the monoid structure of $\Vmon^*(R)$.
This suggests to endow $\Vmon^*(R)$ with an additional operation, to allow us to also study \emph{infinite} direct-sum decompositions by monoid-theoretical means.
Our paper presents a first step in this endeavor.

More generally, for any infinite cardinal $\kappa$ we introduce the notion of a $\kappa$-monoid (\cref{def:kappa-monoid}), a structure that allows sums on index sets of cardinality $\kappa$, with the sum obeying an associativity law paralleling the one for direct sums of modules.
Every $\kappa$-monoid is a commutative monoid (\cref{l:kappa-basic} and the discussion following it).
A variant of the Eilenberg--Mazur swindle applies (\cref{infinite}) to show that the monoid is also reduced (that is, any equation of the form $a+b=0$ implies $a=b=0$), as we expect for a monoid of modules.
We establish the basic properties of $\kappa$-monoids in \cref{sec:kappa-monoids}.
For a regular cardinal $\lambda$ we also define a $\lambda^-$-monoid to be an analogous structure in which sums over index sets of cardinality strictly less than $\lambda$ are defined (\cref{subsec:kappaminus-monoids}).

It is not hard to see that a reduced commutative monoid may admit several extensions to a $\kappa$-monoid.
For instance, one can extend the monoid $\bR_{\ge 0}$ to an $\aleph_0$-monoid $\bR_{\ge 0} \cup \{\infty\}$, by using the notion of convergent (and divergent) series to assign a value to countable sums.
The restriction to non-negative reals ensures the required associativity (and in turn, commutativity).
Alternatively, one can declare any countable sum in $\bR_{\ge 0}$ with infinite support to have value $\infty$, to obtain a different $\aleph_0$-monoid structure on $\bR_{\ge 0} \cup \{\infty\}$.
A first natural question is therefore:
which $\kappa$-monoids appear as $\kappa$-monoids of modules?

\cref{sec:braiding} lays the groundwork for answering this question (to some degree).
We introduce the crucial property for two $\kappa$-indexed families over a $\lambda^-$-monoid to be \emph{$\lambda^-$-braided}.
Consequently we also define the concept of a $\kappa$-monoid $H$ being \emph{$\lambda^-$-braided} over a $\lambda^-$-submonoid $X$.
As one of our main results, we show that $H$ being $\lambda^-$-braided over $X$ is equivalently characterized by a natural universal property in the category of $\kappa$-monoids, and that this \emph{universal $\kappa$-extension} always exists (\cref{p:braiding-extend,t:universal-exist}).
Returning to the example of the monoid $\bR_{\ge 0}$, perhaps at first surprisingly, neither of the two $\aleph_0$-monoid structures that we mentioned turn out to be the universal $\aleph_0$-extension.
Instead, the universal $\aleph_0$-extension is
\[
  \bR_{\ge 0} \cup \widetilde{\bR}_{> 0} \cup \{\infty\}
\]
(see \cref{e:kappa-braided} and \ref{e:univ-kappa-ext}).

Now let $\cC$ be a class of modules closed under $\kappa$-indexed direct sums, isomorphisms, and direct summands.
Assume that the isomorphism classes of $\cC$ form a set.
For a regular cardinal $\lambda \le \kappa$, let $\cC_{\lambda^-}$ be the subclass of $\cC$ consisting of modules generated by strictly fewer than $\lambda$ many elements.
Connecting the notion of $\kappa$-monoids back to monoids of modules, our first main result is \cref{t:module-braiding}, showing that $\Vmon^{\kappa}(\cC)$ is the universal $\kappa$-extension of $\Vmon^{\lambda^-}(\cC_{\lambda^-})$ if every module in $\cC$ is a direct sum of at most $\kappa$ many modules of $\cC_{\lambda^-}$.
Thus, whereas a $\lambda^-$-monoid affords potentially different extensions to a $\kappa$-monoid, in the module-theoretic setting there is a unique extension, characterized by a natural universal property (under the assumption that every module in $\cC$ is a direct sum of at most $\kappa$ many $<\!\lambda$-generated modules).
More explicitly, in the setting of modules, the fact that every module in $\cC$ is a direct sum of at most $\kappa$ many $<\!\lambda$-generated modules suffices to ensure that also all relations between direct sums of modules are induced by direct sums of $<\!\lambda$-generated modules on index sets of cardinality $<\!\lambda$.

The proof of \cref{t:module-braiding} proceeds by showing that $\Vmon^{\kappa}(\cC)$ is $\lambda^-$-braided over $\Vmon^{\lambda^-}(\cC_{\lambda^-})$.
At its heart, this is a transfinite recursion of an elementary module-theoretic property recently observed by Bergman \cite{Bergman23}.
Naturally, a similar strategy and property already appears in Levy and Robson's characterization of non-finitely generated modules over hereditary noetherian prime rings (see the proof of \cite[Theorem 46.1]{LevyRobson11}).

The remainder of \cref{sec:modules} is concerned with the implications of this result on the module-theoretic side, and with a discussion of the existing literature in the context of this new setting of $\kappa$-monoids.
We summarize the key points in the following.

Kaplansky's Theorem states that every projective module is a direct sum of countably generated modules.
However, this theorem says nothing about relations between infinite direct sums of countably generated projective modules.
In the $\kappa$-monoid language, Kaplansky's Theorem says that the $\kappa$-monoid $\Vmon^{\kappa}(R)$ is generated by the subset $\Vmon^{\aleph_0}(R)$.
The $\aleph_0$-monoid $\Vmon^{\aleph_0}(R)$ then completely determines $\Vmon^{\kappa}(R)$ (\cref{cor:projective-braided}).
Translated back into more elementary terms, we obtain the following supplement to Kaplansky's Theorem: every relation between infinite direct sums of projective modules is induced by relations between countable direct sums of countably generated projective modules.

In many interesting classes of rings, in particular, over hereditary rings, projective modules are direct sums of finitely generated modules (see \cref{cor:fin-braided} and the paper by McGovern, Puninski, and Rothmaler \cite{McGovernPuninskiRothmaler07}).
In this case, our results even show that $\Vmon(R)$ fully determines $\Vmon^{\kappa}(R)$.
So, for instance, if $\Vmon(R) \cong \bR_{\ge 0}$ for a hereditary ring (such a ring exists by the Bergman--Dicks Theorem), then the description of the universal $\aleph_0$-extension of $\bR_{\ge 0}$ implies $\Vmon^{\aleph_0}(R) \cong \bR_{\ge 0} \cup \widetilde{\bR}_{>0} \cup \{\infty\}$, giving us a description of the countably generated projective modules for free.
In particular, since the monoids appearing as $\Vmon(R)$ of hereditary rings are precisely the reduced commutative monoids with order-unit, the $\kappa$-monoids appearing as $\Vmon^{\kappa}(R)$ for hereditary rings $R$ are precisely the universal $\kappa$-extensions of reduced commutative monoids with order-unit.
  
In \cref{e:module-examples} we discuss several important classes of monoids of modules appearing in the literature in the context of the $\kappa$-monoid structure.
In particular, we describe the case of the semilocal noetherian rings studied by Herbera and Příhoda, and that of hereditary noetherian prime rings studied by Levy and Robson.

Kaplansky used his result to show that projective modules over local rings are free, by reducing to the countably generated case.
In a similar spirit, our results allow the lifting of the KRSA property.
Theorems of the following type can easily be obtained: if $\cC$ is closed under isomorphism, direct sums, and direct summands, and every module in $\cC$ is a direct sum of finitely generated modules, then KRSA for finite direct sums of finitely generated modules implies KRSA for arbitrary direct sums in $\cC$.

Finally, in \cref{sec:twogen} we give a more explicit description of two-generated $\aleph_0$-monoids that appear as $\Vmon^{\aleph_0}(R)$ of a hereditary ring (with the cyclic case being trivial).

\smallskip\paragraph{Notations and conventions.}
All rings considered are unital and associative, but in general not commutative.
Modules are right modules, unless specified otherwise.
We use infinite ordinals and cardinals, but to avoid set-theoretical difficulties, we always consider such cardinals up to a fixed bound $\kappa$.
As underlying axiom system we have in mind the usual ZFC (with informal classes).
As usual in this context, we use the von Neumann definition of ordinals, and a cardinal is the smallest ordinal of a given cardinality (but cardinals are by default just considered as sets, forgetting the well-order).
We make liberal use of the axiom of choice, to ensure that the basic properties of cardinal arithmetic are well-behaved.
See \cite[Chapter 5]{roitman90} for background.
The cardinality of a set $X$ is denoted by $\card{X}$.

\smallskip\paragraph{Acknowledgments.}
Nazemian would like to thank George M. Bergman for useful discussions.
Smertnig thanks Andrej Bauer for discussions about autoformalizing the paper.

\smallskip\paragraph{Funding.}
The authors were supported by the Austrian Science Fund (FWF): P 36742.
Smertnig was supported by the Slovenian Research and Innovation Agency (ARIS): Grants P1-0288 and J1-60025.
Part of the research was conducted while Smertnig was employed by the University of Graz, NAWI Graz, Austria.

\smallskip\paragraph{Use of AI/LLMs.}
LLMs were used to produce a Lean autoformalization of large parts of the paper (excluding the list of examples \cref{e:module-examples} connecting our results to the literature).
Mainly, this uncovered some small issues in the proofs of \cref{l:braided-no-order,l:braid-repartition} and in \cref{p:ext-diophantine}, which we have fixed in the present version of the paper.
LLMs were also used for additional proof-reading and checking of the paper.
No LLM-generated texts or proofs are contained in the paper.

\section{Infinite summation: \texorpdfstring{$\kappa$}{κ}-monoids} \label{sec:kappa-monoids}

In this section we introduce \defit{$\kappa$-monoids} for infinite cardinals $\kappa$.
This notion extends the notion of a commutative monoid to allow the summation of up to $\kappa$ many elements.
The axioms are taken in such a way that they model the behavior of direct sums of modules.

If $x \coloneqq (x_i)_{i \in \kappa} \in H^{\kappa}$ is a family indexed by elements of $\kappa$ over some set $H$, and $\Sigma \colon H^{\kappa} \to H$ is a map, we use the notation
\[
  \sum_{i \in \kappa} x_i \coloneqq \Sigma\big( (x_i)_{i \in \kappa}\big),
\]
mimicking the typical summation notation. 
Further, $\supp(x) \coloneqq \{\, i \in \kappa : x_i \ne 0 \,\}$ denotes the \defit{support} of the family.

\begin{defi} \label{def:kappa-monoid}
  Let $\kappa$ be an infinite cardinal.
  A $\kappa$-monoid is a set $H$ together with an element $0 \in H$ and a map $\Sigma \colon H^{\kappa} \to H$ such that the following conditions are satisfied.
  \begin{enumerate}[label=\textup{(A\arabic*)},leftmargin=1.25cm]
  \item \label{a:trivial} If $x=(x_i)_{i \in \kappa} \in H^\kappa$ with $x_i=0$ for all $i \ne 0$, then $\sum_{i \in \kappa} x_i=x_0$.
  \item \label{a:flatten}
    If $(x_{i,j})_{i,j \in \kappa} \in H^{\kappa \times \kappa}$ and $\pi \colon \kappa \times \kappa \to \kappa$ is a bijection, then
    \[
      {\sum_{i \in \kappa}} {\sum_{j \in \kappa}} x_{i,j} = {\sum_{k \in \kappa}} x_{\pi^{-1}(k)}.
    \]
  \end{enumerate}
\end{defi}

\begin{remark}
  \begin{enumerate}
  \item Axiom~\ref{a:flatten} will automatically imply that all $\kappa$-monoids are commutative (see \cref{l:kappa-basic} below).
    To introduce a notion that permits noncommutative operations, one should presumably parametrize the family by ordinals and assume an order-preserving bijection in the analogue of \ref{a:flatten}.
    However, we have no use for such a notion, and do not pursue this.

  \item It is occasionally useful to permit an arbitrary index set $I$ with $\card{I}=\kappa$.
    One can define the summation over the index set $I$ by choosing an arbitrary bijection between $I$ and $\kappa$.
    This is well-defined and independent of the chosen bijection, as \cref{l:kappa-basic} below will show that the operation does not depend on the order of the elements in the summation.

  \item Infinitary operations have occasionally appeared before in the literature.
    In particular, our notion of a $\kappa$-monoid corresponds to that of a \emph{complete $\kappa$-$\Sigma$-algebra} satisfying axioms (U) and (GP) in \cite[Chapter IV.1]{HebischWeinert98}.
    We also mention the notes \cite{vocab}, where infinitary monoids appear in 3.4.11.8 on page 89.
  \end{enumerate}
\end{remark}

We start with a number of easy examples of $\kappa$-monoids.

\begin{examples} \label{e:kappa}
  \mbox{}
    \begin{enumerate}
  \item \label{e-kappa:trivial} Let $(M,+,0)$ be a reduced\footnote{A monoid is \defit{reduced} (or \defit{conical}) if $a+b=0$ implies $a=b=0$, for $a$,~$b \in M$.} commutative monoid and set $H \coloneqq M \uplus \{\infty\}$.
    For a cardinal $\kappa$ and a family $x=(x_i)_{i \in \kappa}$ we define $\Sigma^\kappa\big( (x_i)_{i \in \kappa} \big)$ to be the sum of the family in $M$ if the support of $x$ is finite and $x_i \in M$ for all $i \in \kappa$, and we define $\Sigma^\kappa\big( (x_i)_{i \in \kappa} \big) = \infty$ otherwise.
    That is, we have $\Sigma^\kappa\big( (x_i)_{i \in \kappa} \big)=\infty$ if $x_i=\infty$ for some $i$ or infinitely many $x_i$ satisfy $x_i \ne 0$.
    With these operations, the set $H$ is a $\kappa$-monoid for every cardinal $\kappa$, the \defit{trivial $\kappa$-extension} of $M$.
    
    The assumption that $M$ is reduced is necessary for \ref{a:flatten} to hold: if we had $0 \ne a$ and $-a$ in $M$, then
    \[
      \sum_{i \in \kappa} (a + (-a) + 0 + \cdots) = 0, \qquad\text{but} \qquad \sum_{i \in \kappa} x_i = \infty,
    \]
    with $x_i$ alternating between $a$ and $-a$.

  \item \label{e-kappa:reals} Let $H = \bR_{\ge 0}\cup \{\infty\}$.
    For a sequence $(x_i)_{i \ge 0}$, define
    \[
      \Sigma^{\aleph_0} (x_i)_{i \ge 0} \coloneqq \sum_{i=0}^\infty x_i,
    \]
    where the sum on the right side is understood to be the sum of the corresponding series in the usual sense of analysis (i.e., the limit of the partial sums).
    Since $x_i \ge 0$ for all $i \ge 0$, the summation does not depend on the order of the elements and is associative for double-indexed sequences.
    Therefore $H$ is an $\aleph_0$-monoid.
    Observe that the operation differs from the one in the first example in the value that is assigned to convergent series.

  \item \label{kappa:cardinals}
    If $I$ is a set, $(\alpha_i)_{i \in I}$ is a family of cardinals, and $(X_i)_{i \in I}$ is a disjoint family of sets with $\card{X_i}=\alpha_i$ for each $i$, then one defines
    \[
      \sum_{i \in I} \alpha_i \coloneqq \card[\Big]{\biguplus_{i \in I} X_i}.
    \]
    The axiom of choice implies that this definition does not depend on the particular choice of the sets $X_i$.

    If $\kappa$ is an infinite cardinal, and $\alpha_i \le \kappa$ for all $i \in I$ as well as $\card{I} \le \kappa$, then $\sum_{i \in I} \alpha_i \le \kappa$.
    Thus the set $F_\kappa \coloneqq \{\, \alpha : \alpha \le \kappa \,\}$ of all cardinals bounded by $\kappa$ is naturally a $\kappa$-monoid.

  \item \label{kappa:module-class}
    Let $R$ be a ring and let $\mathcal C$ be a class of modules whose isomorphism classes form a set.
If $\mathcal C$ is closed under isomorphisms and direct sums over index sets of cardinality at most $\kappa$, then the set of isomorphism classes naturally forms a $\kappa$-monoid.
In particular, the class $\mathcal C$ of all projective (right) modules generated by at most $\kappa$ many elements is closed under isomorphism and direct sums over index sets of cardinality at most $\kappa$.
  \end{enumerate}
\end{examples}

The $\kappa$-monoids in \ref{kappa:module-class} will be the main instances of $\kappa$-monoids of interest later in the paper.
Hence we introduce some corresponding notation.

\begin{defi} \label{def:vmon}
  Let $R$ be a ring, let $\mathcal C$ be a class of modules whose isomorphism classes form a set, and let $\kappa$ be an infinite cardinal.
  \begin{enumerate}
  \item
    If $\mathcal C$ is closed under isomorphisms and direct sums over index sets of cardinality at most $\kappa$, we denote the $\kappa$-monoid of isomorphism classes of modules in $\mathcal C$ by $\Vmon^{\kappa}(\mathcal C)$.

  \item $\Vmon^{\kappa}(R)$ is the $\kappa$-monoid of isomorphism classes of projective \textup(right\textup) modules generated by at most $\kappa$ elements.

  \item $\Vmon(R)$ is the monoid of finitely generated projective modules. 
  \end{enumerate}
\end{defi}

The stated axioms for $\kappa$-monoids are sufficient to derive commutativity and associativity laws, as we show next.

\begin{lemma} \label{l:kappa-basic}
  Let $H$ be a $\kappa$-monoid.
  Then the following properties hold.
  \begin{enumerate}[label=\textup{(A\arabic*)},leftmargin=1.2cm]
    \setcounter{enumi}{2}
  \item \label{a:commutative} If $x=(x_i)_{i \in \kappa} \in H^\kappa$ and $\pi \colon \kappa \to \kappa$ is bijective, then
    \[
      \sum_{i \in \kappa} x_i = \sum_{i \in \kappa} x_{\pi(i)}.
    \]
  \item \label{a:associative}
    If $(x_{i,j})_{i,j \in \kappa} \in H^{\kappa \times \kappa}$, then
    \[
      {\sum_{i \in \kappa}} {\sum_{j \in \kappa}} x_{i,j} = {\sum_{j \in \kappa}} {\sum_{i \in \kappa}} x_{i,j}.
    \]
  \end{enumerate}
\end{lemma}

\begin{proof}
  \begin{proofenumerate}
  \item[\ref{a:commutative}]
  This commutativity law is a consequence of \ref{a:flatten}:
  define the family $(y_{i,j})_{i,j \in \kappa}$ by $y_{i,0}=x_i$ and $y_{i,j}=0$ for $j \ne 0$.
  By \ref{a:trivial}, we have $\sum_{j \in \kappa} y_{i,j} = x_i$.
  Let $f\colon \kappa \times \kappa \to \kappa$ be a bijection.
  Then $f \circ (\pi^{-1},\id) \colon \kappa \times \kappa \to \kappa$ is also a bijection.
  Since \ref{a:flatten} holds for arbitrary bijections, we get
  \[
    \sum_{i \in \kappa} x_i = \sum_{i \in \kappa} \sum_{j \in \kappa} y_{i,j} =\sum_{l \in \kappa} y_{(\pi,\id) \circ f^{-1}(l)}.
  \]
  Applying once more \ref{a:flatten} to $(y_{\pi(i),j})_{i,j \in \kappa}$  and the bijection $f$, gives
  \[
    \sum_{l \in \kappa} y_{(\pi,\id) \circ f^{-1}(l)} = \sum_{i \in \kappa} \sum_{j \in \kappa} y_{\pi(i),j} = \sum_{i \in \kappa} x_{\pi(i)}.
  \]

  \item[\ref{a:associative}]
  Let $f \colon \kappa \times \kappa \to \kappa$ be a bijection and let $\tau\colon \kappa \times \kappa \to \kappa \times \kappa$ be the transposition $\tau(i,j)=(j,i)$.
  Using \ref{a:flatten} twice,
  \[
    \sum_{i \in \kappa} \sum_{j \in \kappa} x_{i,j} = \sum_{l \in \kappa} x_{\tau \circ f^{-1}(l)} = \sum_{j \in \kappa} \sum_{i \in \kappa} x_{i,j}. \qedhere
  \]
  \end{proofenumerate}
\end{proof}

Suppose that $H$ is a $\kappa$-monoid with operation $\Sigma^\kappa=\Sigma$.
If $\alpha < \kappa$ is another cardinal (not necessarily infinite), and $\iota \colon \alpha \to \kappa$ is an injective map, we define $\Sigma^\alpha \colon H^\alpha \to H$ by
\[
  \Sigma^\alpha\big( (x_i)_{i \in \alpha} \big) \coloneqq \Sigma^\kappa\big( (x_{\iota^{-1}(j)})_{j \in \kappa} \big),
\]
where we use the convention $x_{\iota^{-1}(j)}=0$ if $j$ is not in the image of $\iota$.
The commutativity property \ref{a:commutative} ensures that this definition is independent of the particular choice of $\iota$.
It is now easy to check the following.

\begin{itemize}
\item If $\alpha \le \kappa$ is an infinite cardinal, then $(H, \Sigma^\alpha)$ is an $\alpha$-monoid.
\item $(H,\Sigma^2)$ is a commutative monoid, and we write $+ \coloneqq \Sigma^2$.
\item For every finite $n > 2$, the map $\Sigma^n \colon H^n \to H$ is an $n$-ary operation, and one checks using \ref{a:flatten} that
  \[
    \Sigma^n(x_1,\ldots,x_n) = x_1 + (x_2 + (x_3 + \cdots + x_n)) = x_1 + \cdots + x_n,
  \]
  where $+=\Sigma^2$ is the binary operation of the previous point.
\item As degenerate special cases, the map $\Sigma^1 \colon H \to H$ is the identity, and with the convention that $H^0=\{0\}$, the map $\Sigma^0 \colon \{0\} \to H$ is the constant map equal to $0$.
\end{itemize}

Given these properties, and for consistency of notation, an $n$-monoid for a finite cardinal $n$ shall simply be a commutative monoid if $n \ge 2$, shall be the set $H$ together with the distinguished element $0$ and the identity map if $n=1$, and shall be the set $H$ together with the distinguished element $0$ and the constant zero map if $n=0$.

\begin{defi}
  Let $H$ be a $\kappa$-monoid.
  Let $x \in H$ and $\alpha \leq \kappa$ a \textup(finite or infinite\textup) cardinal.
  We define
  \[
    \alpha x \coloneqq \sum_{i \in \alpha} x,
  \]
  where the sum on the right side is to be understood in terms of the natural $\alpha$-monoid structure on $H$.
\end{defi}

If $\alpha$ is infinite, then property \ref{a:flatten} implies $x + \alpha x  =  \alpha x $.
For this scalar multiplication, the expected properties hold, as we now show.

\begin{lemma} \label{l:kappa-rules}
  Let $H$ be a $\kappa$-monoid and let $x \in H$. Then
  \begin{enumerate}
  \item \label{kappa-rules:trivial} $0 x = 0$ and $1x=x$,
  \item \label{kappa-rules:distr1} If $(\lambda_i)_{i \in I}$ is a family of cardinals with $\card{I} \leq  \kappa$ and $\lambda_i \leq  \kappa$ for $i \in I$, then
    \[
      \Big(\sum_{i \in I} \lambda_i\Big) x = \sum_{i \in I} \lambda_i x.
    \]
  \item \label{kappa-rules:distr2} If $\alpha \le \kappa$ is a cardinal and $(x_i)_{i \in I}$ is a family in $H$ with $\card{I} \le \kappa$, then
    \[
      \alpha \sum_{i \in I} x_i = \sum_{i \in I} \alpha x_i.
    \]
  \end{enumerate}
\end{lemma}

\begin{proof}
  \begin{proofenumerate}
  \item[\ref{kappa-rules:trivial}] is immediate.

  \item[\ref{kappa-rules:distr1}]
  If $x=0$, then the equality holds trivially. We may assume $x\ne 0$.
  By definition
  \[
    \sum_{i \in I} \lambda_i x = \sum_{i \in I} \sum_{j \in \lambda_i} x.
  \]
  Let $\lambda \coloneqq \sup\{\, \lambda_i : i \in I \,\}$.
  Then $\lambda \le \sum_{i \in I} \lambda_i \le \kappa$.
  We define $x_i=(x_{i,j})_{j \in \lambda} \in H^\lambda$ by $x_{i,j} = x$ if $j \in \lambda_i$ and $x_{i,j}=0$ otherwise.
  Then
  \[
    \sum_{i \in I} \sum_{j \in \lambda_i} x = \sum_{i \in I} \sum_{j \in \lambda} x_{i,j}.
  \]
  Now $\card{ \{\, (i,j) \in I \times \lambda : x_{i,j} \ne 0 \,\}} = \sum_{i \in I} \lambda_i$ by construction.
  By \ref{a:flatten}, it therefore holds that $\sum_{i \in I} \lambda_i x = (\sum_{i \in I} \lambda_i) x$.

  \item[\ref{kappa-rules:distr2}]
  Unwrapping the definitions and applying \ref{a:associative},
  \[
    \begin{split}
      \alpha \sum_{i \in I} x_i = \sum_{j \in \alpha} \sum_{i \in I} x_i = \sum_{i \in I} \sum_{j \in \alpha} x_i = \sum_{i \in I} \alpha x_i. \qedhere
    \end{split}
  \]
  \end{proofenumerate}
\end{proof}

In \ref{e-kappa:trivial} of Example~\ref{e:kappa} we have embedded an arbitrary reduced commutative monoid $M$ in a $\kappa$-monoid.
We can now see that $M$ necessarily has to be reduced, as a variant of the Eilenberg--Mazur swindle is applicable.

\begin{lemma} \label{infinite}
  Let $H$ be a $\kappa$-monoid with $\kappa$ an infinite cardinal. Then
  \begin{enumerate}
  \item \label{infinite:reduced} $H$ is reduced,
  \item \label{infinite:split} $t_1 + t_2 = \kappa t_3$ implies that $t_1 + \kappa t_3 = \kappa t_3$.
  \end{enumerate}
\end{lemma}

\begin{proof}
  We use the properties from Lemma~\ref{l:kappa-rules}.

  \begin{proofenumerate}
  \item[\ref{infinite:reduced}]
  Suppose that $x+y=0$ for $x$,~$y \in H$. Then
  \[
    \begin{split}
      x & = x + 0 = x + \kappa 0 = x + \kappa (x+y) = x + (\kappa x + \kappa y) = (x+ \kappa x) + \kappa y \\
      &=  \kappa x + \kappa  y =   \kappa (x + y) =  0.
    \end{split}
  \]
  By symmetry also $y=0$.

  \item[\ref{infinite:split}]
  We get
  \[
    \begin{split}
      t_1 + \kappa t_3 &= t_1 + \kappa (t_1 + t_2) =  (t_1 + \kappa t_1 ) + \kappa t_2 \\
      &=  \kappa t_1 + \kappa t_2 = \kappa (t_1 + t_2) = \kappa t_3. \qedhere
    \end{split}
  \]
  \end{proofenumerate}
\end{proof}

\subsection{The category of \texorpdfstring{$\kappa$}{κ}-monoids and free objects} The set
\[
  F_{\kappa} \coloneqq \{\, \alpha : \alpha \le \kappa \,\}
\]
of all cardinals at most $\kappa$ is a $\kappa$-monoid (\ref{kappa:cardinals} of Examples~\ref{e:kappa}).
The operations here are particularly simple.
If $(\alpha_i)_{i \in I}$ is a family with finite support and all cardinals being finite, then $\sum_{i \in I} \alpha_i$ is just the usual sum of natural numbers.
However if the family has infinite support $I' \subseteq I$ or $\alpha_i$ is infinite for at least one $i$, then
\[
  \sum_{i\in I} \alpha_i = \max\big\{ \card{I'},\, \sup\{\, \alpha_i : i \in I \,\} \big\}.
\]

A \defit{$\kappa$-submonoid $H'$} of a $\kappa$-monoid $H$ is a subset $H' \subseteq H$ that contains $0$ and is closed under the restriction of the operation of $H$.
A \defit{homomorphism of $\kappa$-monoids} (in short, a \defit{$\kappa$-homomorphism} or just \defit{homomorphism}) is a map that respects the operation and preserves the neutral element.
The class of $\kappa$-monoids together with homomorphisms forms a category.

Let $B$ be a set.
The Cartesian product $F_{\kappa}^B$ is a $\kappa$-monoid with coordinate-wise addition.
Let $\iota\colon B \hookrightarrow F_\kappa^B$ denote the map of sets that maps $b$ to the family $(x_i)_{i \in B}$ having $x_i=1$ if $i=b$ and $x_i=0$ otherwise.
The following straightforward construction produces free objects in the category of $\kappa$-monoids.

\begin{prop} \label{p:free}
  Let $\kappa$ be an infinite cardinal and let $B$ be a set.
  Let
  \[
    F_{\kappa}(B) \coloneqq \{\, x \in F_\kappa^{B} : \card{\supp(x)} \le \kappa \,\}.
  \]
  Then the following statements hold.
  \begin{enumerate}
  \item \label{free:submonoid} $F_{\kappa}(B)$ is a $\kappa$-submonoid of $F_\kappa^B$.
    \item \label{free:up}  For every $\kappa$-monoid $H$ and every map of sets $f \colon B \to H$, there exists a 
  unique $\kappa$-homomorphism $\overline f \colon F_\kappa(B) \to H$  satisfying $\overline f \circ \iota = f$.
  \end{enumerate}
\end{prop}

\begin{proof}
  \begin{proofenumerate}
  \item[\ref{free:submonoid}]
  Let $(x_i)_{i \in \kappa}$ be a family of elements of $F_\kappa(B)$ with $\card{\supp(x_i)}=\beta_i \le \kappa$.
  Then $\supp(\sum_{i \in \kappa} x_i)$ has cardinality at most $\sum_{i \in \kappa} \beta_i$ (this uses the axiom of choice).
  Now $\sum_{i \in \kappa} \beta_i \le \kappa \sup\{\, \beta_i : i \in \kappa \,\} \le \kappa$.

  \item[\ref{free:up}]
  Now we need to verify the universal property.
  For $x \in F_{\kappa}(B)$ with $x = (\lambda_b)_{b \in B}$, we define
  \[
    \overline f(x) \coloneqq \sum_{b \in B} \lambda_b f(b).
  \]
  This is well-defined because $\card{\supp(x)} \le \kappa$, that is, we understand the sum to be taken only over the nonzero elements.
  By definition $\overline f \circ \iota = f$, and it remains to check that $\overline f$ is a $\kappa$-homomorphism.
  
  Let $x=(x_i)_{i \in \kappa}$ be a family of elements of $F_\kappa(B)$, and let $x_i = (\lambda_{i,b})_{b \in B}$ with $\card{\supp(x_i)} \le \kappa$.
  Then
  \[
    \overline f \Big( \sum_{i \in \kappa} x_i\Big) = \sum_{b \in B} \Big( \sum_{i \in \kappa} \lambda_{i,b} \Big) f(b) = \sum_{b \in B} \sum_{i \in \kappa} \lambda_{i,b} f(b)  = \sum_{i \in \kappa} \sum_{b \in B} \lambda_{i,b} f(b) = \sum_{i \in \kappa} \overline f(x_i).
  \]
  For the second equality we used \ref{kappa-rules:distr1} of Lemma~\ref{l:kappa-rules}.
  The third equality uses \ref{a:associative}, taking into account that $\card{ \{ b \in B : \lambda_{i,b} \ne 0 \text{ for some } i \in \kappa \,\} } \le \kappa$. \qedhere
  \end{proofenumerate}
\end{proof}

When $\card{B} \le \kappa$, in particular when $B$ is finite and $\kappa$ is infinite, one has $F_\kappa(B) = F_\kappa^{B}$.

\begin{defi}
  Let $\alpha$ be a \textup(finite or infinite\textup) cardinal with $\alpha \le \kappa$.
  \begin{enumerate}
  \item A $\kappa$-monoid $H$ is \defit{$\alpha$-generated} if there exists a set $B$ of cardinality $\alpha$ and a surjective homomorphism $F_\kappa(B) \to H$.
  \item A $\kappa$-monoid $H$ is \defit{cyclic} if it is $1$-generated, that is, if it is an image of $F_\kappa$ under a $\kappa$-homomorphism.
  \end{enumerate}

\end{defi}

Equivalently, a $\kappa$-monoid is $\alpha$-generated, if there exists a subset $X \subseteq H$ of cardinality at most $\alpha$, such that every element of $H$ can be expressed as a $\kappa$-sum in $X$.
In this case we write $H = \langle X \rangle_\kappa$.
One has to be mindful of the cardinal $\kappa$ in this context.
For instance, the $\aleph_1$-monoid $F_{\aleph_1}$ is cyclic (generated by the cardinal $1$) as $\aleph_1$-monoid, but it is not cyclic as $\aleph_0$-monoid.

\subsection{Order-units}

\begin{defi}
  Let $H$ be a $\kappa$-monoid.
  For $x$,~$y \in H$, let $x \le y$ if there exists $x' \in H$ such that $x+x'=y$.
\end{defi}

The relation $\le$ is reflexive and transitive, and therefore a preorder on $H$, but it is not necessarily anti-symmetric.
For instance, consider the monoid $H=\bN_0/\!\sim$ where $n \sim m$ if and only if $m \equiv n \mod 2$. Then $H = \{ \overline 0, \overline 1\}$, with $\overline 0 + \overline 1 = \overline 1$ and $\overline 1 + \overline 1 = \overline 0$, so that $\overline 0 \le \overline 1$ and $\overline 1 \le \overline 0$, while $\overline 0 \ne \overline 1$.

If $H$ is a commutative monoid, then an element $u \in H$ is called an \defit{order-unit} if for every $x \in H$ there exists some $n \in \bN_0$ such that $x \le nu$.
In monoids of finitely generated projective modules, that is, monoids of the form $\Vmon(R)$, the isomorphism class $[R]$ is an order-unit (though it may not be unique, every progenerator gives rise to an order-unit).
In the $\kappa$-monoids $\Vmon^\kappa(R)$, the isomorphism class $[R]$ (or more generally, any progenerator) will play a similar distinguished role, and so we are led to the following definition.

\begin{defi}
  Let $H$ be a $\kappa$-monoid with $\kappa$ an infinite cardinal.
  \begin{enumerate}
  \item An element $u \in H$ is an \defit{order-unit} if for all $x \in H$, one has $x \le \kappa u$.
  \item An order-unit $u \in H$ is \defit{faithful} if $\beta u \not\le \alpha u$ for all cardinals $\alpha < \beta \le \kappa$ with $\beta$ infinite.
  \end{enumerate}
\end{defi}

Let $H$ be a $\kappa$-monoid with order-unit $u$.
For an element $x \in H$ we define $\size(x)$ to be the minimal cardinal $\alpha$ such that $x \le (\alpha +n)u$ for some $n \in \bN_0$.
In other words, $\size(x) = 0$ if $x \le nu$ for some $n \in \bN_0$ and otherwise $\size(x)$ is the smallest (infinite) cardinal $\alpha$ such that $x \le \alpha u$.
Alternatively, we may think of $\size$ as taking values in cardinals modulo finite cardinals, that is, identifying $\alpha \le \beta$ if there exists $n \in \bN_0$ such that $\alpha+n=\beta$.
We must keep in mind that the definition of $\size$ depends on the choice of order-unit.

Let $H_\alpha \coloneqq \{\, x \in H : \size(x) \le \alpha \,\}$.
Then $H_\alpha$ is an $\alpha$-submonoid of $H$ if $\alpha$ is infinite, and a submonoid if $\alpha=0$.
We obtain a filtration
\[
  H_0 \subseteq H_{\aleph_0} \subseteq H_{\aleph_1} \subseteq \cdots
\]
The order-unit $u$ is faithful if and only if all the inclusions in this filtration are proper, that is, if and only if $H_\alpha \subsetneq H_\beta$ whenever $\alpha < \beta \le \kappa$ and $\beta$ is infinite.

\begin{example}
  Let $R$ be nonzero and $H = \Vmon^{\kappa}(R)$.
  Then $[R]$ is an order-unit.
  If $I$, $J$ are infinite sets such that $R_R^{(I)} \oplus P_R \cong R_R^{(J)}$ for some module $P$, then $\card{I} \le \card{J}$, because $R_R^{(I)}$ can be generated by at most $\card{J}$ elements, and the rank of a non-finitely generated free module is unique.
  Thus $[R]$ is a faithful order-unit.
  The order-unit $[R]$ is not unique: If $P$ is a progenerator, then $[P]$ is also a faithful order-unit.
  
  For the filtration above, we get $H_0 = \Vmon(R)$, $H_{\aleph_0} = \Vmon^{\aleph_0}(R)$, and more generally $H_\alpha = \Vmon^{\alpha}(R)$ for all infinite cardinals $\alpha \le \kappa$.
  By a theorem of Kaplansky every projective module is a direct sum of countably generated projective modules \cite{Kaplansky58}.
  Thus, $H$ is generated as a $\kappa$-monoid by the subset $H_{\aleph_0}$ of isomorphism classes of countably generated projective modules.
\end{example}

For later use we make note of the following easy fact.

\begin{lemma} \label{sumwithbig}
  Let $\kappa$ be an infinite cardinal and suppose $H$ is a $\kappa$-monoid with order-unit $u$.
  If $t = \kappa u +l $ for some $l \in H$, then $t = \kappa u$.
\end{lemma}

\begin{proof}
  There exists $l' \in H$ such that $l + l' = \kappa u$.
  Then \ref{infinite:split} of Lemma \ref{infinite} implies $l + \kappa u = \kappa u$.
  Therefore $t = \kappa u$.
\end{proof}

\subsection{Basic realization results}

We shall later be concerned with the realization of $\kappa$-monoids as $\Vmon^{\kappa}(R)$ for hereditary rings.
It is illustrative, to familiarize ourselves with the notions, to first consider two very easy realization results of this style.

\subsubsection{Monoids of free modules} \label{ss:free-modules}
It is well-known that noncommutative rings do not necessarily have invariant basis number (IBN).
That is, over a noncommutative ring it may happen that $R_R^m \cong R_R^n$ for $m \ne n$.
However, for free modules of infinite rank that rank is uniquely determined.
Let $\mathcal F$ denote the class of all finitely generated free $R$-modules.
Then $\Vmon(\mathcal F)$ is a cyclic monoid, generated by $[R]$.

Every cyclic monoid is isomorphic to either $\bN_0$ or to $C_{m,n}\coloneqq\bN_0/\!\sim_{m,n}$ where $m \ge 0$, $n \ge 1$ and the congruence relation $\sim_{m,n}$ is given by $k \sim_{m,n} l$ if and only if $k=l$ or $\min\{k,l\} \ge m$ and $k \equiv l \mod n$.
Thus we may view $C_{m,n}$ as $m+n$ element set
\[
   \{ \overline 0, \overline 1, \ldots, \overline m, \overline{m+1}, \dots, \overline{m+(n-1)} \},
 \]
with $\overline{l+kn}$ identified with $\overline{l}$ whenever $l \ge m$ and $k \ge 0$.
Note that $\overline{1}$ is an order-unit.

A classical result of Leavitt shows that for every cyclic monoid $C_{m,n}$ with $m \ge 1$, there exists a ring $R$ with $\Vmon(\mathcal F) \cong C_{m,n}$ \cite{Leavitt62}.
Realizing a monoid of the form $C_{0,n}$ implies $0 \cong R^n$ for some $n \ge 1$, so only trivial ring $R=0$ can realize such a monoid, forcing $n=1$.
Thus, the cyclic monoids realizable as $\Vmon(\mathcal F)$ are
\[
  \bN_0, \quad C_{m,n} \text{ with } m, n \ge 1 \quad \text{and} \quad C_{0,1} = \{0\}.
\]

Let us now consider the $\kappa$-monoid $\Vmon^\kappa(\mathcal F^\kappa)$ where $\mathcal F^\kappa$ is the class of all free $R$-modules generated by at most $\kappa$ elements.
As the rank of an infinite rank free module is always uniquely determined, whereas we can easily construct cyclic $\kappa$-monoids in which some intervals of cardinals are identified, it is clear that we cannot expect every cyclic $\kappa$-monoid to appear as $\Vmon^{\kappa}(\mathcal F^\kappa)$.
The generator of a cyclic $\kappa$-monoid is always an order-unit; the additional assumption for the generator to be a \emph{faithful} order-unit fixes the realization issue.

\begin{lemma}
  Let $\kappa$ be an infinite cardinal.
  Let $C$ be a cyclic $\kappa$-monoid generated by a faithful order-unit.
  Then $C$ is isomorphic to
  \[
    C_{0} \ \uplus \{\, \alpha : \alpha \text{ a cardinal with } \aleph_0 \le \alpha \le \kappa  \,\},
  \]
  with $C_0$ the cyclic submonoid of size zero elements, and the obvious operation.
\end{lemma}

\begin{proof}
  Let $u$ be a faithful order-unit that generates $C$ as $\kappa$-monoid.
  By faithfulness of $u$, necessarily $u \ne 0$.
  The submonoid $C_0 = \{\, nu : n \in \bN_0 \,\}$ of size zero elements is a cyclic monoid.
  Faithfulness of $u$ implies that $\alpha u \ne \beta u$ for infinite cardinals $\alpha \ne \beta$, and moreover that $\{\, \alpha u : \aleph_0 \le \alpha \le \kappa \,\}$ is disjoint from $C_0$.
  The claimed isomorphism is therefore given by $\alpha u \mapsto \alpha$.
\end{proof}

\begin{prop}
  Let $\kappa$ be an infinite cardinal and let $C$ be a cyclic $\kappa$-monoid generated by a faithful order-unit.
  Then there exists a ring $R$ such that its class of $\kappa$-generated free modules $\mathcal F^\kappa$ satisfies $\Vmon^\kappa(\mathcal F^\kappa) \cong C$.
\end{prop}

\begin{proof}
  Let $C_0 \subseteq C$ be the submonoid of size zero elements.
  Then $C_0 \cong \bN_0$ or $C_0 \cong C_{m,n}$ for some $m \ge 0$ and $n \ge 1$.
  If $m =0$, then $n u = 0 = \kappa u$, which contradicts the faithfulness of $u$.
  Thus $m \ge 1$.

  By Leavitt's theorem \cite{Leavitt62}, there exists a ring $R$ such that $\Vmon(\mathcal F) \cong C_0$ as monoids.
  Since $\Vmon^\kappa(\mathcal F^\kappa)$ and $C$ are both cyclic $\kappa$-monoids generated by faithful order-units, and whose size zero submonoids are isomorphic, the previous lemma implies $\Vmon^\kappa(\mathcal F^\kappa) \cong C$ as $\kappa$-monoids.
\end{proof}

Thus, in a trivial extension of Leavitt's result, the $\kappa$-monoids having realization as $\kappa$-monoids of free modules over a nonzero ring are precisely the cyclic $\kappa$-monoids generated by a faithful order-unit.

\subsubsection{Semisimple rings}

If $R$ is a semisimple ring, it has a finite number of isomorphism classes of simple modules.
Every module is a direct sum of simple modules, and the decomposition is unique up to order and isomorphism of the factors.
Moreover, for every integer $k \ge 0$ we can find a semisimple ring with precisely $k$ distinct isomorphism classes of simple modules.
Hence semisimple rings realize precisely the finitely generated free $\kappa$-monoids as their $\kappa$-monoids of projective modules. So we have the following.

\begin{prop}
  \mbox{}
  \begin{enumerate}
  \item If $R$ is a semisimple ring with $n \ge 0$ isomorphism classes of simple modules, then $\Vmon^{\kappa}(R) \cong F_\kappa^n$.
  \item For every $n \ge 0$ there exists a semisimple ring $R$ with $\Vmon^{\kappa}(R) \cong F_\kappa^n$.
  \end{enumerate}
\end{prop}

\subsection{\texorpdfstring{$\kappa^-$}{κ}-monoids} \label{subsec:kappaminus-monoids}

In the following section we also have need for a slight variation on the definition of a $\kappa$-monoid.
Recall that an infinite cardinal $\kappa$ is \defit{regular}, if for every family of sets $(S_i)_{i \in I}$ where $\card{I} < \kappa$ and $\card{S_i} < \kappa$, it holds that
\[
  \card[\Big]{\bigcup_{i\in I}S_i} < \kappa.
\]
For a set $S$ we denote by $S^{(\kappa)}$ the families indexed by $\kappa$ whose support has cardinality strictly less than $\kappa$.

\begin{defi}
  Let $\kappa$ be a regular cardinal.
  A $\kappa^-$-monoid is a set $H$ together with an element $0 \in H$ and a map $\Sigma \colon H^{(\kappa)} \to H$ such that the following conditions are satisfied.
  \begin{enumerate}[label=\textup{(B\arabic*)},leftmargin=1.2cm]
  \item \label{b:trivial} If $x=(x_i)_{i \in \kappa} \in H^{(\kappa)}$ with $x_i=0$ for all $i \ne 0$, then $\sum_{i \in \kappa}x_i=x_0$.
  \item \label{b:flatten}
    If $(x_{i,j})_{i \in \kappa, j \in \kappa} \in H^{\kappa \times \kappa}$ is a family such that
    \begin{itemize}
    \item there are strictly less than $\kappa$ many $i \in \kappa$ for which there exists $j \in \kappa$ with $x_{i,j} \ne 0$,
    \item there are strictly less than $\kappa$ many $j \in \kappa$ for which there exists $i \in \kappa$ with $x_{i,j} \ne 0$, 
    \end{itemize}
    and $\pi \colon \kappa \times \kappa \to \kappa$ is a bijection, then
    \[
      {\sum_{i \in \kappa}} {\sum_{j \in \kappa}} x_{i,j} = {\sum_{k \in \kappa}} x_{\pi^{-1}(k)}.
    \]
  \end{enumerate}
\end{defi}

The analogues of \cref{l:kappa-basic,l:kappa-rules,infinite} hold true (in \cref{l:kappa-rules} the cardinals must be taken to be strictly less than $\kappa$; in \cref{infinite} for \ref{infinite:reduced} one needs $\kappa > \aleph_0$, and in \ref{infinite:split} the cardinal $\kappa$ may be replaced by any infinite cardinal strictly less than $\kappa$).
A $\kappa^-$-homomorphism between $\kappa^-$-monoids is a map respecting the neutral element and the operation.
The set $F_{\kappa^-} = \{\, \lambda : \lambda < \kappa \,\}$ forms a $\kappa^-$-monoid.
Then the analogue of \cref{p:free} holds, so that $F_{\kappa^-}(B)$ is the free $\kappa^-$-monoid on the basis $B$.
An $\aleph_0^-$-monoid is simply a monoid.

Let $\cC$ be a class of modules whose isomorphism classes form a set.
If $\cC$ is closed under isomorphisms and direct sums on index sets of cardinality strictly less than $\lambda$, we define the $\lambda^-$-monoid of modules $\Vmon^{\lambda^-}(\cC)$, analogous to \cref{def:vmon}.

\begin{remark}
  To relate the two notions, note that every $\kappa$-monoid is a $(\kappa^+)^{-}$-monoid, where $\kappa^+$ denotes the successor cardinal of $\kappa$. (Since we have assumed the axiom of choice, every successor cardinal is regular.)

  Conversely, a $\kappa^-$-monoid is a $\lambda$-monoid for all cardinals $\lambda < \kappa$ by restricting the operation, and these $\lambda$-monoid structures are compatible in the natural way.
  On the other hand, given $\lambda$-monoid structures on $H$ for each $\lambda < \kappa$ that are compatible, we can define a $\kappa^-$-monoid structure on $H$.
  The only difficulty is \ref{b:flatten}.
  However, the assumption that $\kappa$ is regular implies that for every family of cardinals $\lambda_i < \kappa$ indexed by a set $I$ of cardinality strictly less than $\kappa$, we have $\sum_{i \in I} \lambda_i = \lambda' < \kappa$.
  Thus, any sum occurring in \ref{b:flatten} can also be expressed in a $\lambda'$-monoid for $\lambda' < \kappa$ sufficiently large.

  Having available the notion of $\kappa^-$-monoids nevertheless is notationally convenient, as it allows us to avoid dealing with a whole family of structures.
\end{remark}

\section{Braiding and universal \texorpdfstring{$\kappa$}{κ}-extensions} \label{sec:braiding}

Throughout the section, let $\lambda$ and $\kappa$ be infinite cardinals (they may be the same or different).
Let $\kappa^+$ denote the successor cardinal of $\kappa$.
We assume throughout that $\lambda$ is regular, as we will be considering $\lambda^-$-monoids.
Recall that we use the notation $\langle X \rangle_\kappa$ to denote the $\kappa$-submonoid of $H$ generated by a subset $X$ of $H$.

While our sums are indexed by the cardinal $\kappa$, or more generally an index set of cardinality $\kappa$, in this section we will often have need to fix a well-order on the index set.
This well-order will always be a limit well-order, that is, every element $i \in \kappa$ must have a successor that we denote by $i+1$ by a slight abuse of notation.
Elements which are not successors are called limit elements (we consider the minimal element, denoted by $0$, to be a limit element).

An \defit{indexed partition} of the set $\kappa$ is a family of pairwise disjoint sets $(I_\mu)_{\mu \in \kappa}$ such that $\kappa = \bigcup_{\mu \in \kappa} I_\mu$.
Some of the sets $I_\mu$ may be empty.

\begin{defi} \label{def:braiding}
  Fix a limit well-order on $\kappa$.
  Let $X$ be a $\lambda^-$-monoid.
  \begin{enumerate}
  \item Two families $(x_i)_{i \in \kappa}$ and $(y_j)_{j \in \kappa}$ in $X$ are \defit{$\lambda^-$-braided} if there exist indexed partitions $(I_{\mu})_{\mu \in \kappa}$ and $(J_\mu)_{\mu \in \kappa}$ of $\kappa$ with $\card{I_\mu}$,~$\card{J_\mu} < \lambda$, and families $(u_\mu)_{\mu \in \kappa}$ and $(v_\mu)_{\mu \in \kappa}$ in $X$ such that $v_\mu=0$ whenever $\mu$ is a limit element, and 
    \[
      \sum_{i \in I_\mu} x_i = v_\mu + u_\mu,
      \qquad\text{and}\qquad
      \sum_{j \in J_\mu} y_j = v_{\mu+1} + u_{\mu} \qquad \text{for all } \mu \in \kappa.
    \]
  \item Suppose $\lambda \le \kappa^+$.
    Let $H$ be a $\kappa$-monoid and $X$ a $\lambda^-$-submonoid of $H$.
    Then $H$ is \defit{$\lambda^-$-braided over $X$} if $H = \langle X \rangle_\kappa$, and all families $(x_i)_{i \in \kappa}$, $(y_j)_{j \in \kappa}$ in $X$ with $\sum_{i \in \kappa} x_i=\sum_{j \in \kappa} y_j$ are $\lambda^-$-braided.
  \end{enumerate}
\end{defi}

\begin{figure}
  \begin{center}
    \includegraphics{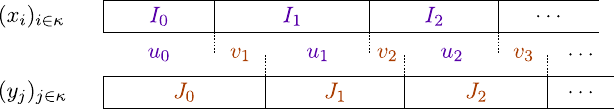}
  \end{center}
  \caption{Illustration of a $\lambda^-$-braiding of two families (\cref{def:braiding}). While the figure only indicates the part of the braiding with finite indices, a general braiding breaks down into a disjoint union of such countable ones (\ref{coll:layers} and \ref{coll:layers-rev} of \cref{l:braiding-collapse}).
    We may think of the sets $I_i$ and $J_j$ as intervals, as the order of the elements in the family does not matter (\cref{l:braiding-basic}).}
\end{figure}

With notation as in the definition, we call $(I_{\mu})_{\mu \in \kappa}$ and $(J_\mu)_{\mu \in \kappa}$ the corresponding \defit{braiding partitions}, and $(u_\mu)_{\mu \in \kappa}$ and $(v_\mu)_{\mu \in \kappa}$ the \defit{braiding families}.
We will later see that the notion does not depend on the choice of well-order on $\kappa$ (\cref{l:braided-no-order}), so it is sensible to index the families by cardinals and not by ordinals.
This is also the reason for suppressing the well-order in the notation.
After establishing the independence from the well-order, it will also make sense to index the families by arbitrary sets of cardinality $\kappa$ and not just by $\kappa$ itself.
With some more work, we will see that braiding is an equivalence relation on families indexed by $\kappa$ (\cref{l:braiding-transitive}).

The case $\lambda=\aleph_0$, in which the sets $I_\mu$ and $J_\mu$ are finite, will turn out to be the most interesting one (due to \ref{coll:uncountable} of \cref{l:braiding-collapse}). We therefore say that two families are \defit{braided} if they are $\aleph_0^-$-braided, and similarly that $H$ is \defit{braided over $X$} if it is $\aleph_0^-$-braided over $X$.

A telescoping argument shows that $\lambda^-$-braided families have equal $\kappa$-sums.

\begin{lemma} \label{l:telescope}
  Let $H$ be a $\kappa$-monoid and $\lambda \le \kappa^+$.
  If $(x_i)_{i \in \kappa}$ and $(y_j)_{j \in \kappa}$ are families in $H$ that are $\lambda^-$-braided, then
  \[
    \sum_{i \in \kappa} x_i = \sum_{j \in \kappa} y_j.
  \]
\end{lemma}

\begin{proof}
  Let $(I_\mu)_{\mu \in \kappa}$, $(J_\mu)_{\mu \in \kappa}$ be braiding partitions, and $(u_\mu)_{\mu \in \kappa}$, $(v_\mu)_{\mu \in \kappa}$ be braiding families.
  Then
  \[
    \begin{split}
      \sum_{i \in \kappa} x_i &= \sum_{\mu \in \kappa} \sum_{i \in I_\mu} x_i = \sum_{\mu \in \kappa} u_\mu + v_\mu = \sum_{\substack{\mu \in \kappa \\ \text{$\mu$ a limit}}} v_\mu + \sum_{\mu \in \kappa} v_{\mu+1} + u_\mu \\
      &=
      \sum_{\mu \in \kappa} v_{\mu+1} + u_\mu = \sum_{\mu \in \kappa} \sum_{j \in J_\mu} y_j = \sum_{j \in \kappa} y_j. \qedhere
    \end{split}
  \]
\end{proof}

So $\lambda^-$-braided families have equal $\kappa$-sums, and if $H$ is $\lambda^-$-braided over $X$, the converse holds.
We illustrate the braiding relation in some examples.

\begin{examples} \label{e:kappa-braided} \mbox{}
  \begin{enumerate}
  \item \label{e:kappa-braided:N} Consider the monoid $(\bN_0,+)$.
    Two countable families are $\aleph_0^-$-braided if they both have finite support and the same sum, or if both of them have infinite support.
    The only non-trivial part is to observe that two families $(a_i)_{i \ge 0}$ and $(b_i)_{i \ge 0}$ with infinite supports are braided.
    But given any finite subsum of one family, say, $a_m+ \cdots + a_n$, we can always find $k \le l$ with $a_m + \cdots + a_n \le b_k + \cdots + b_l$, because $b_i \ge 1$ for infinitely many indices $i$.
    In this way we can  inductively construct a braiding of the two families.
    It follows that the $\aleph_0$-monoid $(\bN_0 \cup \{\infty\}, +)$ with operation as in \ref{e-kappa:trivial} of \cref{e:kappa} is $\aleph_0^-$-braided over the monoid $\bN_0$.

  \item \label{e:kappa-braided:R} Now consider $(\bR_{\ge 0},+)$.
    It is not hard to check that two countable families are $\aleph_0^-$-braided if they have the same sums in the sense of convergent series (which may be $\infty$ if the series diverge) and if either both have finite support or both have infinite support.
    (To see that two countable families with infinite support and the same limit are braided, note that every finite subsum of one series must be strictly dominated by a finite subsum of the second one. An inductive argument then gives the braiding.)
    For each $a \in \bR_{>0}$ we can find series representations with finite support and with infinite support.
    It follows that neither of the two $\aleph_0$-monoid structures on $\bR_{\ge 0} \cup \{\infty\}$ introduced in \cref{e:kappa} \ref{e-kappa:trivial} and \ref{e-kappa:reals} is braided over $\bR_{\ge 0}$.

    We can construct an $\aleph_0$-monoid that is braided over $\bR_{\ge 0}$ as follows.
    Let $\widetilde{\bR}_{>0}$ denote a disjoint copy of $\bR_{>0}$ in which we denote the element corresponding to $a \in \bR_{>0}$ by $\widetilde a$.
    Let
    \[
      H \coloneqq \bR_{\ge 0} \cup \widetilde{\bR}_{>0} \cup \{\infty\}.
    \]
    Note that there is only one copy of $0$ and of $\infty$.
    For $(a_i)_{i \ge 0} \in H^{\aleph_0}$, let $a=\sum_{i=0}^\infty a_i \in \bR_{\ge 0} \cup \{ \infty \}$ be the sum of the convergent series of real numbers (disregarding whether the elements of the sequence live in $\bR_{\ge 0}$ or in $\widetilde \bR_{>0}$).
    If $a=\infty$, we define the $\aleph_0$-sum to be $\sum_{i \in \aleph_0} a_i \coloneqq \infty$.
    Suppose $a$ is finite.
    If all $a_i$ are in $\bR_{\ge 0}$ and the family has finite support, we set $\sum_{i \in \aleph_0} a_i \coloneqq a \in \bR_{\ge 0}$.
    Otherwise, that is, if the family has infinite support or one $a_i$ is in $\widetilde{\bR}_{> 0}$, we set $\sum_{i \in \aleph_0} a_i \coloneqq \widetilde{a}$.
    Then $H$ is $\aleph_0^-$-braided over $\bR_{\ge 0}$, as two families in $\bR_{\ge 0}$ have the same $\aleph_0$-sum in $H$ if and only if they are braided.
    However $H$ is \emph{not} $\aleph_0^-$-braided over itself, as we can repeat the same argument that showed $\bR_{\ge 0} \cup \{\infty\}$ is not $\aleph_0^-$-braided over $\bR_{\ge 0}$ with the elements in $\{0\} \cup \widetilde{\bR}_{>0}$.

  \item \label{e:kappa-braided:Q} We can repeat the construction in \ref{e:kappa-braided:R} with $(\bQ_{\ge 0}, +)$, and end up with
    \[
      \bQ_{\ge 0} \cup \widetilde{\bR}_{>0} \cup \{\infty\}.
    \]
    Elements in $\bR_{>0} \setminus \bQ$ can only be represented as convergent series with infinite support.
    Accordingly we only get one copy of these irrational numbers.
  \end{enumerate}
\end{examples}

In \cref{e:kappa-braided} we always considered $\lambda=\aleph_0$.
Indeed, we now observe that the braiding property collapses to a vastly simpler one if $\lambda \ne \aleph_0$ and becomes trivial in some other edge-cases.

\begin{lemma} \label{l:braiding-collapse}
  Let $X$ be a $\lambda^-$-monoid, let $(x_i)_{i \in \kappa}$, $(y_j)_{j \in \kappa}$ be two families in $X$, and fix a limit well-order on $\kappa$.
  \begin{enumerate}
  \item \label{coll:trivial} If $(x_i)_{i \in \kappa}$ and $(y_j)_{j \in \kappa}$ have support of cardinality strictly less than $\lambda$, and $\sum_{i \in \kappa} x_i = \sum_{i \in \kappa} y_i$, then the families are $\lambda^-$-braided.
  \item \label{coll:layers} Suppose that $(x_i)_{i \in \kappa}$ and $(y_j)_{j \in \kappa}$ are $\lambda^-$-braided with braiding partitions $(I_\mu)_{\mu \in \kappa}$ and $(J_\mu)_{\mu \in \kappa}$.
    Let $L \subseteq \kappa$ denote the set of limit elements.
    For $l \in L$ define $I(l)\coloneqq \bigcup_{m \in \bN_0} I_{l+m}$ and $J(l)\coloneqq \bigcup_{m \in \bN_0} J_{l+m}$.
    Then, for all $l \in L$, the families $(x_{i})_{i \in I(l)}$ and $(y_j)_{j \in J(l)}$, padded by zeroes to an index set of cardinality $\kappa$, are $\lambda^-$-braided.
  \item \label{coll:layers-rev} Suppose that there are indexed partitions $(I(l))_{l \in \kappa}$ and $(J(l))_{l \in \kappa}$ of $\kappa$ such that the families $(x_i)_{i \in I(l)}$ and $(y_j)_{j \in J(l)}$ are $\lambda^-$-braided for all $l \in \kappa$
    \textup{(}with $\card{I(l)}=\card{J(l)}$ infinite for all $l \in \kappa$\textup).
     Then there exists a limit well-order on $\kappa$ such that $(x_i)_{i \in \kappa}$ and $(y_j)_{j \in \kappa}$ are $\lambda^-$-braided.

   \item \label{coll:uncountable} If $\lambda \ne \aleph_0$, then $(x_i)_{i \in \kappa}$ and $(y_j)_{j \in \kappa}$ are $\lambda^-$-braided if and only if there exist indexed partitions $(I_\mu)_{\mu \in \kappa}$ and $(J_{\mu})_{\mu \in \kappa}$ of $\kappa$, such that $\card{I_\mu}$, $\card{J_\mu} < \lambda$, and
    \[
      \sum_{i \in I_\mu} x_i = \sum_{j \in J_\mu} y_j \qquad\text{for all $\mu \in \kappa$.}
    \]
  \end{enumerate}
\end{lemma}

\begin{proof}
  \begin{proofenumerate}
  \item[\ref{coll:trivial}]
  First observe that, due to the restriction on the support, the expressions $u_0\coloneqq \sum_{i \in \kappa} x_i=\sum_{j \in \kappa} y_j$ can actually be interpreted as $\lambda^-$-sums and are therefore well-defined.
  Now let $I_0=J_0$ be a set of cardinality strictly less than $\lambda$ containing the supports of both families.
  Set $u_\mu \coloneqq 0$ for all $\mu \in \kappa \setminus\{0\}$, set $v_\mu \coloneqq 0$ for all $\mu \in\kappa$, and distribute the remaining indices in an arbitrary fashion among the $I_\mu$ and $J_\mu$ in such a way that the cardinality restriction $\card{I_\mu}$,~$\card{J_\mu} < \lambda$ is satisfied.
  Then the braiding property of \cref{def:braiding} is satisfied.

  \item[\ref{coll:layers}]
    We obtain a braiding of $(x_{\mu})_{\mu \in I(l)}$ and $(y_{\mu})_{\mu \in J(l)}$ from the corresponding subfamilies $(I_{l+m})_{m \in \bN_0}$ and $(J_{l+m})_{m \in \bN_0}$ of the braiding partitions and $(u_{l+m})_{m \in \bN_0}$ and $(v_{l+m})_{m \in \bN_0}$ of the braiding families.

  \item[\ref{coll:layers-rev}]
    We have $\kappa = \bigcup_{l \in \kappa} I(l) = \bigcup_{l \in \kappa} J(l)$, with each of the unions being disjoint.
    Let $\kappa_l \coloneqq \card{I(l)} = \card{J(l)} \le \kappa$.
    For $l \in \kappa$ and the subfamilies $(x_i)_{i \in I(l)}$ and $(y_j)_{j \in J(l)}$, let $(I_{l,\mu})_{\mu \in \kappa_l}$, $(J_{l,\mu})_{\mu \in \kappa_l}$ be the corresponding braiding partitions and $(u_{l,\mu})_{\mu \in \kappa_l}$,~$(v_{l,\mu})_{\mu \in \kappa_l}$ the braiding families.
    The braiding of $(x_i)_{i \in I(l)}$ and $(y_j)_{j \in J(l)}$ involves the choice of a limit well-order on $\kappa_l$.
    Consider the set $\Omega \coloneqq \{\, (l,j) \in \kappa \times \kappa : j \in \kappa_l \,\}$.
    It is well-ordered by $(l,j) < (l',j')$ if and only if $l<l'$ or $l=l'$ and $j < j'$.
    Let $\pi \colon \kappa \to \Omega$ be a bijection.
    Transferring the  well-order of $\Omega$ to $\kappa$ via $\pi$, and setting $I'_\mu \coloneqq I_{\pi(\mu)}$, $J'_\mu \coloneqq J_{\pi(\mu)}$, and $u'_\mu \coloneqq u_{\pi(\mu)}$, $v'_\mu \coloneqq v_{\pi(\mu)}$, shows that $(x_i)_{i \in \kappa}$ and $(y_j)_{j \in \kappa}$ are $\lambda^-$-braided.
  
  \item[\ref{coll:uncountable}]
  It is clear that two families satisfying the property in \ref{coll:uncountable} are braided over $X$ by choosing $v_\mu\coloneqq 0$  and $u_\mu \coloneqq \sum_{i \in I_\mu} x_i$ for all $\mu \in \kappa$.

  For the converse, suppose $(x_i)_{i \in \kappa}$ and $(y_j)_{j \in \kappa}$ are braided over $X$ with braiding partitions $(I_\mu)_{\mu \in \kappa}$, $(J_\mu)_{\mu \in \kappa}$ and braiding families $(u_\mu)_{\mu \in \kappa}$, $(v_\mu)_{\mu \in \kappa}$.
  For a limit element $\mu \in \kappa$, set $I_\mu' \coloneqq \bigcup_{n \in \bN_0} I_{\mu+n}$ and $J_\mu' \coloneqq \bigcup_{n \in \bN_0} J_{\mu+n}$; for all other $\mu$ simply set $I_\mu'\coloneqq J_\mu'\coloneqq \emptyset$.
  Because $\lambda$ is uncountable, we still have $\card{I_\mu'}$, $\card{J_\mu'} < \lambda$, and now
  \[
    \sum_{i \in I_\mu'} x_i = \sum_{j \in J_\mu'} y_j \qquad\text{for all $\mu \in \kappa$},
  \]
  by \ref{coll:layers} and the telescoping argument (\cref{l:telescope}). \qedhere
  \end{proofenumerate}
\end{proof}

We show that \cref{def:braiding} is actually independent of the choice of well-order.

\begin{lemma} \label{l:braided-no-order}
  Being $\lambda^-$-braided does not depend on the choice of limit well-order on $\kappa$.
\end{lemma}

\begin{proof}
  We can assume $\lambda = \aleph_0$, as otherwise the braiding property collapses to the one in \ref{coll:uncountable} of \cref{l:braiding-collapse}, which obviously does not depend on the chosen well-order on $\kappa$.

  Suppose first that $\kappa$ is uncountable.
  To be explicit about the chosen well-order on $\kappa$, let us fix two order-isomorphisms $\kappa \cong \alpha$ and $\kappa\cong \beta$ with limit ordinals $\alpha$ and $\beta$.
  Let $X$ be a $\lambda^-$-monoid and let $(x_i)_{i \in \kappa}$, $(y_j)_{j \in \kappa}$ be $\lambda^-$-braided families in $X$, with respect to the well-order $\alpha$.
  (Recall, from \cref{def:braiding}, that the role of the well-order is to determine the meaning of the successor $\mu+1$ in the term $v_{\mu+1}$ of the braiding relations, as well as which elements of $\kappa$ are considered to be limit elements.)

  Let $L_{\alpha}$ and $L_{\beta}$ denote the set of limit elements in $\kappa$ under the respective well-order.
  Then $\card{L_\alpha} = \card{L_\beta} = \kappa$.
  Consequently, there is a bijection $f \colon \kappa \to \kappa$ that maps $L_\alpha$ bijectively to $L_\beta$ and for which $f(\mu+_\alpha 1) = f(\mu)+_\beta 1$ for all $\mu \in \kappa$.
  If $(I_\mu)_{\mu \in \kappa}$, $(J_\mu)_{\mu \in \kappa}$ are braiding partitions and $(u_\mu)_{\mu \in \kappa}$, $(v_\mu)_{\mu \in \kappa}$ are braiding families with respect to the well-order $\alpha$, then $(I_{f^{-1}(\mu)})_{\mu \in \kappa}$, $(J_{f^{-1}(\mu)})_{\mu \in \kappa}$ and $(u_{f^{-1}(\mu)})_{\mu \in \kappa}$, $(v_{f^{-1}(\mu)})_{\mu \in \kappa}$ give a braiding with respect to the well-order $\beta$.

  We now consider the case $\kappa=\aleph_0$.
  It suffices to consider order isomorphisms $\kappa \cong \omega$ with the usual well-order on $\omega$ and $\kappa \cong \alpha$ with an arbitrary countable limit ordinal $\alpha$.
  Since $\omega \subseteq \alpha$, it is easy to see a braiding with respect to $\omega$ gives a braiding with respect to $\alpha$, by setting the additional sets of the braiding partitions to the empty set, and the additional elements of the braiding families to $0$.

  Conversely, suppose that $(x_i)_{i \in \kappa}$, $(y_j)_{j \in \kappa}$ are $\aleph_0^-$-braided with respect to the well-order $\alpha$, with braiding partitions $(I_\mu)_{\mu \in \kappa}$, $(J_\mu)_{\mu \in \kappa}$ and braiding families $(u_\mu)_{\mu \in \kappa}$, $(v_\mu)_{\mu \in \kappa}$.
  We can now use a diagonal argument.
  Let $L=\{l_1, l_2, \ldots\} \subseteq \kappa$ denote the subset of all limit elements of $\kappa$ with respect to the well-order $\alpha$.
  Because $\alpha$ is countable, the set $L$ is countable or finite.
  We define
  \[
    I_n' \coloneqq \bigcup_{i+j=n} I_{l_i +_\alpha\, j}, \quad J_n' \coloneqq \bigcup_{i+j=n} J_{l_i+_\alpha\, j}, \quad u_n' \coloneqq \sum_{i+j=n} u_{l_i+_\alpha\, j}, \quad \text{and} \quad v_n'\coloneqq\sum_{i+j=n} v_{l_i+_\alpha\, j},
  \]
  for $n \in \omega$ and where $i \ge 1$ and $j \ge 0$ (if $L$ is finite, we tacitly assume that the non-defined terms are omitted from the unions, respectively, sums).
  Keeping in mind $v_{l_i}=0$, it is routine to verify that this gives a braiding with respect to $\omega$.
\end{proof}

We now observe some basic properties of the $\lambda^-$-braidedness relation.

\begin{lemma} \label{l:braiding-basic}
  Let $X$ be a $\lambda^-$-monoid.
  \begin{enumerate}
  \item \label{b-basic:commutative} If $(x_\mu)_{\mu \in \kappa}$ is a family in $X$ and $\pi\colon \kappa\to\kappa$ is a bijection, then $(x_\mu)_{\mu \in \kappa}$ is braided to $(x_{\pi(\mu)})_{\mu \in \kappa}$.
  \item \label{b-basic:refsym} Being $\lambda^-$-braided is a reflexive and symmetric relation on families indexed by $\kappa$.
  \end{enumerate}
\end{lemma}

\begin{proof}
  \begin{proofenumerate}
  \item[\ref{b-basic:commutative}]
  Define $I_\mu \coloneqq \{\pi(\mu)\}$, $J_\mu \coloneqq \{\mu\}$ and set $u_\mu \coloneqq x_{\pi(\mu)}$ and $v_\mu\coloneqq 0$ for all $\mu \in \kappa$.
  This defines a braiding of $(x_\mu)_{\mu \in \kappa}$ and $(x_{\pi(\mu)})_{\mu \in \kappa}$.

  \item[\ref{b-basic:refsym}]
  Reflexivity follows by applying \ref{b-basic:commutative} with $\pi$ the identity map.
  Symmetry is a matter of shifting the indices: suppose $(x_i)_{i \in \kappa}$ and $(y_j)_{j \in \kappa}$ are braided over $X$ with braiding partitions $(I_\mu)_{\mu \in \kappa}$, $(J_\mu)_{\mu \in \kappa}$ and braiding families $(u_\mu)_{\mu \in \kappa}$, $(v_\mu)_{\mu \in \kappa}$.
  For limit elements $\mu \in \kappa$, we set $u_\mu' = u_\mu + v_{\mu+1}$, $v_\mu'=0$, and $I_\mu' = J_\mu$, $J_\mu'=I_\mu \cup I_{\mu+1}$.
  For non-limit elements $\mu$ we set $u_{\mu}'\coloneqq v_{\mu+1}$, $v_\mu'=u_\mu$ and $I_\mu'=J_\mu$, $J_\mu'=I_{\mu+1}$.
  For $\mu$ a limit element, we then have
  \[
    \sum_{i \in I_\mu'} y_i = v_{\mu+1}+u_\mu= v_\mu' + u_\mu' \quad\text{and}\quad \sum_{j \in J_\mu'} x_j = v_{\mu} + u_\mu + v_{\mu+1} + u_{\mu+1} = u_\mu' + v_{\mu+1}'.
  \]
  Similarly, for a non-limit $\mu \in \kappa$,
  \[
    \sum_{i \in I_\mu'} y_i = v_{\mu+1} + u_{\mu} = v_\mu' + u_\mu' \quad\text{and}\quad \sum_{j \in J_\mu'} x_j = v_{\mu+1} + u_{\mu+1} = v_{\mu+1}' + u_\mu'.
  \]
  Thus $(y_j)_{j \in \kappa}$ and $(x_i)_{i \in \kappa}$ are also braided over $X$. \qedhere
  \end{proofenumerate}
\end{proof}

While the proof of the transitivity of the braiding relation is not especially difficult, it is quite technical to state.
The main reason for that is the need to match up two possibly different indexed partitions for one family.
The following preliminary lemma is of help in this step, and encapsulates most of the technicalities.

\begin{lemma} \label{l:braid-repartition}
  Let $X$ be a $\lambda^-$-monoid.
  Fix a limit well-order on $\kappa$.
  Let $(x_i)_{i \in \kappa}$, $(y_j)_{j \in \kappa}$, and $(z_k)_{k \in \kappa}$ be families in $X$ such that $(x_i)_{i \in \kappa}$ and $(y_j)_{j \in \kappa}$ are $\lambda^-$-braided and $(y_j)_{j \in \kappa}$ and $(z_k)_{k \in \kappa}$ are $\lambda^-$-braided.
  Then there exist $\lambda^-$-braidings of $(x_i)_{i \in \kappa}$ and $(y_j)_{j \in \kappa}$ with braiding partitions $(I_\mu)_{\mu \in \kappa}$, $(J_\mu)_{\mu \in \kappa}$, and $\lambda^-$-braidings of  $(y_j)_{j \in \kappa}$ and $(z_k)_{k \in \kappa}$ with braiding partitions $(J_\mu')_{\mu \in \kappa}$, $(K_\mu')_{\mu \in \kappa}$ such that
  \[
    \bigcup_{\nu \le \mu} J_\nu \subseteq \bigcup_{\nu \le \mu} J_\nu' \subseteq \bigcup_{\nu \le \mu+1} J_{\nu} \qquad\text{for } \mu \in \kappa.
  \]
\end{lemma}

\begin{proof}
  It will be convenient to have some additional terminology available.

  We call a subset $S \subseteq \kappa$ a \defit{$\lambda^-$-interval} if it is of the form
  \[
    S = \bigcup_{l \in L} \{\, l + i : n_l \le i < m_l \,\}
  \]
  for a set of limit elements $L \subseteq \kappa$ with $\card{L}< \lambda$, and non-negative integers $n_l < m_l$.
  If $\lambda > \aleph_0$, we also permit $m_l = \aleph_0$.
  Thus, if $\lambda = \aleph_0$, for every limit element $l \in L$, the set $S$ contains a finite (possibly empty) interval of $l + \bN_0$.
  If $\lambda > \aleph_0$, then $S$ may also contain an infinite interval of $l + \bN_0$.

  The set of \defit{left endpoints} of $S$ is $\lep(S) \coloneqq \{\, l + n_l : l \in L \,\}$.
  The set of \defit{right endpoints} of $S$ is $\rep(S) \coloneqq \{\, l + m_l - 1 : l \in L,\ m_l < \aleph_0 \,\}$ and we write $\rep(S)^+$ for the set of successors of the right endpoints.

  An interval $S$ is \defit{left saturated} in $\kappa \setminus T$ (with $T$ an arbitrary subset) if, whenever $s \in S$ is a successor and $s-1 \not\in T$, then $s-1 \in S$.
  
  Let $(L_\mu)_{\mu \in \kappa}$, $(M_\mu)_{\mu \in \kappa}$ be braiding partitions and $(a_\mu)_{\mu \in \kappa}$, $(b_\mu)_{\mu \in \kappa}$ be braiding families for a $\lambda^-$-braiding of $(x_i)_{i \in \kappa}$ and $(y_j)_{j \in \kappa}$.
  Similarly, let $(M_\mu')_{\mu \in \kappa}$, $(N_\mu)_{\mu \in \kappa}$ be braiding partitions and $(c_\mu)_{\mu \in \kappa}$, $(d_\mu)_{\mu \in \kappa}$ braiding families for a $\lambda^-$-braiding of $(y_j)_{j \in \kappa}$ and $(z_k)_{k \in \kappa}$.
    We construct the desired braidings by transfinite recursion.
    Specifically, for $\mu \in \kappa$ we construct $\lambda^-$-intervals $\cA_\mu$, $\cB_\mu$, and $u_\mu$,~$v_\mu$, $v_{\mu+1}$, $g_\mu$, $h_{\mu}$, $h_{\mu+1} \in X$, such that the following properties all hold.
    \begin{enumerate}
    \item If $\mu$ is a limit element, then $v_\mu=h_\mu=0$.
    \item \label{tr:eqn-xy} With $I_\mu \coloneqq \bigcup_{\nu \in \cA_\mu} L_\nu$ and $J_\mu \coloneqq \bigcup_{\nu \in \cA_\mu} M_\nu$,
      \[
        \sum_{i \in I_\mu} x_i = u_\mu + v_\mu \qquad\text{and}\qquad \sum_{j \in J_\mu} y_j = u_\mu + v_{\mu+1}.
      \]
    \item \label{tr:eqn-yz} With $J'_\mu \coloneqq \bigcup_{\nu \in \cB_\mu} M'_\nu$ and $K_\mu \coloneqq \bigcup_{\nu \in \cB_\mu} N_\nu$,
      \[
        \sum_{j \in J_\mu'} y_j = g_\mu + h_\mu \qquad\text{and}\qquad \sum_{k \in K_\mu} z_k = g_\mu + h_{\mu+1}.
      \]
    \item It holds that $\bigcup_{\nu \le \mu} J_\nu \subseteq \bigcup_{\nu \le \mu} J_\nu'$ and $\bigcup_{\nu < \mu} J_{\nu}' \subseteq \bigcup_{\nu \le \mu} J_\nu$.
    \item The $\lambda^-$-interval $\cA_\mu$ is left saturated in $\kappa \setminus \bigcup_{\nu < \mu} \cA_\nu$, and analogously for $\cB_\mu$.
    \item \label{tr:right-saturated} If $\beta \in \kappa$ is such that $\beta + n < \mu$ for all $n \in \bN_0$, then $\bigcup_{n \in \bN_0} \cA_{\beta+n}$ is closed under successors.
      The analogous statement holds for $\cB_\beta$.
    \item \label{tr:formula-uv} $u_\mu = \sum_{\nu \in \cA_\mu} a_\nu + \sum_{\nu \in \cA_\mu \setminus \lep(\cA_\mu)} b_\nu$ and $v_{\mu + 1} = \sum_{\nu \in \rep(\cA_\mu)^+} b_\nu$;
    \item \label{tr:formula-gh} $g_\mu = \sum_{\nu \in \cB_\mu} c_\nu + \sum_{\nu \in \cB_\mu \setminus \lep(\cB_\mu)} d_\nu$ and $h_{\mu + 1} = \sum_{\nu \in \rep(\cB_\mu)^+} d_\nu$.
    \end{enumerate}

    Let $\alpha \in \kappa$ and suppose that these properties hold for all $\mu < \alpha$.
    We now construct the sets $\cA_\alpha$,~$\cB_\alpha$ and elements $u_\alpha$, $v_{\alpha+1}$, $g_\alpha$, $h_{\alpha+1}$ so that these properties remain true for $\mu=\alpha$.
    Note that $v_\alpha$ and $h_\alpha$ are already defined (either by recursion or by $v_\alpha=h_\alpha=0$ if $\alpha$ is a limit element).
    If $\alpha$ is a limit element, by a slight abuse of notation, let $J'_{\alpha-1}=J_{\alpha-1}=\emptyset$ for notational convenience.
    Because $\card{J_{\alpha-1}'} < \lambda$, there exists a set $\cA_\alpha$ with $\card{\cA_\alpha} < \lambda$ and such that we can cover
    \[
      J_{\alpha-1}' \setminus \bigcup_{\nu < \alpha} J_\nu \subseteq \bigcup_{\nu \in \cA_\alpha} M_{\nu}.
    \]
    By \ref{tr:eqn-xy} we can take $\cA_\alpha \subseteq \kappa \setminus \bigcup_{\nu < \alpha} \cA_\nu$.
    Growing $\cA_\alpha$ if necessary, we may also assume that it contains the minimum of $\kappa \setminus \bigcup_{\nu < \alpha} \cA_\nu$ if this set is nonempty.
    If $\alpha$ is a successor, we can further assume that for every $\nu \in \rep(\cA_{\alpha-1})$ we have $\nu+1 \in \cA_\alpha$.
    Finally, without restriction, we can take $\cA_\alpha$ to be a $\lambda^-$-interval that is left saturated in $\kappa \setminus \bigcup_{\nu < \alpha} \cA_\nu$.

    Now define $u_\alpha$, $v_{\alpha+1}$ according to \ref{tr:formula-uv}.
    We verify \ref{tr:eqn-xy}.
    Regularity of $\lambda$ and $\card{\cA_\alpha} < \lambda$ imply $\card{I_\alpha}$, $\card{J_\alpha} < \lambda$.
    Suppose $\nu \in \lep(\cA_\alpha)$.
    Then either $\nu$ is a limit element, in which case $b_\nu = 0$, or $\alpha$ must be a successor and $\nu-1 \in \cA_{\alpha-1}$ (this follows from \ref{tr:right-saturated} and the left saturated choice of $\cA_\alpha$).
    In the second case $\nu-1 \in \rep(\cA_{\alpha-1})$.
    Thus $v_\alpha = \sum_{\nu \in \lep(\cA_\alpha)} b_\nu$, and hence
    \[
      \sum_{i \in I_\alpha} x_i = \sum_{\nu \in \cA_\alpha} a_\nu + b_\nu = u_\alpha + v_\alpha.
    \]
    Similarly
    \[
      \sum_{i \in J_\alpha} y_i = \sum_{\nu \in \cA_\alpha} a_\nu + b_{\nu+1} = u_\alpha + v_{\alpha+1}.
    \]
    
    Observe that \ref{tr:right-saturated} holds because we ensured $\nu + 1 \in \cA_\alpha$ for every $\nu \in \rep(\cA_{\alpha-1})$ whenever $\alpha$ is a successor.

    The construction of $\cB_\alpha$ is analogous.
    Define $\cB_\alpha$ in such a way that
    \[
      J_\alpha \setminus \bigcup_{\nu < \alpha} J_\nu' \subseteq \bigcup_{\nu \in \cB_\alpha} M_\nu'.
    \]
    Again we can assume that $\cB_\alpha$ is a $\lambda^-$-interval, left saturated in $\kappa \setminus \bigcup_{\nu < \alpha} \cB_\nu$.
    We can also assume that $\cB_\alpha$ contains the minimum of $\kappa \setminus \bigcup_{\nu < \alpha} \cB_\nu$ (if this set is nonempty), and, if $\alpha$ is a successor, that $\nu + 1 \in \cB_{\alpha}$ for all $\nu \in \rep(\cB_{\alpha-1})$.
    One verifies the claimed properties in the same way as for $\cA_\alpha$.

    It only remains to observe that the constructed braiding partitions indeed exhaust the set $\kappa$.
    For this it is sufficient for the families $(\cA_\mu)_{\mu \in \kappa}$ and $(\cB_\mu)_{\mu \in \kappa}$ to exhaust $\kappa$.
    This is ensured by adding the minimum of the complement in each step, if this complement is nonempty.
\end{proof}

\begin{lemma} \label{l:braiding-transitive}
  Let $X$ be a $\lambda^-$-monoid.
  The $\lambda^-$-braiding relation is transitive on families indexed by $\kappa$, and hence it is an equivalence relation.
\end{lemma}

\begin{proof}
  Let $(x_i)_{i \in \kappa}$ and $(y_j)_{j \in \kappa}$ be braided by braiding partitions $(I_\mu)_{\mu \in \kappa}$, $(J_\mu)_{\mu \in \kappa}$ and braiding families $(u_\mu)_{\mu \in \kappa}$, $(v_\mu)_{\mu \in \kappa}$, and let $(y_j)_{j \in \kappa}$ and $(z_k)_{k \in \kappa}$ be braided by braiding partitions $(J_\mu')_{\mu \in \kappa}$, $(K_\mu)_{\mu \in \kappa}$ and braiding families $(a_\mu)_{\mu \in \kappa}$, $(b_\mu)_{\mu \in \kappa}$.
  By \cref{l:braid-repartition} we may assume
  \begin{equation} \label{eq:jinc}
    \bigcup_{\nu \le \mu} J_\nu \subseteq \bigcup_{\nu \le \mu} J_\nu' \subseteq \bigcup_{\nu \le \mu+1} J_{\nu} \qquad\text{for } \mu \in \kappa.
  \end{equation}
  Bearing in mind that $J_\mu' \cap (\bigcup_{\nu < \mu} J_\nu') = \emptyset$, \cref{eq:jinc} implies $J_\mu' \subseteq J_\mu \cup J_{\mu+1}$.
  Because \cref{eq:jinc} also readily implies $\bigcup_{\nu < \mu} J_\nu' \subseteq \bigcup_{\nu \le \mu} J_\nu$, we similarly obtain $J_{\mu+1} \subseteq J'_{\mu} \cup J'_{\mu+1}$.
  For $\mu$ a limit element, we get $\bigcup_{\nu < \mu} J_\nu  =\bigcup_{\nu < \mu} J_{\nu}'$ from \cref{eq:jinc}, and so even $J_\mu \subseteq J_\mu'$.
  Set
  \[
    s_\mu \coloneqq \sum_{j \in J'_\mu \setminus J_\mu} y_j = \sum_{j \in J_{\mu+1} \setminus J_{\mu+1}'} y_j \qquad\text{and}\qquad t_\mu \coloneqq \sum_{j \in J_{\mu+1} \setminus J_\mu'} y_j = \sum_{j \in J_{\mu+1}'\setminus J_{\mu+2}} y_j.
  \]
  Then
  \[
    \sum_{j \in J_{\mu+1} \cup J_{\mu + 2}} y_j = \sum_{j \in J'_{\mu+1}} y_j + s_\mu + t_{\mu+1}, \quad\text{and}\quad
    \sum_{j \in J'_{\mu+2} \cup J'_{\mu + 3}} y_j = \sum_{j \in J_{\mu + 3}} y_j + s_{\mu+3} + t_{\mu+1},
  \]
  and thus
  \begin{align}
    u_{\mu+1}+v_{\mu+2}+u_{\mu+2}+v_{\mu+3} &= a_{\mu+1} + b_{\mu+1} + s_\mu + t_{\mu + 1}, \label{eq:subs1} \\
    a_{\mu+2} + b_{\mu+2} + a_{\mu+3} + b_{\mu + 3} &= u_{\mu+3} + v_{\mu+4} + s_{\mu+3} + t_{\mu+1} \label{eq:subs2}.
  \end{align}

  For every limit element $\mu$ and every $l \in \bN_0$, we define $M_\mu\coloneqq I_\mu$, $N_\mu \coloneqq K_\mu$ and $M_{\mu + l +1} \coloneqq I_{\mu+3l+1} \cup I_{\mu + 3l+2} \cup I_{\mu + 3l+3}$ and $N_{\mu+l+1} \coloneqq K_{\mu + 3l+1} \cup K_{\mu + 3l+2} \cup K_{\mu + 3l+3}$.
  Set
  \begin{align*}
    c_\mu &\coloneqq u_\mu,& c_{\mu+l+1} &\coloneqq a_{\mu+3l+1} + t_{\mu+3l+1} + u_{\mu+3l+3}, & &  \\
    d_\mu &\coloneqq 0,& d_{\mu + l +1} &\coloneqq s_{\mu+3l} + b_{\mu +3l+1} + v_{\mu+3l+1}  & &\text{(for $l \ge 0$)}.
  \end{align*}
  We claim that this gives a $\lambda^-$-braiding of $(x_i)_{i \in \kappa}$ and $(z_k)_{k \in \kappa}$.
  Indeed for $\mu$ a limit element trivially $\sum_{m \in M_\mu} x_m = u_\mu$, and (keeping in mind $J_\mu \subseteq J_\mu'$)
  \[
    \begin{split}
      \sum_{n \in N_\mu} z_n &= a_\mu + b_{\mu+1} = \sum_{j \in J'_\mu} y_j + b_{\mu+1} = \sum_{j \in J_\mu} y_j + s_\mu + b_{\mu+1} \\
      &= u_\mu + v_{\mu+1} + s_\mu + b_{\mu + 1} = c_\mu + d_{\mu+1}.
    \end{split}
  \]
  If $l \ge 0$, then, using \cref{eq:subs1},
  \[
    \begin{split}
      \sum_{n \in M_{\mu + l + 1}} x_n &= \sum_{i=1}^3 u_{\mu+3l+i} + v_{\mu+3l+i} \\
      &= u_{\mu+3l+3} + v_{\mu+3l+1} + a_{\mu+3l+1} +b_{\mu+3l+1} + s_{\mu + 3l} + t_{\mu+3l+1}  \\
      &= c_{\mu+l+1} + d_{\mu+l+1}.
    \end{split}
  \]
  Similarly, using \cref{eq:subs2},
  \[
    \begin{split}
      \sum_{n \in N_{\mu + l + 1}} z_n &= \sum_{i=1}^3 a_{\mu+3l+i} + b_{\mu+3l+i+1} \\
      &= a_{\mu+3l+1} + b_{\mu+3l+4} + u_{\mu+3l+3} + v_{\mu+3l+4} + s_{\mu + 3l+3} + t_{\mu + 3l + 1} \\
      &= c_{\mu+l+1} + d_{\mu+l+2}.
    \end{split}
  \]
  Thus $(x_i)_{i \in \kappa}$ and $(z_k)_{k \in \kappa}$ are $\lambda^-$-braided.
\end{proof}

\begin{remark}
  In light of \cref{l:braiding-collapse}, the difficult case of transitivity is the case $\lambda = \aleph_0$.
  Indeed, if $\lambda > \aleph_0$, then  \ref{coll:uncountable} of \cref{l:braiding-collapse} allows us to assume $\sum_{i \in I_\mu} x_i = \sum_{j \in J_\mu} y_j$ and $\sum_{j \in J_\mu'} y_j = \sum_{k \in K_\mu} z_k$ for all $\mu \in \kappa$.
  To match up the middle partitions $(J_\mu)_{\mu \in \kappa}$ and $(J_\mu')_{\mu \in \kappa}$, we can now simply replace them by their common coarsening.
  To see that the sets of this common coarsening still have cardinality strictly less than $\lambda$, consider the graph on the vertex set $\kappa$ with edges between $i$ and $j$ if and only if $\{i,j\} \subseteq J_\mu$ or $\{i,j\} \subseteq J_\mu'$ for some $\mu \in \kappa$.
  The sets of the common coarsening are then the connected components of this graph.
  Regularity of the uncountable cardinal $\lambda$ ensures that these connected components still have cardinality strictly less than $\lambda$.
\end{remark}

\subsection{Universal \texorpdfstring{$\kappa$}{κ}-extensions.}

Having shown that the $\lambda^-$-braiding relation is an equivalence relation on $\kappa$-indexed families over a $\lambda^-$-monoid $X$, in this subsection we construct a $\kappa$-overmonoid of $X$ satisfying a natural universal property.
It will turn out that this $\kappa$-overmonoid is equivalently characterized by being $\lambda^-$-braided over $X$.
We first observe a crucial extension property for $\lambda^-$-homomorphisms.

\begin{prop} \label{p:braiding-extend}
  Let $\lambda\le\kappa^+$, let $\overline{H}$ be a $\kappa$-monoid, and $H \subseteq \overline{H}$ a $\lambda^-$-submonoid.
  Suppose $\overline{H}$ is $\lambda^-$-braided over $H$.
  If $\varphi \colon H \to K$ is a $\lambda^-$-homomorphism to a $\kappa$-monoid $K$, then there exists a unique $\kappa$-homomorphism $\overline{\varphi}: \overline{H} \to K$ extending $\varphi$.
\end{prop}

\begin{proof}
  Since $\langle H \rangle_\kappa= \overline{H}$, every $x \in \overline H$ can be represented in the form $x = \sum_{i \in \kappa} x_i$ with $x_i \in H$, and necessarily
  \[
    \overline{\varphi}(x) = \sum_{i \in \kappa} \overline{\varphi}(x_i) = \sum_{i \in \kappa} \varphi(x_i),
  \]
  so the uniqueness of the extension is clear.
  
  To check its existence, suppose $x=\sum_{i \in \kappa} x_i = \sum_{j \in \kappa} y_j$ with $x_i$, $y_j \in H$.
  Then the families $(x_i)_{i\in \kappa}$ and $(y_j)_{j \in \kappa}$ are $\lambda^-$-braided.
  Let $(I_\mu)_{\mu \in \kappa}$, $(J_\mu)_{\mu \in \kappa}$ be braiding partitions, and $(u_\mu)_{\mu \in \kappa}$, $(v_\mu)_{\mu \in \kappa}$ be corresponding braiding families.
  Keep in mind $\card{I_\mu}$,~$\card{J_\mu} < \lambda$.
  Then
  \[
    \begin{split}
    \sum_{i \in \kappa} \varphi(x_i)
    &= \sum_{\mu \in \kappa} \sum_{i \in I_\mu} \varphi(x_i)
    = \sum_{\mu \in \kappa} \varphi\Big( \sum_{i \in I_\mu} x_i\Big)
    = \sum_{\mu \in \kappa} \varphi(v_\mu + u_\mu)
    = \sum_{\mu \in \kappa} \varphi(v_\mu)  + \varphi(u_\mu) \\
    &= \sum_{\mu \in \kappa} \varphi(v_{\mu+1})  + \varphi(u_\mu)
    = \sum_{\mu \in \kappa} \varphi(v_{\mu+1}  + u_\mu)
    = \sum_{\mu \in \kappa} \varphi\Big(\sum_{j \in J_\mu} y_j\Big)  \\
    &= \sum_{\mu \in \kappa} \sum_{j \in J_\mu} \varphi(y_j)
     = \sum_{j \in \kappa} \varphi(y_j).
    \end{split}
  \]
  Thus there is a well-defined map $\overline{\varphi} \colon \overline{H} \to K$ mapping $x=\sum_{i \in \kappa} x_i$ to $\sum_{i \in \kappa} \varphi(x_i)$.
  Then $\overline{\varphi}|_H = \varphi$ is clear.

  We check that $\overline{\varphi}$ is a $\kappa$-homomorphism.
  Let $(x_i)_{i \in \kappa}$ be a family in $\overline{H}$.
  Since $\overline{H} = \langle H \rangle_\kappa$, we can represent each $x_i$ as $x_i = \sum_{j \in\kappa} x_{i,j}$ with $x_{i,j} \in H$.
  Then
  \[
    \overline\varphi\Big(\sum_{i \in \kappa} x_i \Big) = \overline\varphi\Big(\sum_{i \in \kappa} \sum_{j \in \kappa} x_{i,j}\Big) = \overline\varphi\Big(\sum_{(i,j) \in \kappa^2} x_{i,j} \Big) = \sum_{(i,j) \in \kappa^2} \varphi(x_{i,j}) = \sum_{i \in \kappa} \overline{\varphi}(x_i).
  \]
  Here we used \ref{a:flatten} to transform the double sum into a single one. Then, keeping in mind $x_{i,j} \in H$, the definition of $\overline\varphi$ allows us to exchange the $\kappa$-sum and the application of the homomorphism.
\end{proof}

Given a $\lambda^-$-monoid $H$, we are now going to construct a $\kappa$-monoid $\widehat H$ that is a $\lambda^-$-braided $\lambda^-$-overmonoid of $H$.
Because of the previous proposition, this overmonoid is characterized by a universal property.

\begin{defi}[Universal Property] \label{d:univprop}
  Let $\lambda \le \kappa^+$.
  Let $H$ be a $\lambda^-$-monoid.
  A $\kappa$-monoid $\widehat H$ is a \defit{universal $\kappa$-extension} of $H$ if there is a $\lambda^-$-homomorphism $\iota\colon H \to \widehat H$ satisfying: for every $\lambda^-$-homomorphism $\varphi\colon H \to K$ to a $\kappa$-monoid $K$, there exists a unique $\kappa$-homomorphism $\widehat{\varphi}\colon \widehat H \to K$ such that $\varphi = \widehat{\varphi} \circ \iota$.
\end{defi}

Being characterized by a universal property, universal $\kappa$-extensions are unique up to a unique isomorphism and we may speak of \emph{the} universal $\kappa$-extension.
Therefore this gives an equivalent characterization of $\lambda^-$-braided $\kappa$-overmonoids.

\begin{teor} \label{t:universal-exist}
  Let $\lambda \le \kappa^+$.
  Let $H$ be a reduced $\lambda^-$-monoid.
  \begin{enumerate}
  \item \label{ue:braided} There exists a $\kappa$-monoid $\widehat H \supseteq H$ which is a $\lambda^-$-overmonoid of $H$ and which is $\lambda^-$-braided over $H$.
  \item \label{ue:universal} The $\kappa$-monoid $\widehat H$ is the universal $\kappa$-extension of $H$.
  \end{enumerate}
\end{teor}

\begin{proof}
  \begin{proofenumerate}
  \item[\ref{ue:braided}]
  As a set, let $\widehat H \coloneqq H^\kappa /\! \sim$, where $\sim$ is the relation of being $\lambda^-$-braided.
  We already know that this relation is an equivalence relation (\cref{l:braiding-transitive}).
  We denote equivalence classes by square brackets.
  For $\nu \in \kappa$, let $(x_{\mu}^{(\nu)})_{\mu \in \kappa} \in H^\kappa$.
  The operation on $\widehat H$ is the one induced from the concatenation of families, that is,
  \[
    \sum_{\nu \in \kappa} [ (x_\mu^{(\nu)})_{\mu \in \kappa} ] \coloneqq [(x_{\mu}^{(\nu)})_{\mu, \nu \in \kappa}].
  \]
  To see that this is well-defined, recall first that the braiding relation does not depend on the order of elements in the family (\cref{l:braiding-basic}), so it does not matter which bijection we use to identify $\kappa \times \kappa$ with $\kappa$ on the right side.
  Further, if $(x_\mu^{(\nu)})_{\mu \in \kappa} \sim (y_\mu^{(\nu)})_{\mu \in \kappa}$ for every $\nu \in \kappa$, then these $\lambda^-$-braidings can trivially be concatenated to a $\lambda^-$-braiding of $(x_\mu^{(\nu)})_{\mu,\nu \in \kappa}$ and $(y_\mu^{(\nu)})_{\mu,\nu \in \kappa}$.
  Thus the operation is well-defined.
  It is straightforward that $\widehat H$ is a $\kappa$-monoid.

  We can embed $H$ into $\widehat H$, by mapping $x$ to $[(x_\mu)_{\mu \in \kappa}]$ with $x_0=x$ and $x_\mu = 0$ for $\mu \in \kappa \setminus \{0\}$.
  (By \cref{l:braiding-basic} it does not actually matter into which component we embed $H$.)
  By \ref{coll:trivial} of \cref{l:braiding-collapse}, the embedding is a $\lambda^-$-homomorphism.
  Since $H$ is reduced, the embedding is easily seen to be injective.
  Without restriction, we may therefore consider $H$ to be a $\lambda^-$-submonoid of $\widehat H$.
  By construction $\widehat H$ is generated by $H$ as a $\kappa$-monoid.
  Also by construction, the $\kappa$-monoid $\widehat H$ is $\lambda^-$-braided over $H$.
  
  \item[\ref{ue:universal}] It suffices to verify the universal property.
  Since $\widehat H$ is $\lambda^-$-braided over $H$, this follows from \cref{p:braiding-extend}. \qedhere
  \end{proofenumerate}
\end{proof}

\subsection{Some universal \texorpdfstring{$\kappa$}{κ}-extensions}

To finish this section, we determine some universal $\kappa$-extensions.
Aside from providing some easy examples, this will come in handy in the context of monoids of modules in \cref{e:module-examples}.

\begin{examples} \label{e:univ-kappa-ext}
  From \cref{e:kappa-braided}, we obtain some universal $\aleph_0$-extensions.
  Namely
  \begin{align*}
    \widehat{\bN_0} &\cong \bN_0 \cup \{\infty\},
    & \widehat{\bR_{\ge 0}} &\cong \bR_{\ge 0} \cup \widetilde{\bR}_{>0} \cup \{\infty\},
    & \widehat{\bQ_{\ge 0}} &\cong \bQ_{\ge 0} \cup \widetilde{\bR}_{>0} \cup \{\infty\}.
  \end{align*}
\end{examples}

We start with a lemma. Recall that $F_{\kappa}(B)$ denotes the free $\kappa$-monoid on a basis $B$, and $F_{\lambda^-}(B)$ denotes the free $\lambda^-$-monoid on a basis $B$.
Also recall that if $H$ is a monoid and $S$ is a submonoid, then $S \subseteq H$ is \defit{saturated} if, whenever $s=t+h$ with $s$, $t \in S$ and $h \in H$, then also $h \in S$.

\begin{lemma} \label{l:ext-basic}
  Let $\lambda \le \kappa^+$.
  \begin{enumerate}
  \item \label{ext-basic:free} The free $\kappa$-monoid $F_\kappa(B)$ is the universal $\kappa$-extension of the free $\lambda^{-}$-monoid $F_{\lambda^-}(B)$.
  \item \label{ext-basic:sub}
    Let $H$ be a $\lambda^-$-monoid, let $\widehat H$ be the universal $\kappa$-extension, and let $S \subseteq H$ be a $\lambda^-$-submonoid.
    Suppose that $\lambda \ne \aleph_0$ or that $S \subseteq H$ is a saturated submonoid.
    Then $\langle S \rangle_{\kappa} \subseteq \widehat H$ is the universal $\kappa$-extension of $S$.
  \end{enumerate}
\end{lemma}

\begin{proof}
  \begin{proofenumerate}
  \item[\ref{ext-basic:free}]
  The universal property (\cref{d:univprop}) is satisfied by \cref{p:free}.

\item[\ref{ext-basic:sub}]
  If we can show that $\langle S \rangle_\kappa$ is $\lambda^-$-braided over $S$, then it follows that  $\langle S \rangle_\kappa=\widehat S$.
  Let $(x_i)_{i \in \kappa}$ and $(y_j)_{j \in \kappa}$ be families in $S$ with $\sum_{i \in \kappa} x_i = \sum_{j \in \kappa} y_j \in \widehat H$.
  Because $\widehat H$ is $\lambda^-$-braided over $H$, there exist braiding partitions $(I_\mu)_{\mu \in \kappa}$, $(J_\mu)_{\mu \in \kappa}$ and braiding families $(u_\mu)_{\mu \in \kappa}$, $(v_\mu)_{\mu \in \kappa}$ in $H$.
  We have to show $u_\mu$, $v_\mu \in S$.

  First consider the case in which $\lambda \ne \aleph_0$.
  Then \cref{l:braiding-collapse}\ref{coll:uncountable} allows us to assume $v_\mu = 0$ for all $\mu \in \kappa$.
  Hence $u_\mu = \sum_{i \in I_\mu} x_i = \sum_{j \in J_\mu} y_j \in S$.

  Now consider the case in which $S$ is a saturated submonoid of $H$.
  If $\mu$ is a limit element, then $v_\mu = 0$ and $u_\mu = \sum_{i \in I_\mu} x_i \in S$.
  Observe
  \[
    u_{\mu+n} + v_{\mu+n} = \sum_{i \in I_{\mu + n}} x_i \in S\qquad \text{and}\qquad u_{\mu+n} + v_{\mu+n+1} = \sum_{j \in J_{\mu + n}} y_j \in S
  \]
  for all limit elements $\mu$ and $n \in \bN_0$.
  Thus, the saturatedness of $S$ in $H$ and transfinite induction imply $u_{\mu+n}$, $v_{\mu + n} \in S$. \qedhere
  \end{proofenumerate}
\end{proof}

We now consider a natural class of submonoids of finitely generated free abelian monoids.
Recall that $F_{\kappa}$ denotes the $\kappa$-monoid consisting of all cardinals less than or equal to $\kappa$, and analogously $F_{\kappa^-}$ denotes the $\kappa^-$-monoid of all cardinals strictly less than $\kappa$ (if $\kappa$ is regular).
In particular $F_{\aleph_0^-}=\bN_0$ and $F_{\aleph_0} = \bN_0 \cup \{\aleph_0\}$.

Following \cite{PavelHerbera,infinitepullback}, we consider submonoids of the free $\kappa$-monoid $F_\kappa^n$ (with $n \in \bN_0$) defined by any finite number and combination of
\begin{itemize}
\item homogeneous linear equations of the form $a_1x_1+\cdots + a_nx_n=b_1x_1+ \cdots + b_n x_n$ with coefficients $a_i$, $b_i \in \bN_0$,
\item homogeneous linear inequalities of the form $a_1x_1+\cdots + a_nx_n \le b_1x_1+ \cdots + b_n x_n$ with $a_i$,~$b_i \in \bN_0$, and
\item homogeneous congruences of the form $a_1 x_1 + \cdots + a_n x_n \in d F_\kappa$ for $a_i$, $d \in \bN_0$.
\end{itemize}

Some important differences between submonoids of $\bN_0^n$ versus $\kappa$-submonoids of $F_\kappa^n$ defined by homogeneous linear equations, inequalities, and congruences arise from the fact that $\bN_0^n$ is cancellative while $F_\kappa^n$ is not.

First, while a submonoid of $\bN_0^n$ defined by homogeneous linear inequalities can always be transformed into one defined by equations, through the introduction of additional slack variables, this is no longer the case in $F_{\aleph_0}^n$ (see \cite[Example 4.3]{infinitepullback}).
However, of course, a homogeneous linear equation may always be replaced by two inequalities.

Second, using the cancellativity of $\bN_0^n$, it is also easy to see that a submonoid of $\bN_0^n$ defined by homogeneous linear equations and congruences is saturated.
However, a $\kappa$-submonoid of $F_\kappa^n$ with infinite $\kappa$ defined by homogeneous linear equations and congruences may not be saturated.

\begin{prop} \label{p:ext-diophantine}
  \begin{enumerate}
    \item \label{ext-diophantine:kappa} Let $H \subseteq F_{\aleph_0}^n$ be an $\aleph_0$-monoid defined by any number and combination of homogeneous linear equations, inequalities, and congruences.
    Then the universal $\kappa$-extension $\widehat H$ is the $\kappa$-submonoid of $F_\kappa^n$ defined by the same homogeneous linear equations, inequalities, and congruences.
    \item \label{ext-diophantine:aleph} Let $H \subseteq \bN_0^n$ be a monoid defined by any number and combination of homogeneous linear equations and congruences.
      Then the universal $\aleph_0$-extension $\widehat H$ is the $\aleph_0$-submonoid
      \[
        H + \aleph_0 H \subseteq F_{\aleph_0}^n.
      \]
      Explicitly, the elements of $\aleph_0 H$ are obtained from those of $H$ by replacing every nonzero component by $\aleph_0$ while the zeroes remain.
  \end{enumerate}
\end{prop}

We point out that \ref{ext-diophantine:aleph} may no longer hold if inequalities are involved in the definition of $H$ (see \cref{e:diophantine-inequalities} below).
Nevertheless, the proposition is applicable in this case as well, by first transforming the inequalities into equations through slack variables, as illustrated in the same example.

\begin{proof}
  \begin{proofenumerate}
  \item[\ref{ext-diophantine:kappa}]
  By \ref{ext-basic:sub} of \cref{l:ext-basic}, which is applicable with $\lambda = \aleph_1$, it suffices to show that $\langle H \rangle_{\kappa}$ in $F_\kappa^n$ is defined by the same equations, inequalities, and congruences as $H$ itself in $F_{\aleph_0}^n$.
  Any element of $\langle H \rangle_{\kappa}$ will satisfy these equations, inequalities, and congruences.
  So it suffices to show: given a system of equations/inequalities/congruences in $F_\kappa^n$, every solution can be represented as a $\kappa$-sum of solutions in $F_{\aleph_0}^n$.

  First consider a single linear equation
  \begin{equation} \label{eq:lineareq}
    a_1 x_1 + \cdots + a_n x_n = b_1 x_1 + \cdots + b_n x_n.
  \end{equation}
  Let $\alpha=(\alpha_1,\ldots,\alpha_n) \in F_\kappa^n$ be a solution of \cref{eq:lineareq}.
  If any $\alpha_i$ is infinite and $a_i \ne 0$ or $b_i \ne 0$, then $(\alpha_1,\ldots,\alpha_n)$ being a solution of \cref{eq:lineareq} is equivalent to
  \[
    \max\{ a_1 \alpha_1, \ldots, a_n \alpha_n \} = \max\{ b_1 \alpha_1, \ldots, b_n \alpha_n\}.
  \]
  Thus, if we define $\beta=(\beta_1,\ldots,\beta_n)$ as $\beta_i = \min\{\alpha_i,\aleph_0\}$,  we obtain a solution in $F_{\aleph_0}^n$.
  Moreover, for each cardinal $\aleph_0 \le \lambda \le \kappa$, we can set $\gamma_i^{(\lambda)}=0$ if $\alpha_i < \lambda$ and $\gamma_i^{(\lambda)}=\aleph_0$ if $\alpha_i \ge \lambda$ to obtain another solution $\gamma^{(\lambda)}=(\gamma_1^{(\lambda)},\ldots,\gamma_n^{(\lambda)})$.
  
  The definitions of $\beta$ and $\gamma^{(\lambda)}$ depend on $\alpha$ but not on the coefficients in \cref{eq:lineareq}, and so the same process works for a solution of a system of linear equations.
  Moreover, it also works for inequalities and congruences.
  Thus, if $\alpha \in F_\kappa^n$ is a solution to a system of such equations, inequalities, and congruences, then
  \[
    \alpha = \beta + \sum_{\aleph_0 \le \lambda \le \kappa} \lambda \gamma^{(\lambda)}
  \]
  decomposes as a $\kappa$-sum with $\beta$, $\gamma^{(\lambda)} \in F_{\aleph_0}^n$ solutions of the same system of equations, inequalities, and congruences.

\item[\ref{ext-diophantine:aleph}]
  The submonoid $H$ of $\bN_0^n$ is saturated.
  By \ref{ext-basic:sub} of \cref{l:ext-basic} it suffices to show $\langle H \rangle_{\aleph_0} = H + \aleph_0H$.
  The inclusion $H +\aleph_0 H \subseteq \langle H \rangle_{\aleph_0}$ is clear.
  Let $a = \sum_{j \in \aleph_0} a^{(j)} \in F_{\aleph_0}^n$ with $a^{(j)} \in H$.
  Suppose $a=(a_1,\ldots,a_n)$.
  Let $I \subseteq \{1,\ldots,n\}$ be the set of all $i$ for which $a_i$ is finite, and let $\overline{I}$ be its complement in $\{1,\ldots,n\}$.
  There is a finite subset $J \subseteq \aleph_0$ such that $a^{(j)}_i = 0$ for all $j \not\in J$ and $i \in I$.
  Then
  \[
    a = \sum_{j \in J} a^{(j)}  + \aleph_0 \sum_{j \in \aleph_0 \setminus J} a^{(j)}.
  \]
  There exists a finite $J' \subseteq \aleph_0 \setminus J$ such that all components of $\sum_{j \in J'} a^{(j)}$ with index in $\overline{I}$ are nonzero.
  Hence
  \[
    \aleph_0 \sum_{j \in \aleph_0 \setminus J} a^{(j)} = \aleph_0 \sum_{j \in J'} a^{(j)} \in \aleph_0 H. \qedhere
  \]
  \end{proofenumerate}
\end{proof}

At this point it is illustrative to recall an example of Herbera and Příhoda \cite[Example 2.8]{PavelHerbera} to see that the conclusion of \ref{ext-diophantine:kappa} indeed does \emph{not} hold for the extension from monoids to $\aleph_0$-monoids.

\begin{example}  \label{exm:eqn-aleph0}
  The monoid $H = \{\, (n,n) : n \in \bN_0 \,\} \subseteq \bN_0^2$ can be defined as solution set of the equation $x=y$ or of $2x=x+y$.
  One gets
  \[
    \widehat H = \{\, (n,n) : n \in \bN_0 \,\} \cup \{(\aleph_0,\aleph_0)\} \cong F_{\aleph_0}.
  \]
  This corresponds to the submonoid of $F_{\aleph_0}^2$ defined by $x=y$.
  However, the submonoid of $F_{\aleph_0}^2$ defined by $2x=x+y$ is
  \[
    \{\, (n,n) : n \in \bN_0 \,\} \cup \{(\aleph_0,n) :n \in \bN_0 \} \cup \{ (\aleph_0,\aleph_0) \}.
  \]
\end{example}

\cref{p:ext-diophantine} implies that this phenomenon disappears for larger cardinals.

The next example shows that \ref{ext-diophantine:aleph} of \cref{p:ext-diophantine} does not hold for submonoids of $\bN_0^n$ defined by homogeneous linear inequalities.
However, the universal $\kappa$-extension of such a submonoid can still be described by first replacing the inequalities by equations through the introduction of slack variables, and then applying \ref{ext-diophantine:aleph} of \cref{p:ext-diophantine}.

\begin{example} \label{e:diophantine-inequalities}
  Let $H = \{\, (x,y) \in \bN_0^2 : x \le y \,\}$.
  Then $(0,1)$, $(1,1) \in H$, but their difference $(1,0)$ is not contained in $H$.
  Therefore, the submonoid $H$ is not saturated in $\bN_0^2$.
  Now $\aleph_0 H = \{ (0,0), (0, \aleph_0), (\aleph_0, \aleph_0) \}$, and so
  \[
  H + \aleph_0 H = H \cup \{\, (n, \aleph_0) : n \in \bN_0 \,\} \cup \{ (\aleph_0, \aleph_0) \}.
  \]
  However, this cannot be the universal $\aleph_0$-extension of $H$:
  consider the two families
  \begin{align*}
  (a_i)_{i \ge 0} &= (0,1), (1,1), (1,1), (1,1), \dots \\
  (b_i)_{i \ge 0} &= (1,1), (1,1), (1,1), (1,1), \dots
  \end{align*}
  Their sum in $H + \aleph_0 H$ could only be $(\aleph_0, \aleph_0)$.
  However, these families are not braided. 
  Indeed, if they were braided, there would be braiding families $(u_i)_{i \ge 0}$ and $(v_i)_{i\ge 0}$ in $H$ with $v_0=0$, and $k$,~$l \ge 0$ such that
  \[
  (l-1, l) = u_0 \qquad\text{and}\qquad (k,k) = v_1 + u_0.
  \]
  (By tacitly telescoping the braiding up to the block containing the exceptional element $(0,1)$, we can assume that it lies in the first block.)
  This is impossible because $v_1 = (k-l+1, k-l) \not \in H$.

  To determine $\widehat H$, note instead $H \cong H' = \{\, (x,y,z) \in \bN_0^3 : y = x+z \,\}$.
  Now $\aleph_0 H' = \{ (0,0,0), (0,\aleph_0,\aleph_0), (\aleph_0,\aleph_0,0), (\aleph_0,\aleph_0,\aleph_0) \}$, and so, using \cref{p:ext-diophantine}, 
  \[
  \widehat H \cong \widehat H' \cong H' \cup \{\, (n,\aleph_0,\aleph_0) : n \in \bN_0 \,\} \cup \{ (\aleph_0,\aleph_0,n) : n \in \bN_0 \} \cup \{ (\aleph_0,\aleph_0,\aleph_0) \}.
  \]
  Now the $\aleph_0$-sum of $(a_i)_{i \ge 0}$ corresponds to $(\aleph_0,\aleph_0,1)$, while the one of $(b_i)_{i \ge 0}$ corresponds to $(\aleph_0,\aleph_0,0)$.
  Hence, the two non-braided families are correctly distinguished in $\widehat H$.
\end{example}

\begin{remark}
  Saturated submonoids of $\bN_0^n$ are finitely generated reduced Krull monoids.
  Conversely, every finitely generated reduced Krull monoid is isomorphic to a Diophantine monoid, that is, a submonoid of $\bN_0^n$ defined by linear homogeneous equations \cite[Theorem 3.1]{ChapmanKrauseOeljeklaus02}.
  \cref{p:ext-diophantine} therefore determines the universal $\kappa$-extensions of finitely generated reduced Krull monoids. In other contexts, finitely generated reduced Krull monoids are also known as reduced normal affine monoids.
\end{remark}

\section{\texorpdfstring{$\kappa$}{κ}-Monoids of modules} \label{sec:modules}

Having developed the notion of braidings and universal $\kappa$-extension, we are now ready to prove our main result on $\kappa$-monoids of modules.
We use a slight variation on the notion of a $\kappa$-small module \cite[Chapters 2.9 and 2.10]{Facchini98}.

\begin{defi}
  Let $\lambda$ be an infinite cardinal and let $R$ be a ring.
  An $R$-module $M$ is \defit{$\lambda^-$-small}, if whenever $\varphi\colon M \to \bigoplus_{i \in I} M_i$ is a homomorphism, then $\im \varphi$ is contained in some $\bigoplus_{i \in I'} M_i$ with $I' \subseteq I$ and $\card{I'} < \lambda$.
\end{defi}

\begin{example}
  \begin{enumerate}
  \item A $\lambda$-small module is one for which only the non-strict inequality $\card{I'} \le \lambda$ holds.
    If $\lambda^+$ denotes the successor of $\lambda$, then every $\lambda$-small module is $(\lambda^+)^-$-small.
    
  \item
    Every module generated by strictly fewer than $\lambda$ many elements is $\lambda^-$-small.
    In particular, finitely generated modules are $\aleph_0^-$-small.
    Countably generated modules are $\aleph_1^{-}$-small (equivalently, $\aleph_0$-small).

  \item Let $\cC$ be a class of modules closed under isomorphisms, under $\kappa$-indexed direct sums, and under direct summands.
    Let $\lambda \le \kappa^+$ be a regular cardinal.
    If $\cC_{\lambda^-}$ is the subclass of modules generated by strictly fewer than $\lambda$ many elements, then $\cC_{\lambda^-}$ is a class of $\lambda^-$-small modules closed under direct sums of cardinality strictly less than $\lambda$, and under direct summands.
  \end{enumerate}
\end{example}

\begin{teor} \label{t:module-braiding}
  Let $R$ be a ring, let $\kappa$ be an infinite cardinal, and let $\cC$ be a class of modules whose isomorphism classes form a set.
  Suppose that $\cC$ is closed under direct sums over index sets of cardinality at most $\kappa$, under direct summands, and under isomorphisms.
  Consider $\Vmon^\kappa(\cC)$ and let $\lambda \le \kappa^+$ be a regular cardinal.
  Let $\cC_{\lambda^-} \subseteq \cC$ be a subclass of $\lambda^-$-small modules that is closed under isomorphisms, direct summands, and under direct sums on index sets of cardinality strictly less than $\lambda$, and let
  \[
    \Vmon^{\lambda^-}(\cC_{\lambda^-}) = \big\{\, [M] \in \Vmon^\kappa(\cC) : M \in \cC_{\lambda^-} \,\big\}.
  \]
  Then the $\kappa$-submonoid of $\Vmon^{\kappa}(\cC)$ generated by $\Vmon^{\lambda^-}(\cC_{\lambda^-})$ is $\lambda^-$-braided over $\Vmon^{\lambda^-}(\cC_{\lambda^-})$.

  In particular, if every module in $\cC$ is a direct sum of at most $\kappa$ many modules in $\cC_{\lambda^-}$, then $\Vmon^{\kappa}(\cC)$ is $\lambda^-$-braided over $\Vmon^{\lambda^-}(\cC_{\lambda^-})$.
\end{teor}

We recall the following basic facts that we use in the proof of \cref{t:module-braiding}.
\begin{enumerate}[label=(M\arabic*),leftmargin=1.25cm] \label{m-facts}
\item \label{m-fact:cancel} If $A$, $B$, $C$ are submodules of some module $M$ and $M=A\oplus B = A \oplus C$, then $B \cong M/A \cong C$.
  (Here it is important that $A$ is the identical submodule on both sides; in general we can of course not conclude $B \cong C$ from $A \oplus B \cong A \oplus C$.)
\item \label{m-fact:summand} If $M=A\oplus B$ and $A \subseteq C \subseteq M$, then $C = A \oplus (B \cap C)$.
  In particular, if $A \subseteq C \subseteq M$ and $A$ is a direct summand of $M$, it is a direct summand of $C$.
\end{enumerate}

\begin{proof}[Proof of \cref{t:module-braiding}]
  Let $H=\big\langle \Vmon^{\lambda^-}(\cC_{\lambda^-}) \big\rangle_{\kappa}$ be the $\kappa$-submonoid of $\Vmon^{\kappa}(\cC)$ generated by $\Vmon^{\lambda^-}(\cC_{\lambda^-})$.
  The claim of the theorem is that $H$ is $\lambda^-$-braided over the $\lambda^-$-monoid $\Vmon^{\lambda^-}(\cC_{\lambda^-})$.
  Suppose that $\bigoplus_{i \in \kappa} A_i \cong \bigoplus_{j \in \kappa} B_j$ for modules $A_i$,~$B_j \in \cC_{\lambda^-}$.
  We have to show that the families $([A_i])_{i \in \kappa}$ and $([B_j])_{j \in \kappa}$ are $\lambda^-$-braided over $\Vmon^{\lambda^-}(\cC_{\lambda^-})$.
  Without restriction we may assume $\bigoplus_{i \in \kappa} A_i = \bigoplus_{j \in \kappa} B_j$, so that the direct sums are internal ones with all modules considered to be submodules of $M\coloneqq \bigoplus_{i \in \kappa} A_i$.

  Fix a limit well-order on $\kappa$, thus ensuring that every $\mu \in \kappa$ has a successor $\mu+1$.
  We will construct indexed partitions $(I_\mu)_{\mu \in \kappa}$ and $(J_\mu)_{\mu \in \kappa}$ of $\kappa$ with $\card{I_\mu}$, $\card{J_\mu} < \lambda$ and families of submodules $(S_\mu)_{\mu \in \kappa}$, $(T_\mu)_{\mu \in \kappa}$ of $M$ such that $T_\mu=0$ for every limit element, and for all $\alpha \in \kappa$
  \begin{equation} \label{eq:braiding:induction-claim}
    \bigoplus_{\mu \le \alpha} \bigoplus_{i \in I_\mu} A_i = \bigoplus_{\mu \le \alpha} S_\mu \oplus T_\mu
    \qquad\text{and}\qquad
    \bigoplus_{\mu \le \alpha} \bigoplus_{j \in J_\mu} B_j = \bigoplus_{\mu \le \alpha} T_{\mu +1} \oplus S_\mu.
  \end{equation}
  
  By \ref{m-fact:cancel}, \cref{eq:braiding:induction-claim} then implies
  \[
    \bigoplus_{i \in I_\mu} A_i \cong S_\mu \oplus T_\mu,
    \qquad\text{and}\qquad
    \bigoplus_{j \in J_\mu} B_j \cong T_{\mu+1} \oplus  S_{\mu} \qquad \text{for all } \mu \in \kappa.
  \]
  Furthermore, since the $A_i$, respectively $B_j$, are in $\cC_{\lambda^-}$ and $\card{I_\mu}$, $\card{J_\mu} < \lambda$, the modules $S_\mu$ and $T_\mu$ are also in $\cC_{\lambda^-}$.

  The construction proceeds by transfinite recursion on $\alpha \in \kappa$.
  So let $\alpha \in \kappa$ and suppose that, for all $\mu < \alpha$, we have defined $I_\mu \subseteq \kappa$, $J_\mu \subseteq \kappa$ with $\card{I_\mu}$, $\card{J_\mu} < \lambda$ and submodules $S_\mu$, $T_\mu$ and $T_{\mu+1}$ of $M$, so that $T_\mu=0$ whenever $\mu$ is a limit element, and the following properties hold:
  \begin{equation} \label{eq:braiding:induction-hypothesis}
    \bigoplus_{\mu < \alpha} \bigoplus_{i \in I_\mu} A_i = \bigoplus_{\mu < \alpha} S_\mu \oplus T_\mu
    \qquad\text{and}\qquad
    \bigoplus_{\mu < \alpha} \bigoplus_{j \in J_\mu} B_j = \bigoplus_{\mu < \alpha} T_{\mu +1} \oplus S_\mu.
  \end{equation}
  We need to define index sets $I_\alpha$, $J_\alpha$ and modules $S_\alpha$, $T_{\alpha+1}$ so that these properties remain true when the direct sums include $\mu=\alpha$.

  Keeping in mind $T_\mu=0$ for limit elements $\mu$, \cref{eq:braiding:induction-hypothesis} yields
  \begin{equation} \label{eq:limit-sum-finite}
    T_{\alpha} \oplus \bigoplus_{\mu < \alpha} \bigoplus_{i \in I_\mu} A_i = \bigoplus_{\mu < \alpha} \bigoplus_{j \in J_\mu} B_j.
  \end{equation}
  Set $I' \coloneqq \kappa \setminus \bigcup_{\mu < \alpha} I_\mu$ and $J' \coloneqq \kappa \setminus \bigcup_{\mu < \alpha} J_\mu$.

  We first deal with the (trivial) case $I'=\emptyset$.
  Then 
  \begin{equation} \label{eq:braiding:limit-sum}
    M = \bigoplus_{\mu < \alpha} \bigoplus_{i \in I_\mu} A_i,
  \end{equation}
  so necessarily $T_\alpha=0$.
  Thus, setting $I_\alpha\coloneqq J_\alpha\coloneqq \emptyset$ and $S_\alpha\coloneqq T_{\alpha+1}\coloneqq 0$, we have that \cref{eq:braiding:induction-claim} is trivially satisfied.

  We may now assume $I' \ne \emptyset$.
  Since $T_\alpha$ is $\lambda^-$-small, there exists some $I_\alpha \subseteq I'$ with $\card{I_\alpha} < \lambda$ such that
  \[
    T_{\alpha} \oplus \bigoplus_{\mu < \alpha} \bigoplus_{i \in I_\mu} A_i \subseteq \bigoplus_{\mu \le \alpha} \bigoplus_{i \in I_\mu} A_i \subseteq M.
  \]
  Enlarging $I_\alpha$ if necessary, we may in addition assume $\min I' \in I_\alpha$.
  Since the left side is a direct summand of $M = \bigoplus_{j \in \kappa} B_j$ (by \cref{eq:limit-sum-finite}), it is also a direct summand of $\bigoplus_{\mu \le \alpha} \bigoplus_{i \in I_\mu} A_i$ by \ref{m-fact:summand}.
  So we can choose $S_\alpha$ such that
  \[
    S_\alpha \oplus T_{\alpha} \oplus \bigoplus_{\mu < \alpha} \bigoplus_{i \in I_\mu} A_i = \bigoplus_{\mu \le \alpha} \bigoplus_{i \in I_\mu} A_i.
  \]
  This deals with the left side of \cref{eq:braiding:induction-claim}.

  Using \ref{m-fact:cancel}, we get $S_\alpha \oplus T_\alpha \cong \bigoplus_{i \in I_\alpha} A_i$.
  Therefore $S_\alpha \in \cC_{\lambda^-}$, in particular, the module $S_\alpha$ is $\lambda^-$-small.
  Thus there exists $J_\alpha \subseteq J'$ with $\card{J_\alpha} < \lambda$ such that
  \[
    S_\alpha \oplus \bigoplus_{\mu < \alpha} \bigoplus_{j \in J_\mu} B_j \subseteq \bigoplus_{\mu \le \alpha} \bigoplus_{j \in J_\mu} B_j.
  \]
  As before, keeping in mind \ref{m-fact:summand}, we find $T_{\alpha+1}$ such that
  \[
    T_{\alpha+1} \oplus S_\alpha \oplus \bigoplus_{\mu < \alpha} \bigoplus_{j \in J_\mu} B_j = \bigoplus_{\mu \le \alpha} \bigoplus_{j \in J_\mu} B_j.
  \]

  Now we have constructed families $(I_\mu)_{\mu \in \kappa}$, $(J_\mu)_{\mu \in \kappa}$ and $(S_\mu)_{\mu \in \kappa}$, $(T_{\mu})_{\mu \in \kappa}$ satisfying \cref{eq:braiding:induction-claim} for all $\alpha \in \kappa$.
  By construction the sets in the families $(I_\mu)_{\mu \in \kappa}$ are pairwise disjoint, and the same holds for $(J_\mu)_{\mu \in \kappa}$.
  It remains to establish $\kappa = \bigcup_{\mu \in \kappa} I_\mu = \bigcup_{\mu \in \kappa} J_\mu$.
  For this, note that when constructing $I_\alpha$, whenever $I' \ne \emptyset$, we have ensured $\min I' \in I_\alpha$.
  Thus, at the point we construct $I_\alpha$, we must have $\min I' \ge \alpha$.
  Therefore $\alpha \in I_\beta$ for some $\beta \le \alpha$.
  Let $J' \coloneqq \kappa \setminus \bigcup_{\mu \in \kappa} J_\mu$.
  Taking the union over all $\alpha$ in \cref{eq:braiding:induction-claim}, we have
  \[
    M = \bigoplus_{\mu \in \kappa} \bigoplus_{i \in I_\mu} A_i = \bigoplus_{\mu \in \kappa} \bigoplus_{j \in J_\mu} B_j.
  \]
  Thus, for every $j \in J'$ we have $B_j=0$.
  We now modify the $J_\mu$'s in such a way, that we add the elements of $J'$ to distinct $J_\mu$'s.
  This preserves $\card{J_\mu} < \lambda$, while all the other equations in \cref{eq:braiding:induction-claim} are trivially preserved due to $B_j=0$.
  Now $\kappa = \bigcup_{\mu \in \kappa} J_\mu$.
\end{proof}

\begin{cor}
  Let $R$ be a ring, let $\kappa$ be an infinite cardinal, and let $\lambda \le \kappa^+$ be a regular cardinal.
  Let $\cC$ be a class of modules whose isomorphism classes form a set.
  Suppose that $\cC$ is closed under direct sums over index sets of cardinality at most $\kappa$, under direct summands, and under isomorphisms.
  Let $\cC_{\lambda^-}$ denote the subclass of modules that are generated by strictly fewer than $\lambda$ many elements.
  \begin{enumerate}
  \item The $\kappa$-submonoid of $\Vmon^{\kappa}(\cC)$ generated by $\Vmon^{\lambda^-}(\cC_{\lambda^-})$ is $\lambda^-$-braided over $\Vmon^{\lambda^-}(\cC_{\lambda^-})$.
  \item If every module in $\cC$ is a direct sum of at most $\kappa$ many modules generated by strictly fewer than $\lambda$ many elements, then  $\Vmon^{\kappa}(\cC)$ is $\lambda^-$-braided over $\Vmon^{\lambda^-}(\cC_{\lambda^-})$.
  \end{enumerate}
\end{cor}

We note the consequences for projective modules.

\begin{cor} \label{cor:projective-braided}
  Let $R$ be a ring.
  Let $\kappa$ be an infinite cardinal, and let $\lambda \le \kappa^+$ be a regular cardinal.
  \begin{enumerate}
  \item \label{pb:generic} For all $\lambda \ge \aleph_1$, the $\kappa$-monoid $\Vmon^\kappa(R)$ is the universal $\kappa$-extension of the $\lambda^-$-monoid $\Vmon^{\lambda^-}(R)$.
  \item \label{pb:countable} $\Vmon^{\kappa}(R)$ is the universal $\kappa$-extension of the $\aleph_0$-monoid $\Vmon^{\aleph_0}(R)$.
  \item \label{pb:fg} If every projective module is a direct sum of finitely generated modules, then $\Vmon^{\kappa}(R)$ is the universal $\kappa$-extension of the monoid $\Vmon(R)$.
  \end{enumerate}
\end{cor}

\begin{proof}
  \begin{proofenumerate}
  \item[\ref{pb:generic}]
  By Kaplansky's Theorem \cite{Kaplansky58}, every projective module is a direct sum of countably generated projective modules, and therefore in particular a direct sum of $\lambda^-$-generated modules.
  Thus $\Vmon^{\kappa}(R)$ is $\lambda^-$-braided over $\Vmon^{\lambda^-}(R)$ by \cref{t:module-braiding}.

  \item[\ref{pb:countable}]
  This is the special case $\lambda=\aleph_1$ of \ref{pb:generic}.

  \item[\ref{pb:fg}]
  Under the stated assumption, we may apply \cref{t:module-braiding} with $\lambda=\aleph_0$. \qedhere
  \end{proofenumerate}
\end{proof}

\Cref{cor:projective-braided} shows that $\Vmon^{\aleph_0}(R)$ fully determines $\Vmon^{\kappa}(R)$ for every infinite cardinal $\kappa$.
It therefore supplements Kaplansky's theorem: not just is every projective module a direct sum of countably generated modules, but all the relations on these direct sums are also induced from relations between countable direct sums of countably generated projective modules.

Of particular interest is the case where every projective module is a direct sum of finitely generated modules.
Then the monoid $\Vmon(R)$ fully determines $\Vmon^{\kappa}(R)$, and thus all relations between direct sums of projective modules are induced from relations between finite direct sums of finitely generated projective modules.
Large classes of rings satisfy this condition.
McGovern, Puninski, and Rothmaler \cite{McGovernPuninskiRothmaler07} give a great overview over these results and give a characterization of rings over which every projective module is a direct sum of finitely generated modules \cite[Theorem 4.2]{McGovernPuninskiRothmaler07}.
We summarize the consequences in the following.

\begin{cor} \label{cor:fin-braided}
  Let $R$ be a ring and $\kappa$ an infinite cardinal.
  Under each of the following conditions, $\Vmon^{\kappa}(R)$ is the universal $\kappa$-extension of $\Vmon(R)$.
  \begin{enumerate}
  \item \label{fb:weakhereditary}$R$ is weakly semihereditary.
  \item \label{fb:hereditary} $R$ is left or right semihereditary \textup(that is, finitely generated left or right ideals are projective\textup).
  \item \label{fb:exchange} $R$ is an exchange ring.
  \item \label{fb:semiperfect} $R$ is semiperfect.
  \item \label{fb:noetherian} $R$ is a weakly noetherian commutative ring.
  \item \label{fb:bezout} $R$ is a Bézout ring with one-sided Krull dimension; or $R$ is a commutative Bézout ring with finite uniform dimension.
  \end{enumerate}
\end{cor}

\begin{proof}
  Statement \ref{fb:weakhereditary} is a result of Bergman \cite[Proposition 4.4]{Bergman72}, see also \cite[Section 5]{McGovernPuninskiRothmaler07}, generalizing earlier results of Kaplansky in the commutative case \cite{Kaplansky58}, Albrecht in the hereditary case \cite{Albrecht61}, and Bass in the semihereditary case \cite[Theorem 3]{Bass64}.
  Statement \ref{fb:hereditary} is a special case of \ref{fb:weakhereditary}.
  Note that the modules under consideration are always \emph{right} modules, but also the \emph{left} semihereditary property implies that right projective modules are direct sums of finitely generated modules.

  The case of exchange rings is due to Warfield \cite[Corollary 3]{Warfield72}, see also \cite[Fact 3.6]{McGovernPuninskiRothmaler07}, with the special case of semiperfect rings due to Mueller~\cite[Theorem 3]{Mueller70}.
  For weakly noetherian commutative rings the result was shown by Hinohara \cite{Hinohara63}, see also \cite[Fact 3.1]{McGovernPuninskiRothmaler07}.
  The last result is due to McGovern, Puninski, and Rothmaler \cite[Corollary 8.2 and Proposition 8.6]{McGovernPuninskiRothmaler07}.
\end{proof}

Given a $\kappa$-monoid $H$ and element $x \in H$, we write
\[
  \add(x) \coloneqq \{\, y \in H : \text{there exist } z \in H,\, n \in \bN_0 \text{ with } y+z=nx \,\}
\]
for the set of all summands of multiples of $x$.
(In multiplicative language it is the divisor-closed submonoid of $H$ generated by $x$.)
For an infinite cardinal $\lambda \le \kappa$, we similarly define
\[
  \add_{\lambda}(x) \coloneqq \{\, y \in H : \text{there exists } z \in H \text{ with } y+z=\lambda x \,\} = \add\big(\langle x \rangle_\lambda\big).
\]

The monoid $\Vmon(R)$ is a reduced commutative monoid with order-unit (induced by the isomorphism class of the right module $R$).
In their seminal papers Bergman \cite{Bergman} and Bergman and Dicks \cite{BergWar} show that for hereditary rings a converse is true: for every reduced commutative monoid $H$ with order-unit and every field $k$, there exists a hereditary $k$-algebra $R$ with $\Vmon(R) \cong H$.
This realization result depends crucially on Bergman's universal localizations.

More recently, Ara and Goodearl have shown that one can take $R$ to be a (non-unital) Leavitt path algebra of a finitely separated graph \cite[Proposition 4.4, Corollary 4.5]{AraGoodearl12}.
At its heart the construction of Ara and Goodearl still makes use of Bergman's universal localizations.
However it has the attractive feature that the algebras are defined by a separated graph, and given a presentation of a monoid it is straightforward to construct a suitable separated graph.

Since $\Vmon^{\kappa}(R)$ is braided over $\Vmon(R)$ for hereditary rings $R$, the machinery we have developed allows us to immediately extend the realization results to $\kappa$-monoids.

\begin{cor}\label{conclusion}
  Let $H$ be a $\kappa$-monoid.
  \begin{enumerate}
  \item \label{hereditary}
    The following statements are equivalent.
    \begin{equivenumerate}
    \item\label{conclusion-hereditary:monoid} There exists $x \in H$ such that $H$ is braided over $\add(x)$.
    \item\label{conclusion-hereditary:all}  For every field $k$ there exists a hereditary $k$-algebra $R$ such that $\Vmon^{\kappa}(R) \cong H$.
    \item\label{conclusion-hereditary:one} There exists a right hereditary ring $R$ such that $\Vmon^{\kappa}(R) \cong H$.
    \end{equivenumerate}

  \item \label{nonhereditary}
    The following statements are equivalent for a ring $R$ and a $\kappa$-monoid $H$.
    \begin{equivenumerate}
      \item \label{conclusion-nonhereditary:monoid} There exists $x \in H$ such that $H$ is $\aleph_1^-$-braided over $\add_{\aleph_0}(x)$ and $\Vmon^{\aleph_0}(R) \cong \add_{\aleph_0}(x)$.
      \item \label{conclusion-nonhereditary:one} $\Vmon^{\kappa}(R) \cong H$.
    \end{equivenumerate}
  \end{enumerate}
\end{cor}

\begin{proof}
  \begin{proofenumerate}
  \item[\ref{hereditary}]
  \ref{conclusion-hereditary:monoid}$\,\Rightarrow\,$\ref{conclusion-hereditary:all}
  The monoid $\add(x)$ is a reduced monoid with order-unit.
  By the Bergman--Dicks realization result there exists a $k$-algebra $R$ such that $\Vmon(R) \cong \add(x)$.
  Now \cref{cor:fin-braided} implies that $\Vmon^{\kappa}(R)$ is braided over $\Vmon(R)$.
  Hence $\Vmon^{\kappa}(R)$ and $H$ are both the universal $\kappa$-extension of $\add(x)$ and therefore isomorphic as $\kappa$-monoids.

  \ref{conclusion-hereditary:all}$\,\Rightarrow\,$\ref{conclusion-hereditary:one} Trivial.
  
  \ref{conclusion-hereditary:one}$\,\Rightarrow\,$\ref{conclusion-hereditary:monoid}
  We have $\Vmon(R) = \add([R])$.
  By \cref{cor:fin-braided} the $\kappa$-monoid $\Vmon^{\kappa}(R)$ is braided over $\Vmon(R)$.
  
  \item[\ref{nonhereditary}]
  \ref{conclusion-nonhereditary:monoid}$\,\Rightarrow\,$\ref{conclusion-nonhereditary:one}
  By \ref{pb:countable} of Corollary \ref{cor:projective-braided} and uniqueness of the universal $\kappa$-extension, we have $\Vmon^{\kappa}(R) \cong H$.

  \ref{conclusion-nonhereditary:one}$\,\Rightarrow\,$\ref{conclusion-nonhereditary:monoid}
  By \ref{pb:countable} of Corollary \ref{cor:projective-braided}. \qedhere
  \end{proofenumerate}
\end{proof}

In case there are projective modules that are not direct sums of finitely generated modules, the monoid $\Vmon(R)$ in general does not determine $\Vmon^{\aleph_0}(R)$.
However, in this case $\Vmon^{\aleph_0}(R)$ still determines $\Vmon^{\kappa}(R)$ for all $\kappa > \aleph_0$.

Herbera and Příhoda \cite{PavelHerbera,infinitepullback} as well as Herbera, Příhoda, and Wiegand \cite{HerberaPrihodaWiegand23} have recently studied the \emph{monoid} of countably generated projective modules $\Vmon^*(R)$ for several interesting classes of rings.
In these cases the monoids are typically submonoids of some $F_{\aleph_0}^n$, with the components induced from some dimension function, and so it is easy to see that the induced $\aleph_0$-monoid structure is the one coming from $F_{\aleph_0}^n$.
One may then apply our results to extend the description from $\Vmon^*(R)$ to $\Vmon^{\kappa}(R)$.

We discuss these as well as some other examples in the context of our results.

\needspace{3cm}
\begin{examples} \label{e:module-examples} \mbox{} 
  \begin{enumerate}
  \item \textbf{Ascent of finite/countable KRSA to infinite KRSA.}
    Let $\cC$ be a class of modules closed under $\kappa$-indexed direct sums, direct summands, and isomorphisms.
    We also assume that the isomorphism classes in $\cC$ form a set.
    The class $\cC$ satisfies the Krull--Remak--Schmidt--Azumaya (KRSA) property for \emph{finite} direct sums if and only if every module is a finite direct sum of indecomposables and 
    \[
      \bigoplus_{i=1}^m A_i \cong \bigoplus_{i=1}^n B_i
    \]
    with indecomposable modules $A_i$,~$B_i \in \cC$ implies that $m=n$ and that there exists a permutation $\sigma$ such that $A_i \cong B_{\sigma(i)}$ for all $i$.
    KRSA for finite direct sums is equivalent to $\Vmon(\cC)$ being a free abelian monoid.

    Let $\cC_{\text{fg}}$ denote the subclass of finitely generated modules.
    If every module in $\cC$ is a direct sum of at most $\kappa$ many finitely generated modules, then $\Vmon^{\kappa}(\cC)$ is braided over $\Vmon(\cC_{\text{fg}})$ by \cref{t:module-braiding}.
    If $\Vmon(\cC_{\text{fg}})$ is a free abelian monoid, then $\Vmon^{\kappa}(\cC)$ is a free $\kappa$-monoid (\cref{l:ext-basic}).
    Thus, direct sums of indecomposable modules in $\cC$ indexed by arbitrary index sets are unique.
    In other words, our results show that finite KRSA implies infinite KRSA if every module in $\cC$ is a direct sum of finitely generated modules.

    Analogously, if every module in $\cC$ is a direct sum of at most $\kappa$ many $<\!\lambda$-generated modules for some regular cardinal $\lambda$ and $\Vmon^{\lambda^-}(\cC_{\lambda^-})$ is a free $\lambda^-$-monoid, then $\Vmon^{\kappa}(\cC)$ is a free $\kappa$-monoid for all $\kappa \ge \lambda$.
   
    In particular: if countable direct-sum decompositions of projective modules into countably generated indecomposable modules exist and are unique, then (arbitrary) direct-sum decompositions into indecomposable projective modules are unique. (Using Kaplansky's theorem to get that all indecomposable projectives are countably generated.)

    Suppose that $R$ has finite KRSA for finitely generated projective modules.
    Then \cref{t:module-braiding} together with \cref{l:ext-basic} shows that direct-sum decompositions of projective modules into direct sums of \emph{finitely generated} indecomposable projective modules are unique.
    However, one must be careful in that direct sums of finitely-generated projective modules may in general also have direct summands that are \emph{not} direct sums of finitely generated modules, and \cref{t:module-braiding} does not imply anything about decompositions involving such modules.
    In terms of the $\kappa$-monoids, the $\kappa$-monoid $\langle \Vmon(R) \rangle_{\kappa}$ is braided over $\Vmon(R)$, but  $\big\langle\! \Vmon(R) \big\rangle_{\kappa}$ may not be closed under summands (in other words, divisor-closed) in $\Vmon^{\kappa}(R)$.
    
  \item \textbf{Direct summands of infinite direct sums.}
    When studying infinite direct sums in some class of modules, a natural question that arises is whether every direct summand of a direct sum of finitely generated modules must itself be a direct sum of finitely generated modules (see \cite[Question 1.1]{HerberaPrihodaWiegand23}).
    This question can now naturally be translated into the setting of $\kappa$-monoids as follows.
    Let $\cC$ denote a class of modules of interest (closed under direct sums over index sets of cardinality $\kappa$, direct summands, and isomorphism), and let $\cC_{\text{fg}}$ denote the subclass of finitely generated modules.
    The question then becomes: when is the $\kappa$-submonoid of $\Vmon^\kappa(\cC)$ generated by the submonoid $\Vmon(\cC_{\text{fg}})$ closed under summands (i.e., divisor-closed in the multiplicative language)?

  \item Suppose that $R$ is a ring over which every projective module is a direct sum of finitely generated modules.
    Then $\Vmon(R)$ fully determines $\Vmon^{\kappa}(R)$, and in particular $\Vmon^{\aleph_0}(R)$.

    \begin{itemize}
    \item If $\Vmon(R) \cong \bN_0$, then $\Vmon^{\aleph_0}(R) \cong \bN_0 \cup \{\infty\}$ by \ref{e:kappa-braided:N} of \cref{e:kappa-braided}.
    \item If $\Vmon(R)\cong\bQ_{\ge 0}$ or $\Vmon(R)\cong\bR_{\ge 0}$, then \ref{e:kappa-braided:R} and \ref{e:kappa-braided:Q} of \cref{e:kappa-braided} show
      \[
        \Vmon^{\aleph_0}(R) \cong \Vmon(R) \cup \widetilde{\bR}_{> 0} \cup \{\infty\}.
      \]
    \item If $\Vmon(R)$ is isomorphic to a submonoid of $\bN_0^n$, defined by homogeneous linear equations, inequalities, and congruences, then \cref{p:ext-diophantine} describes $\Vmon^{\aleph_0}(R)$ and further $\Vmon^{\kappa}(R)$ for $\kappa > \aleph_0$.
    (To apply \cref{p:ext-diophantine}, inequalities must be turned into equations by introducing slack variables, as in \cref{e:diophantine-inequalities}.)
    \end{itemize}

  \item Let $R$ be a commutative hereditary domain, that is, a Dedekind domain.
    In this case Steinitz's theorem fully describes $\Vmon(R)$.
    Every finitely generated projective module $P \ne 0$ can be represented as $P \cong R^n \oplus I$ with $n \ge 0$ and $0 \ne I$ an ideal of $R$.
    The rank $n+1$ and the class $[I] \in G\coloneqq \Pic(R)$ uniquely determine $P$ up to isomorphism.
    Thus
    \[
      \Vmon(R) = \{0\} \cup \big((\bN_{\ge 1},+) \times G\big) = \big\{\, (n,g) \in \bN_0 \times G : n \ge 1 \text{ or } g = 0\,\big\}.
    \]
    Because $R$ is hereditary, every projective module is a direct sum of finitely generated projective modules.
    It is easy to check that two families with infinite support over $\Vmon(R)$ are braided if and only if their ranks (the first components) add up to the same cardinal.
    This follows inductively from the observation that $(n,g)$ is always a summand of $(m,h)$ if $1 \le n < m$.
    Thus
    \[
      \Vmon^{\kappa}(R) = \widehat{\Vmon(R)} \cong \{\, (\alpha,g) \in F_\kappa \times G : 1 \le \alpha < \aleph_0 \text{ or } g = 0\,\}.
    \]
    In particular, non-finitely generated projective modules are free.
    (Of course this is well-known, see for instance the next point.)

  \item 
    Let $R$ be a nonzero commutative noetherian ring that is connected (that is, it does not contain any non-trivial idempotent).
    By a well-known theorem of Bass \cite[Corollary 4.5]{Bass63} non-finitely generated projective modules over $R$ are free.
    Thus, in this case,
    \[
      \Vmon^{\kappa}(R) = \Vmon(R) \cup \{\, \lambda : \aleph_0 \le \lambda \le \kappa \,\}.
    \]

  \item \label{modexm:semilocalnoetherian}
    A ring $R$ is semilocal if $R/J(R)$ is semisimple (here $J(R)$ denotes the Jacobson radical).
    In this case $\Vmon(R/J(R)) \cong \bN_0^n$ and the map $R \to R/J(R)$ induces a monoid monomorphism $\Vmon(R) \to \Vmon(R/J(R))$, that is in fact even a divisor homomorphism.
    Then $\Vmon(R)$ is a Krull monoid and the embedding $\Vmon(R) \to \Vmon(R/J(R))$ allows one to study this monoid.

    Denote by $\Vmon^*(R)$ the monoid of countably generated projective modules.
    Herbera and Příhoda have shown that $\Vmon^*(R)$ can also be described as a submonoid of $\Vmon^*(R/J(R)) \cong F_{\aleph_0}^n$ (their notation for $F_{\aleph_0}$ is $\bN_0^*= \bN_0 \cup \{\infty\}$) by homogeneous linear equations and congruences, in case $R$ is noetherian \cite[Theorem 2.6]{PavelHerbera}.
    This involves the use of significant additional results from module theory, beyond the ones needed for the finitely generated case.
    
    If $R$ is non-noetherian one can also realize monoids defined by homogeneous linear inequalities \cite[Theorem 1.6]{infinitepullback}, but a full description of all possible $\Vmon^*(R)$ is not known.
    
    From the fact that the components of $\bN_0^n$ are induced by a dimension vector on modules over $R/J(R)$, it is easy to see that the monoids described by Herbera and Příhoda are in fact $\aleph_0$-submonoids of $F_{\aleph_0}^n$.
    Thus their work fully describes $\Vmon^{\aleph_0}(R)$.
    It is now possible, with our purely monoid-theoretical methods, to extend this to a description of $\Vmon^{\kappa}(R)$: indeed $\Vmon^{\kappa}(R)$ is the $\kappa$-submonoid of $F_\kappa^n$ defined by the same linear equations, inequalities, and congruences as $\Vmon^{\aleph_0}(R)$ by \ref{ext-diophantine:kappa} of \cref{p:ext-diophantine}.

    In the special case that every projective module is a direct sum of finitely generated projective modules, the monoids $\Vmon^{\aleph_0}(R)$ in turn afford an easy description in terms of $\Vmon(R)$ \cite[Corollary 7.9]{PavelHerbera}.
    As one would expect, we can now recover this result using our monoid-theoretical methods (specifically, this is \ref{ext-diophantine:aleph} of \cref{p:ext-diophantine}).

  \item For hereditary noetherian prime (HNP) rings, a far reaching generalization of Steinitz's theorem describes (stable) isomorphism classes of finitely generated projective modules.
    This theory is developed in the monograph by Levy and Robson \cite{LevyRobson11}.
    Two modules $M$, $N$ are stably isomorphic if $M \oplus R^n \cong N \oplus R^n$ for some $n \ge 1$.
    In the case of HNP rings one can always take $n = 1$.
    We write $U(R)$, $U^{\kappa}(R)$, etc. for monoids ($\kappa$-monoids) of stable isomorphism classes of projective modules.

    If the uniform dimension of the modules is at least two, then stably isomorphic modules are isomorphic \cite[Theorem 34.6]{LevyRobson11}, and so the difference between $\Vmon^{\kappa}(R)$ and $U^{\kappa}(R)$ is not as big as one might initially fear.
    In particular, any stable isomorphism of an infinite direct sum will in fact be an isomorphism and so the conclusion of \cref{t:module-braiding} holds.
    Thus $U^{\kappa}(R)$ is braided over $U(R)$, and so in principle $U^{\kappa}(R)$ can be determined from $U(R)$ with monoid-theoretic methods.
    However, we do not carry this out here: $U^{\kappa}(R)$ has already been described  in detail \cite[Chapter 8]{LevyRobson11}.
    We merely describe the two resulting monoids.

    A description of $U(R)$ can be deduced from \cite[Chapter 6]{LevyRobson11}. These monoids and arithmetical invariants describing the factorization of elements therein (hence arithmetical invariants of direct-sum decompositions of finitely generated projective modules) were studied in \cite[Section 6]{BaethGeroldingerGrynkiewiczSmertnig15}, see in particular \cite[Theorem 6.5]{BaethGeroldingerGrynkiewiczSmertnig15}.
    For a reduced commutative cancellative monoid $D$ and an abelian group $G$ define
    \[
      D \propto G = \{\, (h,g) \in D \times G : h \ne 0 \text{ or } g = 0 \,\}.
    \]
    Let $\bN_0^{\Omega}(\mathbf c, \Lambda)$ be defined as in \cite[Definition 6.3]{BaethGeroldingerGrynkiewiczSmertnig15}, as a submonoid of $\bN_0^{\Omega}$ (the set $\Omega$ corresponds to the $0$ module and the isomorphism classes of unfaithful simple modules in non-trivial towers; the vector $\mathbf c=(c_i)_{i \in \Omega} \in \bQ_{\ge 0}^\Omega$ encodes the local ranks of $R$ at the unfaithful simple modules, and $\Lambda$ is a partition of a subset of $\Omega\setminus \{0\}$ encoding the structure of cycle towers).
    By \cite[Theorem 6.5]{BaethGeroldingerGrynkiewiczSmertnig15}
    \[
      U(R) \cong \bN_0^{\Omega}(\mathbf c, \Lambda) \propto G(R),
    \]
    with $G(R)$ the ideal class group of $R$.
    Then $U^{\kappa}(R) \setminus U(R) \subseteq F_{\kappa}^{\Omega}$ consists of all $\mathbf x=(x_i)_{i \in \Omega} \in F_{\kappa}^{\Omega}$ for which $x_0$ is infinite, and the following conditions are satisfied:
    \begin{itemize}
    \item $x_i \le x_0$ for all $i  \in \Omega$,
    \item for every infinite cardinal $\alpha < x_0$, there are only finitely many $i \in \Omega$ with $x_i \le \alpha$, and
    \item for every finite $n$, there are only finitely many $i \in \Omega$ with $x_i \le n c_i$.
    \item for every cycle tower $T \subseteq \Omega \setminus \{0\}$, there is at least one $i \in T$ with $x_i=x_0$.
    \end{itemize}
    This follows from \cite[Theorems 45.7, 45.13, and 46.1]{LevyRobson11}.

    If $R$ is a Dedekind prime ring, then $\bN_0^{\Omega}(\mathbf c, \Lambda)=\bN_0$ and the description becomes much simpler; indeed
    \[
      U^{\kappa}(R) = (\bN_0 \propto G(R)) \cup \{\, \lambda : \aleph_0 \le \lambda \le \kappa \,\},
    \]
    and the only (possible) difference to the commutative case is that stably isomorphic uniform right ideals need not be isomorphic.
  \end{enumerate}
\end{examples}

\section{Realization to hereditary rings for two-generated \texorpdfstring{$\aleph_0$}{ℵ₀}-monoids} \label{sec:twogen}

Over a ring where every projective module is a direct sum of finitely generated modules (such as a hereditary ring), the $\aleph_0$-monoid $\Vmon^{\aleph_0}(R)$ is completely determined by the monoid $\Vmon(R)$ (\cref{conclusion}).
Every reduced monoid $H$ with an order-unit is isomorphic to $\Vmon(R)$ for a hereditary $k$-algebra for an arbitrary field $k$, by a theorem of Bergman and Dicks (see \cref{sec:modules}).
Thus, understanding the $\aleph_0$-monoids $H$ that can be realized as $\Vmon^{\aleph_0}(R)$ of a hereditary ring amounts to understanding the universal $\aleph_0$-extensions of reduced commutative monoids with order-unit.

It is easily seen that a cyclic $\aleph_0$-monoid with generator $x$ can be realized as $\Vmon^{\aleph_0}(R)$ of a nonzero hereditary ring if and only if $\aleph_0 x \ne n x$ for all $n \in \bN_0$.
Equivalently, the monoid is of the form $C \cup \{\infty\}$ with $C$ a nonzero cyclic monoid (these monoids are described in \cref{ss:free-modules}).
In this section we give a description of all two-generated $\aleph_0$-monoids that can be realized as $\Vmon^{\aleph_0}(R)$ of a hereditary ring.

Throughout the section, let $H$ be a non-cyclic $\aleph_0$-monoid generated by two elements $x_1$ and $x_2$.
Then every $y \in H$ can be represented as $y = \sum_{k \in \aleph_0} y_k$ with $y_k$ equal to $0$, $x_1$, or $x_2$.
Since the order of the elements in the family $(y_k)_{k \in \aleph_0}$ does not matter, it is convenient to represent such a family as $\alpha X_1 + \beta X_2$, where $0 \le \alpha \le \aleph_0$ is the number of elements in the family that are equal to $x_1$, and $0 \le \beta \le \aleph_0$ is the number of elements that are equal to $x_2$.
We call such a representation $\alpha X_1 + \beta X_2$ a \defit{form}.
Each form $\alpha X_1 + \beta X_2$ represents an element $\alpha x_1 + \beta x_2 \in H$ (but of course, multiple forms may represent the same element).
We say that $\alpha X_1 + \beta X_2$ is an \defit{infinite form} if at least one of $\alpha$,~$\beta$ is infinite, and a \defit{finite form} otherwise.
Depending on $H$, an element may have both infinite and finite forms.
Whenever we consider a braiding we need to fix a well-order on $\aleph_0$. 
Since the choice of well-order does not matter, we shall without restriction assume that $\aleph_0$ is well-ordered as the natural numbers $\bN_0$.

We start with a number of basic observations.

\begin{lemma} \label{l:twogen-braided}
  Suppose $H \cong \Vmon^{\aleph_0}(R)$ for some ring $R$ over which projective modules are direct sums of finitely generated modules.
  Then $x_1$ and $x_2$ correspond to finitely generated modules, the $\aleph_0$-monoid $H$ is braided over $\add(x_1+x_2)$, and
  \[
    \Vmon(R) \cong \add(x_1+x_2)=\langle x_1,x_2 \rangle.
  \]
\end{lemma}

\begin{proof}
  Let $[P_i]$ be the image of $x_i$ under the isomorphism.
  As $x_1$,~$x_2 \ne 0$ by assumption, the modules $P_1$ and $P_2$ must both be nonzero.
  
  Every countably generated projective module $M$ can be expressed as $M \cong P_1^{(\alpha)} \oplus P_2^{(\beta)}$ for some $0 \le \alpha$,~$\beta \le \aleph_0$.
  If $M$ is finitely generated, then $\alpha$ and $\beta$ must both be finite.
  Furthermore, both $P_1$ and $P_2$ must show up as direct summands of some finitely generated modules; for otherwise $\Vmon(R)$ would be cyclic, and hence so would be $\Vmon^{\aleph_0}(R)$, as it is generated as $\aleph_0$-monoid by $\Vmon(R)$.
  We conclude that $P_1$ and $P_2$ are finitely generated and $\Vmon(R) = \big\langle [P_1], [P_2] \big\rangle = \add([P_1]+[P_2])$.
  By \ref{pb:fg} of \cref{cor:projective-braided}, the $\aleph_0$-monoid $\Vmon^{\aleph_0}(R)$ is braided over $\add([R])=\add([P_1]+[P_2])$.
\end{proof}

\begin{lemma} \label{easyfactlemma1} \mbox{}
\begin{enumerate}
  \item\label{efl1:inf-fin} An infinite and a finite form cannot be braided.
  \item\label{efl1:fin} Two finite forms of an element are braided over $\add(x_1 + x_2)$.
  \item\label{efl1:add-inclusion} If $x_i \in \add(x_j)$ and no element in $H$ has an infinite and a finite form, then the two forms  $\alpha X_i + \aleph_0 X_j$ and $\beta X_i + \aleph_0 X_j$ are braided over $\add(x_1 + x_2)$ for every $0 \le \alpha$,~$\beta \le \aleph_0$.
  \item\label{efl1:add-incomparable} Suppose $x_i \not\in \add(x_j)$ and no element in $H$ has an infinite and a finite form.
    If $m X_i + \aleph_0 X_j$ and $n X_i + \aleph_0 X_j$ with $m$,~$n$ finite are braided over $\add(x_1+x_2)$, then $m x_i + k x_j = n x_i + k' x_j$ for some finite $k$, $k'$.
  \item\label{efl1:add-inclusion-iff} Let $H$ be braided over $\add (x_1 + x_2)$.
    Then we have $x_i \in \add(x_j)$ if and only if $\aleph_0 (x_1 + x_2)  = \aleph_0 x_j$. 
\end{enumerate}
\end{lemma}

\begin{proof}
  \begin{proofenumerate}
  \item[\ref{efl1:inf-fin}] and \ref{efl1:fin} are clear.
  
  \item[\ref{efl1:add-inclusion}]
  Since the braiding relation is transitive and symmetric, it suffices to consider the case $\beta=0$.
  First consider the case in which $\alpha$ is finite.
  We want to show that $\alpha X_i + \aleph_0 X_j$ and $\aleph_0 X_j$ are braided. 
  Since $\alpha$ is finite and $x_i \in \add(x_j)$, the element $\alpha x_i$ is in $\text{add} (x_j)$.
  Thus, for some $t \in H$ and integer $n \geq 1$, we have $nx_j = t + \alpha x_i$.
  The form $\alpha X_i + \aleph_0 X_j$ corresponds to a family $(y_k)_{k \in \aleph_0}$, where $y_k = x_i$ for $0 \leq k \leq \alpha-1$ and $y_k = x_j$ for $k \geq \alpha$.
  To show that $(y_k)_{k \in \aleph_0}$ and the constant family $(x_j)_{k \in \aleph_0}$ are braided, define partitions
  \[
    \begin{split}
      I_0 & \coloneqq \{ 0, \ldots, \alpha-1\},\\
      I_k &\coloneqq \{ \alpha + n(k-1), \ldots, \alpha + n (k - 1) + n - 1 \} \quad\text{for $k \ge 1$, and}\\
      J_k &\coloneqq \{ n k, \ldots,  n k + n - 1 \} \quad\text{for $k \ge 0$}.
    \end{split}
  \]
  Setting $u_k \coloneqq \alpha x_i$ for every $k \ge 0$ and $v_k \coloneqq t$ for every $k \ge 1$ (with $v_0=0$), gives the desired braiding families.
  
  Now consider the case $\alpha = \aleph_{0}$.
  There exists an integer $m >0$ and $t \in H$ such that $m x_j = x_i + t$.
  Since $H$ is generated by $x_1$ and $x_2$, we can write $t = m' x_i + n' x_j$.
  We have that $m'$ and $n'$ are finite, because otherwise $m x_j$ would have a finite and an infinite form, contradicting our assumption.
  Now $m x_j = (m' + 1) x_i + n' x_j$.
  Adding $x_j$ to both sides if necessary, we may assume $m$,~$n' > 0$.
  Consider the constant family $(x_j)_{k \in \aleph_0}$ and the family $(y_k)_{k \in \aleph_0}$ where
  \[
    y_k =
    \begin{cases}
      x_i & \text{for $l(m'+n'+1) \le k \le l(m'+n'+1) + m'$} \\
      x_j & \text{for $l(m'+n'+1)+m'+1 \le k \le l(m'+n'+1)+m'+n'$}
    \end{cases}
  \]
  for all $l \ge 0$.
  Then
  \[
    \sum_{r=0}^{m'+n'} y_{l(m'+n'+1)+r} = (m'+1)x_i + n'x_j= m x_j,
  \]
  and it follows easily that the two families are braided.

  \item[\ref{efl1:add-incomparable}]
  Let $(I_\mu)_{\mu \in \aleph_0}$ and $(J_\mu)_{\mu \in \aleph_0}$ be braiding partitions of $m X_i + \aleph_0 X_j$ and $n X_i + \aleph_0 X_j$ with corresponding braiding families $(u_\mu)_{\mu \in \aleph_0}$ and $(v_\mu)_{\mu \in \aleph_0}$.
  Choose a finite $\alpha \ge 0$ such that $\bigcup_{\mu \le \alpha} I_\mu$ and $\bigcup_{\mu \le \alpha} J_\mu$ cover all of the finitely many copies of $x_i$.
  Then there exist finite $m'$, $n'$ such that
  \[
    m x_i + m' x_j + v_{\alpha+1} = \sum_{\mu \le \alpha} (u_\mu + v_\mu) + v_{\alpha + 1} = \sum_{\mu \le \alpha} u_{\mu} + v_{\mu + 1} = n x_i + n' x_j.
  \]
  Further $u_{\alpha+1} + v_{\alpha +1}$ must be a finite multiple of $x_j$, so that $v_{\alpha + 1} \in \add(x_j)$, say, $v_{\alpha+1} + w = rx_j$ for some $w \in H$ and some integer $r$.

  Because $x_1$ and $x_2$ generate $H$, we can write $v_{\alpha+1} = \beta x_i + \gamma x_j$ with $0 \le \beta$,~$\gamma \le \aleph_0$.
  Now $x_i \not \in \add(x_j)$ and $v_{\alpha+1} \in \add(x_j)$ implies $\beta = 0$.
  Applying the property that no element has an infinite and a finite form to $rx_j$, we see that $\gamma$ is finite.
  Thus $m x_i + (m'+\gamma) x_j = n x_i + n'x_j$.

  \item[\ref{efl1:add-inclusion-iff}]
  Clearly $x_i \in \add(x_j)$ implies $\aleph_0 x_j = \aleph_0(x_1 + x_2)$.
  Suppose $\aleph_0 (x_1 + x_2) = \aleph_0 x_j$.
  Since $H$ is braided over $\add(x_i + x_j)$, we conclude that there exist positive integers $m$,~$n$, and $t \in \add(x_1+x_2)$ such that $n x_j = m (x_1 +  x_2) + t$.
  Thus, $x_i \in \text{add}(x_j)$. \qedhere
  \end{proofenumerate}
\end{proof}

The following describes two-generated $\aleph_0$-monoids realizable over a hereditary ring. A characterization of these monoids and their realization over von Neumann regular rings is given and studied in detail in the follow-up work \cite{Nazemian2026}.

\begin{teor} \label{hereditarycasecor}
  A non-cyclic $\aleph_0$-monoid $H$ generated by two elements $x_1$,~$x_2$ is isomorphic to $\Vmon^{\aleph_0}(R)$ for a ring $R$ whose projective modules are direct sums of finitely generated modules if and only if the following conditions hold for $1 \leq i \neq j \leq 2$.
  \begin{enumerate}[label=\textup(\roman*\textup),leftmargin=2.5em]
  \item \label{hcc:threeinf} \label{hcc:first} $n x_i + \aleph_0 x_j = \aleph_0 x_i + \aleph_0 x_j$ with $n$ finite implies $\aleph_0 x_j = \aleph_0 x_i +\aleph_0 x_j$ and $x_i \in \add(x_j)$.
  \item \label{hcc:twoinf} If $x_i \not \in \add(x_j)$ and $m x_i + \aleph_0 x_j = n x_i + \aleph_0 x_j$ with $m$,~$n$ finite, then there exist finite $k$,~$k'$ such that $m x_i + k x_j = n x_i + k' x_j$.
  \item \label{hcc:finite-infinite} \label{hcc:last} An element of $H$ cannot have both finite and infinite forms.
  \end{enumerate}
  In fact, if $H$ satisfies these conditions, then $H \cong \Vmon^{\aleph_0}(R)$ for a hereditary ring $R$.
\end{teor}

\begin{proof}
  First suppose that $H \cong \Vmon^{\aleph_0}(R)$ for a ring $R$ whose projective modules are direct sums of finitely generated modules.
  Then $H$ is braided over $\add(x_1+x_2)$ by \cref{l:twogen-braided}.
  Thus \ref{hcc:finite-infinite} follows from \cref{easyfactlemma1}\ref{efl1:inf-fin}.
  Consider condition~\ref{hcc:twoinf}.
  If $m x_i + \aleph_0 x_j=n x_i + \aleph_0 x_j$, then $m X_i + \aleph_0 X_j$ and $n X_i + \aleph_0 X_j$ are braided.
  Hence \cref{easyfactlemma1}\ref{efl1:add-incomparable} yields \ref{hcc:twoinf}.

  We check \ref{hcc:threeinf}.
  If $n x_i + \aleph_0 x_j = \aleph_0 x_i + \aleph_0 x_j$, then the forms $n X_i + \aleph_0 X_j$ and $\aleph_0 X_i + \aleph_0 X_j$ are braided.
  Label the first family, namely $n X_i + \aleph_0 X_j$, by $(y_k)_{k \in \aleph_0}$ in some arbitrary way.
  Suppose we have braiding partitions $(I_\mu)_{\mu \in \aleph_0}$ and $(J_\mu)_{\mu \in \aleph_0}$.
  Since all but finitely many elements of $(y_k)_{k \in \aleph_0}$ are equal to $x_j$, we must have $\sum_{i \in I_\mu} y_i = \card{I_\mu} x_j$ for all but finitely many $\mu \in \aleph_0$.
  However, the form $\aleph_0 X_i + \aleph_0 X_j$ contains infinitely many copies of $x_i$, and so necessarily $x_i \in \add(x_j)$ from the braiding relation.
  Now \cref{easyfactlemma1}\ref{efl1:add-inclusion} implies that $\aleph_0 X_j$ and $\aleph_0 X_i + \aleph_0 X_j$ are braided.
  Hence $\aleph_0 x_j = \aleph_0 x_i + \aleph_0 x_j$, and the claim follows.
  
  Now suppose that $H$ satisfies conditions \ref{hcc:first}--\ref{hcc:last}.
  To show that $H \cong \Vmon^{\aleph_0}(R)$ for some hereditary ring $R$, it suffices to show that $H$ is braided over $\add(x_1+x_2)$, by \ref{hereditary} of \cref{conclusion}.
  By condition~\ref{hcc:finite-infinite}, elements of $\add(x_1+x_2)$ can only have finite forms.
  Hence $\add(x_1+x_2) = \langle x_1, x_2 \rangle$.
  If $(y_j)_{j \in \aleph_0}$ is a family of elements in $\add(x_1+x_2)$ with $y_j = \gamma_j x_1 + \delta_j x_2$ for some finite $\gamma_j$,~$\delta_j$, then clearly $(y_j)_{j \in \aleph_0}$ is braided with $(\sum_{j \in \aleph_0} \gamma_j) X_1 + (\sum_{j \in \aleph_0} \delta_j) X_2$.
  By transitivity of the braiding relation, it therefore suffices to consider families consisting of the two generators $x_1$ and $x_2$.
  It therefore suffices to prove: if $\alpha x_i + \beta x_j =  \alpha' x_i + \beta' x_j$, then $\alpha X_i + \beta X_j$ and $\alpha' X_i + \beta' X_j$ are braided.
  By condition \ref{hcc:finite-infinite}, it suffices to consider the following cases, where $m$, $n$, $m'$, $n'$ are finite.

  \begin{itemize}
  \item  $m x_i + n x_j = m' x_i + n' x_j$:
    In this case $m X_i + n X_j$ and $m' X_i + n' X_j$ are braided by \ref{efl1:fin} of \cref{easyfactlemma1}.
  \item $m x_i + \aleph_0 x_j = m' x_i + \aleph_0 x_j$:
    If $x_i \in \add(x_j)$, then \cref{easyfactlemma1}\ref{efl1:add-inclusion} implies that $m X_i + \aleph_0 X_j$ and $m' X_i + \aleph_0 X_j$ are braided.
    Suppose $x_i \not \in \add(x_j)$.
    Then condition \ref{hcc:twoinf} implies $m x_i + k x_j = m' x_i + k' x_j$ for some finite $k$,~$k'$.
    From this we can easily construct a braiding of $m X_i + \aleph_0 X_j$ and $m' X_i + \aleph_0 X_j$.

  \item $\aleph_0 x_i + n x_j = m x_i + \aleph_0 x_j$:
    Adding $\aleph_0 x_j$, respectively, $\aleph_0 x_i$ to both sides, condition \ref{hcc:threeinf} implies $x_i \in \add(x_j)$ and $x_j \in \add(x_i)$.
    Then \Cref{easyfactlemma1}\ref{efl1:add-inclusion} shows that $\aleph_0 X_i + n X_j$ and $\aleph_0 X_i + \aleph_0 X_j$ are braided.
    The forms $m X_i + \aleph_0 X_j$ and $\aleph_0 X_i + \aleph_0 X_j$ are braided by the same lemma.
    Now transitivity and symmetry of the braiding relation implies that $\aleph_0 X_i + n X_j$ and $m X_i + \aleph_0 X_j$ are braided.

  \item $m x_i + \aleph_0 x_j = \aleph_0 x_i + \aleph_0 x_j$:
    We get $\aleph_0 x_j = \aleph_0 x_i + \aleph_0 x_j$ and $x_i \in \add(x_j)$ from condition~\ref{hcc:threeinf}.
    \Cref{easyfactlemma1}\ref{efl1:add-inclusion} shows that $m X_i + \aleph_0 X_j$ and $\aleph_0 X_i + \aleph_0 X_j$ are braided.
  \end{itemize}
  Thus, in each case two forms representing the same element are braided, and so $H$ is braided over $\add(x_1+x_2)$.
\end{proof}

Conditions of the type $x_i \in \add(x_j)$ and $x_i \not \in \add(x_j)$ appear in the previous theorem, and as they are natural conditions in terms of the modules, it makes some sense to distinguish these cases and look at realizability in each of them separately.
For this we first need a result on trace ideals.

If $P$ is a projective module, the \defit{trace} of $P$ is
\[
  \Tr(P) \coloneqq \sum_{f \in \Hom(P,R)} \im(f).
\]
Recall that $\Tr(P)$ is an idempotent ideal of $R$, and that $\Tr(P)$ is the least among all ideals $I$ such that $P=PI$, see for example \cite{PavelandPuninski} and \cite{whitehead}.

\begin{prop} \label{traceideal}
  Let $R$ be a ring with non-cyclic $\Vmon^{\aleph_0}(R)$ and two countably generated projective modules $P_1$ and $P_2$ whose isomorphism classes generate $\Vmon^{\aleph_0}(R)$.
  Then the following statements are equivalent:
  \begin{multicols}{2}
  \noindent
  \begin{equivenumerate}
  \item \label{trace:ji} $\Tr(P_2) \subseteq \Tr(P_1)$,
  \item \label{trace:ir} $\Tr(P_1)=R$,
  \item \label{trace:summand} $P_2$ is a direct summand of $P_1^{(\aleph_0)}$,
  \item \label{trace:free} $P_1^{(\aleph_0)}$ is free.
  \end{equivenumerate}
\end{multicols}
\noindent
If $R$ is hereditary, then we have $\Tr(P_1)=\Tr(P_2)$ if and only if all countably \textup(non-finitely\textup) generated projective modules are free.
\end{prop} 

\begin{proof}
  Since $[P_1]$ and $[P_2]$ generate $\Vmon^{\aleph_0}(R)$, there are finite $m$,~$n \ge 0$ such that $R \cong P_1^m \oplus P_2^n$.
  Replacing $P_1$ and $P_2$ by isomorphic copies, we may assume $R=P_1^m \oplus P_2^n$.
  Let $I\coloneqq \Tr(P_1)$ and $J\coloneqq \Tr(P_2)$.

  \begin{proofenumerate}
  \item[\ref{trace:ji}$\,\Rightarrow\,$\ref{trace:ir}] If $J \subseteq I$, then $P_2I = P_2$, and hence $I=R$.

  \item[\ref{trace:ir}$\,\Rightarrow\,$\ref{trace:summand}]
  By the definition of the trace we have $R = \Tr(P_1)= \sum_{f \in \Hom(P_1,R)} \im(f)$.
  Let $1_R \in \im(f_1) + \cdots + \im(f_k)$.
  Then we have an epimorphism $f_1 + \cdots + f_k :  P_1^k  \to R$.
  Therefore, the module $R$ is isomorphic to a direct summand of $P_1^k$.
  Consequently, the module $R^{(\aleph_0)}$ is isomorphic to a direct summand of $P_1^{(\aleph_0)}$.
  The claim follows because $P_2$ is a direct summand of $R^{(\aleph_0)}$.

  \item[\ref{trace:summand}$\,\Rightarrow\,$\ref{trace:free}]
  The assumption implies that $P_1^{(\aleph_0)}$ is an order-unit.
  Furthermore $P_1^{(\aleph_0)} \oplus P_2^{(\aleph_0)} \cong R^{(\aleph_0)}$.
  Now \cref{sumwithbig} implies $P_1^{(\aleph_0)} \cong R^{(\aleph_0)}$.

  \item[\ref{trace:free}$\,\Rightarrow\,$\ref{trace:ji}]
  We have $I=\Tr(P_1)=R$ and hence $J \subseteq I$.
  \end{proofenumerate}

  For the final statement, suppose $\Tr(P_1)=\Tr(P_2)$ (the other direction is trivial).
  Then $P_1^{(\aleph_0)}$ and $P_2^{(\aleph_0)}$ are free by \ref{trace:free}.
  By \cref{l:twogen-braided} the modules $P_1$ and $P_2$ are finitely generated.
  Any countably (non-finitely) generated projective module $P$ is therefore isomorphic to $P_1^{(\alpha)} \oplus P_2^{(\beta)}$ with $0 \le \alpha$, $\beta \le \aleph_0$, and at least one of $\alpha$ and $\beta$ infinite.
  Since the isomorphism class of $P_1^{(\aleph_0)}\cong P_2^{(\aleph_0)} \cong R^{(\aleph_0)}$ is an order-unit, \cref{sumwithbig} implies $P_1^{(\alpha)} \oplus P_2^{(\beta)} \cong R^{(\aleph_0)}$.
\end{proof}

In the following, keep in mind that if $H \cong \Vmon^{\aleph_0}(R)$ for some ring $R$ whose projective modules are direct sums of finitely generated modules, then $\Vmon^{\aleph_0}(R)$ is braided over $\add([R])$ and hence there also exists a hereditary $k$-algebra $S$ such that $H \cong \Vmon^{\aleph_0}(S)$.
In particular, any such $H$ satisfies conditions~\ref{hcc:first}--\ref{hcc:last} of \cref{hereditarycasecor}.

\begin{cor}\label{hereditarycase}
Let $H$ be a non-cyclic $\aleph_0$-monoid with two generators $x_1$,~$x_2$.
\begin{enumerate}
\item\label{hc:incomparable} Suppose that $x_i \not \in \add (x_j)$ for $i \neq j$. Then the following statements are equivalent.
\begin{equivenumerate}
\item\label{hc:i:realization} $H\cong \Vmon^{\aleph_0}(R)$ for a ring whose projective modules are direct sums of finitely generated modules.
\item\label{hc:i:relations} Every equality $\alpha_1 x_i + \beta_1 x_j = \alpha _2 x_i + \beta_2 x_j$ \textup(with $i \ne j$\textup) implies that $\alpha_1$ and $\alpha_2$ are both finite or both infinite.
  Furthermore, if $\alpha_1=\alpha_2=\aleph_0$, then $m_1 x_i + \beta_1 x_j = m_2 x_i + \beta _2 x_j$ for some finite integers $m_1$,~$m_2$.
\end{equivenumerate}
\item\label{hc:equality} If $\add(x_1) = \add (x_2)$, then $H$ has only one element with an infinite form.
  The converse is not true.
  Furthermore, the following statements are equivalent.
\begin{equivenumerate}
\item \label{hc:e:equality} $\add (x_1) = \add (x_2)$ and no element has an infinite and a finite form.
\item \label{hc:e:realization} $H \cong \Vmon^{\aleph_0}(R)$ for a ring whose countably \textup(non-finitely\textup) generated projective modules are free.
\end{equivenumerate}
\item\label{hc:inclusion} If $x_1\in \add(x_2)$, then $\aleph_0 x_2 = \aleph_0 x_2 + \beta x_1$ for every cardinal $\beta \le \aleph_0$. The converse is not true.
  If $x_1\in \add(x_2)$ and $x_2\notin \add(x_1)$, then the following statements are equivalent.
\begin{equivenumerate}
\item \label{hc:sub:realization} $H \cong \Vmon^{\aleph_0}(R)$ for a ring $R$ over which projective modules are direct sums of finitely generated modules, and with a finitely generated projective module $P$ such that $P^{(\aleph_0)}$ is not free.
\item \label{hc:sub:relations} If $\aleph _0 x_1 + n x_2 = \aleph_0 x_1 + \beta x_2$ with $n$ finite and $\beta \le \aleph_0$, then $\beta$ is finite and there are finite integers $m$,~$m'$ such that $m x_1 + \beta x_2 = m' x_1 + n x_2$.
  Furthermore, no element has an infinite and a finite form.
\item \label{hc:sub:trace} $H \cong \Vmon^{\aleph_0}(R)$ for a ring $R$ with two finitely generated projective modules $P_1$ and $P_2$ such that $\Vmon^{\aleph_0}(R) = \big\langle [P_1], [P_2] \big\rangle_{\aleph_0}$ and $\Tr(P_1) \subsetneq \Tr (P_2)$.
\end{equivenumerate}
\end{enumerate}
\end{cor}

\begin{proof}
  \begin{proofenumerate}
  \item[\ref{hc:incomparable}]
  \ref{hc:i:realization}$\,\Rightarrow\,$\ref{hc:i:relations}
  The conditions in \cref{hereditarycasecor} hold.
  Suppose $\aleph _0 x_i + \beta _1 x_j = \alpha_2 x_i + \beta_2 x_j$ with $\alpha_2$ finite.
  Adding $\aleph_0 x_j$ yields $\aleph_0 x_i + \aleph_0 x_j = \alpha_2 x_i + \aleph_0 x_j$.
  Then \cref{hereditarycasecor}\ref{hcc:threeinf} implies $x_i \in \add(x_j)$, a contradiction.
  The second implication holds by \cref{hereditarycasecor}\ref{hcc:twoinf}.

  \ref{hc:i:relations}$\,\Rightarrow\,$\ref{hc:i:realization}
  We must verify the conditions of \cref{hereditarycasecor}.
  Condition~\ref{hcc:threeinf} holds vacuously because the assumption is never met,
  condition~\ref{hcc:twoinf} holds by assumption.
  Finally, condition~\ref{hcc:finite-infinite} trivially follows from the stronger assumption in \ref{hc:i:relations}.

  \item[\ref{hc:equality}]
  The element $x \coloneqq x_1+x_2$ is an order-unit.
  From $x_i \in \add(x_j)$ (with $i \ne j$) we have $k x_j = x_i + t$ for some finite $k$ and some $t \in H$.
  Thus $(k+1)x_j = x + t$, and $\aleph_0 x_j = \aleph_0 x + \aleph_0 t$.
  By \cref{sumwithbig} this implies $\aleph_0 x_j = \aleph_0 x$.
  Since $H$ is generated by $x_1$ and $x_2$, every infinite form represents the same element $\aleph_0 x$.

  To see that the converse is not true, let $H \coloneqq \mathbb N_0^2 \cup \{\infty\}$.
  Take $x_1\coloneqq (1, 0)$, $x_2\coloneqq (0, 1)$ and consider $H$ as $\aleph_0$-monoid where sums with infinite support are equal to $\infty$. 
Then clearly $\text{add}(x_1) $ and $\text{add}(x_2)$ are different and neither is contained in the other.

\ref{hc:e:equality}$\,\Rightarrow\,$\ref{hc:e:realization}
  We verify the conditions of \cref{hereditarycasecor}.
  Condition~\ref{hcc:threeinf} holds because there is only one element with an infinite form.
  Condition~\ref{hcc:twoinf} holds vacuously, because the condition is never met.
  Condition~\ref{hcc:finite-infinite} holds by assumption.
  The elements represented by finite forms correspond to finitely generated projective modules, and so the unique (up to isomorphism) countably but not finitely generated projective module is free.

  \ref{hc:e:realization}$\,\Rightarrow\,$\ref{hc:e:equality}
  Since there is a unique countably generated but not finitely generated projective module up to isomorphism, \cref{hereditarycasecor}\ref{hcc:threeinf} implies $x_i \in \add(x_j)$ and $x_j \in \add(x_i)$.
  Condition~\ref{hcc:finite-infinite} implies the remaining claim.

  \item[\ref{hc:inclusion}]
  Let $x \coloneqq x_1 + x_2$.
  As in the proof of \ref{hc:equality}, we see $\aleph_0 x_2 = \aleph_0 x$.
  Consequently, $\aleph_0 x_2 + \beta x_1 = \aleph_0 x_2$ for every $\beta \leq \aleph_0$.
  The example $\bN_0^2 \cup \{\infty\}$ in \ref{hc:equality} satisfies $\aleph_0 x_2 = \aleph_0 x_2 + \beta x_1$ for every cardinal $\beta \le \aleph_0$, but $x_1\notin \add(x_2)$.
  
  Now suppose $x_1\in \add(x_2)$ and  $x_2\notin \add(x_1)$.

  \ref{hc:sub:realization}$\,\Rightarrow\,$\ref{hc:sub:relations}
  The first statement follows from \cref{hereditarycasecor}\ref{hcc:threeinf} because $x_2 \not \in \add(x_1)$.
  The other two statements are \ref{hcc:twoinf} and \ref{hcc:finite-infinite} of \cref{hereditarycasecor}.

  \ref{hc:sub:relations}$\,\Rightarrow\,$\ref{hc:sub:realization}
  Conditions \ref{hcc:twoinf} and \ref{hcc:finite-infinite} of \cref{hereditarycasecor} are clear from \ref{hc:sub:relations}.
  For \ref{hcc:threeinf} we have to check two cases.
  The case $x_i=x_1$ and $x_j=x_2$ is true because $\aleph_0 x_2 = \aleph_0 x_1 + \aleph_0 x_2$ and $x_1 \in \add(x_2)$.
  In the case $x_i=x_2$ and $x_j=x_1$, we have $n x_2 + \aleph_0 x_1 \ne \aleph_0 x_2 + \aleph_0 x_1$ by assumption, and so in this case \ref{hcc:threeinf} holds vacuously.

  This realizes $H \cong V^{\aleph_0}(R)$ for a ring $R$ as claimed.
  To obtain a module $P$ with $P^{(\aleph_0)}$ non-free, take $P=P_1$, where $[P_i]$ corresponds to $x_i$.
  If $P_1^{(\aleph_0)}$ were free, then $P_2$ would be a direct summand of $P_1^{(\aleph_0)}$, and, since $P_2$ is finitely generated, also of $P_1^k$ for some $k$.
  Then $x_2 \in \add(x_1)$, contradicting our assumptions.

  \ref{hc:sub:realization}$\,\Rightarrow\,$\ref{hc:sub:trace}
  Let $[P_i]$ be the image of $x_i$ in $\Vmon^{\aleph_0}(R)$.
  Since $P_2^{(\aleph_0)} \cong P_1^{(\aleph_0)} \oplus P_2^{(\aleph_0)} \cong R^{(\aleph_0)}$, the module $P_2^{(\aleph_0)}$ is free.
  So $\Tr(P_2) = R$ and $\Tr(P_1) \subseteq \Tr(P_2)$.
  If there were equality, then $P_1^{(\aleph_0)} \cong P_2^{(\aleph_0)}$ is free by \cref{traceideal}.
  Since $P_2$ is finitely generated (\cref{l:twogen-braided}), this means $[P_2]$ is in $\add([P_1])$, in contradiction to our assumptions. 

  \ref{hc:sub:trace}$\,\Rightarrow\,$\ref{hc:sub:realization}
  By assumption and Kaplansky's theorem, every projective module is a direct sum of finitely generated projective modules.
  By Proposition \ref{traceideal}, the module $P_1^{(\aleph_0)}$ is non-free.\qedhere
  \end{proofenumerate}
\end{proof}

\bibliographystyle{hyperalphaabbr}
\bibliography{kappa_monoids}

\newcommand{\etalchar}[1]{$^{#1}$}
\begin{thebibliography}{AAB{\etalchar{+}}26}

\bibitem[AAB{\etalchar{+}}25]{AntoineAraBosaPereraVilalta25}
R.~Antoine, P.~Ara, J.~Bosa, F.~Perera, and E.~Vilalta.
\newblock The {C}untz semigroup of a ring.
\newblock {\em Selecta Math. (N.S.)}, 31(1):Paper No. 6, 46, 2025.
\newblock \href {http://dx.doi.org/10.1007/s00029-024-01002-9}
  {\path{doi:10.1007/s00029-024-01002-9}}.

\bibitem[AAB{\etalchar{+}}26]{AntoineAraBosaPereraVilalta26}
R.~Antoine, P.~Ara, J.~Bosa, F.~Perera, and E.~Vilalta.
\newblock Ideals, quotients, and continuity of the cuntz semigroup for rings.
\newblock {\em Commun. Contemp. Math.}, page 2650037, 2026.
\newblock To appear.
\newblock \href {http://arxiv.org/abs/2411.00507} {\path{arXiv:2411.00507}},
  \href {http://dx.doi.org/10.1142/S0219199726500379}
  {\path{doi:10.1142/S0219199726500379}}.

\bibitem[ABP20]{PJuan}
P.~Ara, J.~Bosa, and E.~Pardo.
\newblock The realization problem for finitely generated refinement monoids.
\newblock {\em Selecta Math. (N.S.)}, 26(3):Paper No. 33, 63, 2020.
\newblock \href {http://dx.doi.org/10.1007/s00029-020-00559-5}
  {\path{doi:10.1007/s00029-020-00559-5}}.

\bibitem[AG12]{AraGoodearl12}
P.~Ara and K.~R. Goodearl.
\newblock Leavitt path algebras of separated graphs.
\newblock {\em J. Reine Angew. Math.}, 669:165--224, 2012.
\newblock \href {http://dx.doi.org/10.1515/crelle.2011.146}
  {\path{doi:10.1515/crelle.2011.146}}.

\bibitem[Alb61]{Albrecht61}
F.~Albrecht.
\newblock On projective modules over semi-hereditary rings.
\newblock {\em Proc. Amer. Math. Soc.}, 12:638--639, 1961.
\newblock \href {http://dx.doi.org/10.2307/2034259}
  {\path{doi:10.2307/2034259}}.

\bibitem[Bae09]{Baeth09}
N.~R. Baeth.
\newblock Direct sum decompositions over two-dimensional local domains.
\newblock {\em Comm. Algebra}, 37(5):1469--1480, 2009.
\newblock \href {http://dx.doi.org/10.1080/00927870802064663}
  {\path{doi:10.1080/00927870802064663}}.

\bibitem[Bas63]{Bass63}
H.~Bass.
\newblock Big projective modules are free.
\newblock {\em Illinois J. Math.}, 7:24--31, 1963.
\newblock \href {http://dx.doi.org/10.1215/ijm/1255637479}
  {\path{doi:10.1215/ijm/1255637479}}.

\bibitem[Bas64]{Bass64}
H.~Bass.
\newblock Projective modules over free groups are free.
\newblock {\em J. Algebra}, 1:367--373, 1964.
\newblock \href {http://dx.doi.org/10.1016/0021-8693(64)90016-X}
  {\path{doi:10.1016/0021-8693(64)90016-X}}.

\bibitem[BD78]{BergWar}
G.~M. Bergman and W.~Dicks.
\newblock Universal derivations and universal ring constructions.
\newblock {\em Pacific J. Math.}, 79(2):293--337, 1978.

\bibitem[Ber72]{Bergman72}
G.~M. Bergman.
\newblock Hereditarily and cohereditarily projective modules.
\newblock In {\em Ring theory ({P}roc. {C}onf., {P}ark {C}ity, {U}tah, 1971)},
  pages 29--62. Academic Press, New York-London, 1972.

\bibitem[Ber74]{Bergman}
G.~M. Bergman.
\newblock Coproducts and some universal ring constructions.
\newblock {\em Trans. Amer. Math. Soc.}, 200:33--88, 1974.
\newblock \href {http://dx.doi.org/10.2307/1997247}
  {\path{doi:10.2307/1997247}}.

\bibitem[Ber23]{Bergman23}
G.~M. Bergman.
\newblock An elementary result on infinite and finite direct sums of modules.
\newblock {\em J. Algebra}, 631:731--737, 2023.
\newblock \href {http://dx.doi.org/10.1016/j.jalgebra.2023.05.017}
  {\path{doi:10.1016/j.jalgebra.2023.05.017}}.

\bibitem[BG14]{BaethGeroldinger14}
N.~R. Baeth and A.~Geroldinger.
\newblock Monoids of modules and arithmetic of direct-sum decompositions.
\newblock {\em Pacific J. Math.}, 271(2):257--319, 2014.
\newblock \href {http://dx.doi.org/10.2140/pjm.2014.271.257}
  {\path{doi:10.2140/pjm.2014.271.257}}.

\bibitem[BGGS15]{BaethGeroldingerGrynkiewiczSmertnig15}
N.~R. Baeth, A.~Geroldinger, D.~J. Grynkiewicz, and D.~Smertnig.
\newblock A semigroup-theoretical view of direct-sum decompositions and
  associated combinatorial problems.
\newblock {\em J. Algebra Appl.}, 14(2):1550016, 60, 2015.
\newblock \href {http://dx.doi.org/10.1142/S0219498815500164}
  {\path{doi:10.1142/S0219498815500164}}.

\bibitem[BL11]{BaethLuckas11}
N.~R. Baeth and M.~R. Luckas.
\newblock Monoids of torsion-free modules over rings with finite representation
  type.
\newblock {\em J. Commut. Algebra}, 3(4):439--458, 2011.
\newblock \href {http://dx.doi.org/10.1216/JCA-2011-3-4-439}
  {\path{doi:10.1216/JCA-2011-3-4-439}}.

\bibitem[BW13]{BaethWiegand13}
N.~R. Baeth and R.~Wiegand.
\newblock Factorization theory and decompositions of modules.
\newblock {\em Amer. Math. Monthly}, 120(1):3--34, 2013.
\newblock \href {http://dx.doi.org/10.4169/amer.math.monthly.120.01.003}
  {\path{doi:10.4169/amer.math.monthly.120.01.003}}.

\bibitem[CKO02]{ChapmanKrauseOeljeklaus02}
S.~T. Chapman, U.~Krause, and E.~Oeljeklaus.
\newblock On {D}iophantine monoids and their class groups.
\newblock {\em Pacific J. Math.}, 207(1):125--147, 2002.
\newblock \href {http://dx.doi.org/10.2140/pjm.2002.207.125}
  {\path{doi:10.2140/pjm.2002.207.125}}.

\bibitem[Fac98]{Facchini98}
A.~Facchini.
\newblock {\em Module theory}.
\newblock Modern Birkh\"{a}user Classics. Birkh\"{a}user/Springer Basel AG,
  Basel, 1998.
\newblock Endomorphism rings and direct sum decompositions in some classes of
  modules, [2012 reprint of the 1998 original] [MR1634015].

\bibitem[Fac02]{Facchini02}
A.~Facchini.
\newblock Direct sum decompositions of modules, semilocal endomorphism rings,
  and {K}rull monoids.
\newblock {\em J. Algebra}, 256(1):280--307, 2002.
\newblock \href {http://dx.doi.org/10.1016/S0021-8693(02)00164-3}
  {\path{doi:10.1016/S0021-8693(02)00164-3}}.

\bibitem[Fac06]{Facchini06}
A.~Facchini.
\newblock Geometric regularity of direct-sum decompositions in some classes of
  modules.
\newblock {\em J. Math. Sci. (N.Y.)}, 139(4):6814--6822, 2006.
\newblock \href {http://dx.doi.org/10.1007/s10958-006-0393-2}
  {\path{doi:10.1007/s10958-006-0393-2}}.

\bibitem[Fac19]{Facchini19}
A.~Facchini.
\newblock {\em Semilocal categories and modules with semilocal endomorphism
  rings}, volume 331 of {\em Progress in Mathematics}.
\newblock Birkh\"{a}user/Springer, Cham, 2019.
\newblock \href {http://dx.doi.org/10.1007/978-3-030-23284-9}
  {\path{doi:10.1007/978-3-030-23284-9}}.

\bibitem[FH00a]{FacchiniHerbera00}
A.~Facchini and D.~Herbera.
\newblock {$K_0$} of a semilocal ring.
\newblock {\em J. Algebra}, 225(1):47--69, 2000.
\newblock \href {http://dx.doi.org/10.1006/jabr.1999.8092}
  {\path{doi:10.1006/jabr.1999.8092}}.

\bibitem[FH00b]{FacchiniHerbera00a}
A.~Facchini and D.~Herbera.
\newblock Projective modules over semilocal rings.
\newblock In {\em Algebra and its applications ({A}thens, {OH}, 1999)}, volume
  259 of {\em Contemp. Math.}, pages 181--198. Amer. Math. Soc., Providence,
  RI, 2000.
\newblock \href {http://dx.doi.org/10.1090/conm/259/04094}
  {\path{doi:10.1090/conm/259/04094}}.

\bibitem[FW04]{FacchiniWiegand04}
A.~Facchini and R.~Wiegand.
\newblock Direct-sum decompositions of modules with semilocal endomorphism
  rings.
\newblock {\em J. Algebra}, 274(2):689--707, 2004.
\newblock \href {http://dx.doi.org/10.1016/j.jalgebra.2003.06.004}
  {\path{doi:10.1016/j.jalgebra.2003.06.004}}.

\bibitem[Gry22]{Grynkiewicz22}
D.~J. Grynkiewicz.
\newblock {\em The characterization of finite elasticities---factorization
  theory in {K}rull monoids via convex geometry}, volume 2316 of {\em Lecture
  Notes in Mathematics}.
\newblock Springer, Cham, 2022.
\newblock \href {http://dx.doi.org/10.1007/978-3-031-14869-9}
  {\path{doi:10.1007/978-3-031-14869-9}}.

\bibitem[GZ20]{GeroldingerZhong20}
A.~Geroldinger and Q.~Zhong.
\newblock Factorization theory in commutative monoids.
\newblock {\em Semigroup Forum}, 100(1):22--51, 2020.
\newblock \href {http://dx.doi.org/10.1007/s00233-019-10079-0}
  {\path{doi:10.1007/s00233-019-10079-0}}.

\bibitem[Her14]{Herbera14}
D.~Herbera.
\newblock Construction of modules with a prescribed direct sum decomposition.
\newblock {\em Int. Electron. J. Algebra}, 15:218--248, 2014.
\newblock \href {http://dx.doi.org/10.24330/ieja.266249}
  {\path{doi:10.24330/ieja.266249}}.

\bibitem[Hin63]{Hinohara63}
Y.~Hinohara.
\newblock Supplement to ``{P}rojective modules over weakly noetherian rings''.
\newblock {\em J. Math. Soc. Japan}, 15:474--475, 1963.
\newblock \href {http://dx.doi.org/10.2969/jmsj/01540474}
  {\path{doi:10.2969/jmsj/01540474}}.

\bibitem[HP10]{PavelHerbera}
D.~Herbera and P.~P\v{r}\'{\i}hoda.
\newblock Big projective modules over noetherian semilocal rings.
\newblock {\em J. Reine Angew. Math.}, 648:111--148, 2010.
\newblock \href {http://dx.doi.org/10.1515/CRELLE.2010.081}
  {\path{doi:10.1515/CRELLE.2010.081}}.

\bibitem[HP14a]{infinitepullback}
D.~Herbera and P.~P\v{r}\'{\i}hoda.
\newblock Infinitely generated projective modules over pullbacks of rings.
\newblock {\em Trans. Amer. Math. Soc.}, 366(3):1433--1454, 2014.
\newblock \href {http://dx.doi.org/10.1090/S0002-9947-2013-05798-4}
  {\path{doi:10.1090/S0002-9947-2013-05798-4}}.

\bibitem[HP14b]{HerberaPrihoda14}
D.~Herbera and P.~P\v{r}\'{\i}hoda.
\newblock Reconstructing projective modules from its trace ideal.
\newblock {\em J. Algebra}, 416:25--57, 2014.
\newblock \href {http://dx.doi.org/10.1016/j.jalgebra.2014.06.010}
  {\path{doi:10.1016/j.jalgebra.2014.06.010}}.

\bibitem[HPW23]{HerberaPrihodaWiegand23}
D.~Herbera, P.~Příhoda, and R.~Wiegand.
\newblock Big pure projective modules over commutative noetherian rings:
  comparison with the completion.
\newblock 2023.
\newblock Preprint.
\newblock \href {http://arxiv.org/abs/2311.05338} {\path{arXiv:2311.05338}}.

\bibitem[HW98]{HebischWeinert98}
U.~Hebisch and H.~J. Weinert.
\newblock {\em Semirings: algebraic theory and applications in computer
  science}, volume~5 of {\em Series in Algebra}.
\newblock World Scientific Publishing Co., Inc., River Edge, NJ, 1998.
\newblock Translated from the 1993 German original.
\newblock \href {http://dx.doi.org/10.1142/3903} {\path{doi:10.1142/3903}}.

\bibitem[Kap58]{Kaplansky58}
I.~Kaplansky.
\newblock Projective modules.
\newblock {\em Ann. of Math. (2)}, 68:372--377, 1958.
\newblock \href {http://dx.doi.org/10.2307/1970252}
  {\path{doi:10.2307/1970252}}.

\bibitem[Lea62]{Leavitt62}
W.~G. Leavitt.
\newblock The module type of a ring.
\newblock {\em Trans. Amer. Math. Soc.}, 103:113--130, 1962.
\newblock \href {http://dx.doi.org/10.2307/1993743}
  {\path{doi:10.2307/1993743}}.

\bibitem[LR11]{LevyRobson11}
L.~S. Levy and J.~C. Robson.
\newblock {\em Hereditary {N}oetherian prime rings and idealizers}, volume 174
  of {\em Mathematical Surveys and Monographs}.
\newblock American Mathematical Society, Providence, RI, 2011.
\newblock \href {http://dx.doi.org/10.1090/surv/174}
  {\path{doi:10.1090/surv/174}}.

\bibitem[LW12]{LeuschkeWiegand12}
G.~J. Leuschke and R.~Wiegand.
\newblock {\em Cohen-{M}acaulay representations}, volume 181 of {\em
  Mathematical Surveys and Monographs}.
\newblock Amer. Math. Soc., Providence, RI, 2012.
\newblock \href {http://dx.doi.org/10.1090/surv/181}
  {\path{doi:10.1090/surv/181}}.

\bibitem[MPR07]{McGovernPuninskiRothmaler07}
W.~W. McGovern, G.~Puninski, and P.~Rothmaler.
\newblock When every projective module is a direct sum of finitely generated
  modules.
\newblock {\em J. Algebra}, 315(1):454--481, 2007.
\newblock \href {http://dx.doi.org/10.1016/j.jalgebra.2007.01.043}
  {\path{doi:10.1016/j.jalgebra.2007.01.043}}.

\bibitem[Mue70]{Mueller70}
B.~J. Mueller.
\newblock On semi-perfect rings.
\newblock {\em Illinois J. Math.}, 14:464--467, 1970.
\newblock \href {http://dx.doi.org/10.1215/ijm/1256053082}
  {\path{doi:10.1215/ijm/1256053082}}.

\bibitem[Naz]{Nazemian2026}
Z.~Nazemian.
\newblock Realization of monoids with countable sum.
\newblock {\em Journal of Algebra and Its Applications}.
\newblock URL: \url{https://doi.org/10.1142/S0219498826430012}, \href
  {http://dx.doi.org/10.1142/S0219498826430012}
  {\path{doi:10.1142/S0219498826430012}}.

\bibitem[PP09]{PavelandPuninski}
P.~P\v{r}\'{\i}hoda and G.~Puninski.
\newblock Non-finitely generated projective modules over generalized {W}eyl
  algebras.
\newblock {\em J. Algebra}, 321(4):1326--1342, 2009.
\newblock \href {http://dx.doi.org/10.1016/j.jalgebra.2008.11.015}
  {\path{doi:10.1016/j.jalgebra.2008.11.015}}.

\bibitem[Roi90]{roitman90}
J.~Roitman.
\newblock {\em Introduction to modern set theory}.
\newblock Pure and Applied Mathematics (New York). John Wiley \& Sons, Inc.,
  New York, 1990.
\newblock A Wiley-Interscience Publication.

\bibitem[War72]{Warfield72}
R.~B. Warfield, Jr.
\newblock Exchange rings and decompositions of modules.
\newblock {\em Math. Ann.}, 199:31--36, 1972.
\newblock \href {http://dx.doi.org/10.1007/BF01419573}
  {\path{doi:10.1007/BF01419573}}.

\bibitem[Weh98]{Frid}
F.~Wehrung.
\newblock Non-measurability properties of interpolation vector spaces.
\newblock {\em Israel J. Math.}, 103:177--206, 1998.
\newblock \href {http://dx.doi.org/10.1007/BF02762273}
  {\path{doi:10.1007/BF02762273}}.

\bibitem[Whi80]{whitehead}
J.~M. Whitehead.
\newblock Projective modules and their trace ideals.
\newblock {\em Comm. Algebra}, 8(19):1873--1901, 1980.
\newblock \href {http://dx.doi.org/10.1080/00927878008822551}
  {\path{doi:10.1080/00927878008822551}}.

\bibitem[Wie01]{Wiegand01}
R.~Wiegand.
\newblock Direct-sum decompositions over local rings.
\newblock {\em J. Algebra}, 240(1):83--97, 2001.
\newblock \href {http://dx.doi.org/10.1006/jabr.2000.8657}
  {\path{doi:10.1006/jabr.2000.8657}}.

\bibitem[Wod15]{vocab}
M.~Wodzicki.
\newblock A mathematician's vocabulary, 2015.
\newblock Lecture Notes, 2015-03-15.

\end{thebibliography}
\end{document}